\documentclass[a4paper,reqno,10pt]{amsart}
\usepackage{soul} 
\usepackage[all,cmtip]{xy}
\usepackage{verbatim}
\usepackage{amstext}
\usepackage{amsmath,amssymb}
\usepackage{amscd}
\usepackage{amsthm}
\usepackage{amsfonts}
\usepackage{enumerate}
\usepackage{graphicx}
\usepackage{latexsym}
\usepackage{mathrsfs}
\usepackage{mathtools}
\newcommand{\newword}[1]{\emph{#1}}

\usepackage{color}

\newcommand{\margincolor}{red}      
\definecolor{darkgreen}{rgb}{0,0.7,0}

\newcounter{margincounter}
\newcommand{\marginnum}{
\ifnum\value{margincounter}<10
\textcolor{\margincolor}{\begin{picture}(0,0)\put(2.2,2.4){\circle{9}}\end{picture}\footnotesize\arabic{margincounter}}
\else\ifnum\value{margincounter}<100
\textcolor{\margincolor}{\begin{picture}(0,0)\put(4.256,2.5){\circle{11}}\end{picture}\footnotesize\arabic{margincounter}}
\else
\textcolor{\margincolor}{\begin{picture}(0,0)\put(6.8,2.5){\circle{14}}\end{picture}\footnotesize\arabic{margincounter}}
\fi\fi
}

\usepackage{pstricks}
\usepackage{lscape}
\usepackage{comment}

\newtheorem{theorem}{Theorem}[section]

\newtheorem{corollary}[theorem]{Corollary}
\newtheorem{lemma}[theorem]{Lemma}
\newtheorem{proposition}[theorem]{Proposition}
\newtheorem{prop}[theorem]{Proposition}
\newtheorem{cor}[theorem]{Corollary}
\newtheorem{definition-proposition}[theorem]{Definition-Proposition}

\theoremstyle{definition}
\newtheorem{definition}[theorem]{Definition}

\newtheorem{remark}[theorem]{Remark}
\newtheorem{example}[theorem]{Example}

\newcommand{\DDD}{\mathsf{D}}

\renewcommand{\SS}{\mathscr{S}}
\newcommand{\TT}{\mathscr{T}}
\newcommand{\FF}{\mathscr{F}}

\newcommand{\WW}{\mathscr{W}}

\newcommand{\bo}{\operatorname{b}\nolimits}

\newcommand{\soc}{\operatorname{soc}\nolimits}
\newcommand{\Ext}{\operatorname{Ext}\nolimits}

\newcommand{\Tor}{\operatorname{Tor}\nolimits}
\newcommand{\Hom}{\operatorname{Hom}\nolimits}
\newcommand{\End}{\operatorname{End}\nolimits}

\newcommand{\op}{\operatorname{op}\nolimits}
\newcommand{\RHom}{\mathbf{R}\strut\kern-.2em\operatorname{Hom}\nolimits}
\newcommand{\Lotimes}{\mathop{\stackrel{\mathbf{L}}{\otimes}}\nolimits}
\newcommand{\Image}{\operatorname{Im}\nolimits}

\DeclareMathOperator{\moduleCategory}{\mathsf{mod}}
\renewcommand{\mod}{\moduleCategory}

\DeclareMathOperator{\Sub}{\mathsf{Sub}}
\DeclareMathOperator{\Tors}{\mathsf{T}}
\DeclareMathOperator{\Torf}{\mathsf{F}}

\DeclareMathOperator{\add}{\mathsf{add}}

\DeclareMathOperator{\fd}{\mathsf{fd}}
\DeclareMathOperator{\rad}{\mathsf{rad}}
\DeclareMathOperator{\jirr}{\mathsf{j-Irr}}
\DeclareMathOperator{\mirr}{\mathsf{m-Irr}}

\newcommand{\cut}{\ar@{-}@[|(5)]}

\DeclareMathOperator{\sttilt}{\mathsf{s}\tau\text{-}\mathsf{tilt}}
\DeclareMathOperator{\tors}{\mathsf{tors}}
\DeclareMathOperator{\Fac}{\mathsf{Fac}}
\newcommand{\cyc}{\operatorname{cyc}\nolimits}

\numberwithin{equation}{section}

\DeclareMathOperator{\Con}{\mathsf{Con}}

\DeclareMathOperator{\ftors}{\mathsf{f-tors}}
\DeclareMathOperator{\torf}{\mathsf{torf}}
\DeclareMathOperator{\ftorf}{\mathsf{f-torf}}

\DeclareMathOperator{\lwide}{\mathsf{l-wide}}
\DeclareMathOperator{\brick}{\mathsf{brick}}
\DeclareMathOperator{\Filt}{\mathsf{Filt}}

\newcommand{\set}[1]{{\left\lbrace #1 \right\rbrace}}
\newcommand{\con}{\mathsf{con}}
\newcommand{\ConJI}{\mathsf{Con}_{\mathsf{JI}}}
\newcommand{\Hasse}{\mathsf{Hasse}}
\DeclareMathOperator{\Layers}{\mathsf{layer}}

\DeclareMathOperator{\IndtRig}{\mathsf{i}\tau\text{-}\mathsf{rigid}}
\DeclareMathOperator{\IndtmRig}{\mathsf{i}\tau^{--}\text{-}\mathsf{rigid}}
\newcommand{\FPoly}{\mathsf{FPoly}}
\newcommand{\SFPoly}{\mathsf{SFPoly}}
\newcommand{\join}{\vee}
\newcommand{\meet}{\wedge}
\renewcommand{\Join}{\bigvee}
\newcommand{\Meet}{\bigwedge}

\begin{document}
\title[Lattice structure via representation theory]{Lattice structure of Weyl groups via representation theory of preprojective algebras}

\begin{abstract}
This paper studies the combinatorics of lattice congruences of the weak order on a finite Weyl group $W$, using representation theory of the corresponding preprojective algebra $\Pi$.
Natural bijections are constructed between important objects including join-irreducible congruences, join-irreducible (respectively, meet-irreducible) elements of $W$, indecomposable $\tau$-rigid (respectively, $\tau^-$-rigid) modules and layers of $\Pi$.
The lattice-theoretically natural labelling of the Hasse quiver by join-irreducible elements of $W$ is shown to coincide with the algebraically natural labelling by layers of $\Pi$.
We show that layers of $\Pi$ are nothing but bricks (or equivalently stones, or 2-spherical modules).
The forcing order on join-irreducible elements of $W$ (arising from the study of lattice congruences) is described algebraically in terms of the doubleton extension order.
We give a combinatorial description of indecomposable $\tau^-$-rigid modules for type $A$ and $D$.
\end{abstract}

\author[Iyama]{Osamu Iyama}
\address{O. Iyama: Graduate School of Mathematics, Nagoya University, Chikusa-ku, Nagoya, 464-8602 Japan}
\email{iyama@math.nagoya-u.ac.jp}
\urladdr{http://www.math.nagoya-u.ac.jp/~iyama/}
\author[Reading]{Nathan Reading}
\address{N. Reading:  Department of Mathematics, North Carolina State University, Raleigh, NC 27695-8205, USA}
\email{reading@math.ncsu.edu}
\urladdr{http://www4.ncsu.edu/~nreadin/}
\author[Reiten]{Idun Reiten}
\address{I. Reiten: Department of Mathematical Sciences Norges teknisk-naturvitenskapelige universitet 7491 Trondheim Norway}
\email{idun.reiten@math.ntnu.no}
\urladdr{http://www.ntnu.edu/employees/idun.reiten}
\author[Thomas]{Hugh Thomas}
\address{H. Thomas: D\'epartement de math\'ematiques, Universit\'e du
Qu\'ebec \`a Montr\'eal, CP 8888, Succursale Centre-Ville, Montr\'eal, QC, H3C 3P8, Canada}
\email{hugh.ross.thomas@gmail.com}
\urladdr{http://www.lacim.uqam.ca/~hugh}
\thanks{2010 \emph{Mathematics Subject Classification}. Primary 16G10, Secondary 05E10, 06B10, 18E40, 20F55}
\keywords{preprojective algebra, Weyl group, weak order, $\tau$-tilting theory, brick, join-irreducible element, lattice congruence}
\thanks{Osamu Iyama's work on this project was partially supported by JSPS Grant-in-Aid for Scientific
Research (B) 24340004, (B) 16H03923, (C) 23540045 and (S) 15H05738.
Nathan Reading's work on this project was partially supported by the National Science Foundation under grant numbers DMS-1101568 and DMS-1500949.  
Idun Reiten was supported by the FriNat grants 196600 and 231000 from the 
Research Council of Norway.  
Hugh Thomas's work on this project was partially supported by an NSERC Discovery Grant and the Canada Research Chairs program.  
The authors also gratefully acknowledge the hospitality of 
MSRI, Oberwolfach, Bielefeld University,
and the Mittag-Leffler Institute.}
\maketitle
\tableofcontents

\section{Introduction}

Let $\Delta$ be a simply laced Dynkin diagram and $W$ the corresponding Weyl group.
Once we fix an orientation $Q$ of $\Delta$, then the representation theory of $Q$ categorifies the root system associated with $\Delta$ in the sense that we have Gabriel's bijection between positive roots and indecomposable representations of $Q$.
Throughout this paper, we denote by $\Pi$ the preprojective algebra of $\Delta$ over an arbitrary field $k$.
It unifies the representation theory of different quivers with the same underlying graph $\Delta$, and its various aspects have been studied, e.g. \cite{BKT,BGL,BIRS,CH,DR,GLS,Lu,KS,N}.
Mizuno \cite{Mi} showed that the support $\tau$-tilting theory of $\Pi$ categorifies the Weyl group $W$ with the weak order in the following sense:
There exists a bijection $W\ni w\mapsto I(w)$ from $W$ to the set $\sttilt\Pi$ of support $\tau$-tilting $\Pi$-modules with the property that $v\le w$ in the weak order on $W$ if and only if $I(v)\ge I(w)$ in the generation order on $\sttilt\Pi$.
The ideal $I(w)$ was introduced in \cite{IR,BIRS} and has been studied by several authors, e.g. \cite{AM,A,AIRT,BKT,GLS,K,L,Ma,ORT,SY1}.
In what follows, we will overload the symbol $W$ to denote not only the group $W$, but also the weak order on $W$.

The weak order on $W$ is a lattice \cite{BB}: a partial order such that meets (greatest lower bounds) and joins (least upper bounds) exist.
It is enlightening to take a more algebraic point of view of lattices, viewing a lattice as a set with two binary operations (meet and join). 
Seen in this light, the categorification of $W$ by support $\tau$-tilting theory is the categorification of an algebraic object (a lattice) in terms of another algebraic object (a finite-dimensional algebra).
In both of these algebraic settings, there is an important algebraic quotient operation.
Quotients of the weak order are governed by lattice congruences, while quotients of the preprojective algebra are governed by ideals.
A natural question is whether these two notions of quotient are related.
The answer is yes, and the relationship turns out to be very nice.

This paper and a companion paper~\cite{paperB} concern the relationship between the two notions of quotient.
In the other paper, we observe, for a more general algebra $A$ and an ideal $I$ of $A$, that $\sttilt(A/I)$ is a lattice quotient of $\sttilt A$ and we give necessary conditions for lattice congruences which arise in this way from quotients of $\Pi$.  
We study the combinatorics of such \newword{algebraic quotients} of the weak order in general, and in the special case where $\Pi/I$ is hereditary.
We also work out, in detail, the combinatorics of algebraic quotients in type A.

Whereas~\cite{paperB} starts with algebra quotients and determines what happens to the corresponding lattices, this paper starts from the other direction.
Here, we start with the rich combinatorics of (arbitrary, not necessarily algebraic) lattice congruences of $W$ and find that it appears naturally within the representation theory of $\Pi$.

The set of all lattice congruences of $L$ form a lattice $\Con L$, and the join-irreducible elements of $\Con L$
are called the join-irreducible congruences (see Section 2.1 for details).
The combinatorial approach to congruences of a finite lattice $L$ begins with the connection between arrows in the Hasse quiver of $L$, join-irreducible elements of $L$, and join-irreducible congruences on $L$.
We will overload the symbol $W$, using it to denote the Hasse quiver of the weak order on $W$.

Our first main theorem connects join-irreducible elements of $W$ and join-irreducible congruences on $W$ to layers of $\Pi$.
A $\Pi$-module is called a \emph{layer} if it is isomorphic to $I(w)/I(ws_i)$ for an arrow $ws_i\to w$ in the Hasse quiver of $W$, see \cite{AIRT}.  

\begin{theorem}\label{bijections}
There exist bijections between the following sets.
\begin{itemize}
\item The set $\jirr W$ of join-irreducible elements of $W$.
\item The set $\mirr W$ of meet-irreducible elements of $W$.
\item The set $\ConJI(W)$ of join-irreducible congruences of $W$.
\item The set $\IndtRig\Pi$ of indecomposable $\tau$-rigid $\Pi$-modules.
\item The set $\IndtmRig\Pi$ of indecomposable $\tau^-$-rigid $\Pi$-modules.
\item The set $\Layers\Pi$ of layers of $\Pi$.
\end{itemize}
\end{theorem}

We prove Theorem~\ref{bijections} and give explicit bijections as part of Theorem~\ref{big diagram}.
The fact that join-irreducible elements, meet-irreducible elements, and join-irreducible congruences of $W$ are all in bijection is known, and this property of a lattice is called \newword{congruence uniformity} (in the sense of Day~\cite{Day}).
This was proved in~\cite{bounded} (where in fact an equivalent property called \newword{boundedness} was established).

The main content of Theorem \ref{bijections} is the unexpectedly deep link between the representation theory of the preprojective algebra and the lattice theory of weak order on the corresponding Weyl group.  As part of establishing this link, we have also proved some new results within the representation theory of preprojective algebras which we believe to be of independent interest.
Our second main theorem gives two additional algebraic descriptions of layers of~$\Pi$.
We say that a $\Pi$-module $L$ is a \newword{brick} if $\End_\Pi(L)$ is a division algebra, and a \newword{stone} if $L$ is a brick satisfying $\Ext^1_\Pi(L,L)=0$ \cite{HHKU,KL}.
Let $\widehat{\Pi}$ be the preprojective algebra of the extended Dynkin type corresponding to $\Pi$.
Let $\DDD^{\bo}(\fd\widehat{\Pi})$ be the bounded derived category of finite-dimensional $\widehat{\Pi}$-modules.
An object $L\in\DDD^{\bo}(\fd\widehat{\Pi})$ is called \emph{2-spherical} if
$\dim_k\Hom_{\DDD^{\bo}(\fd\widehat{\Pi})}(L,L[i])$ is $1$ for $i=0,2$ and $0$ otherwise. (See~\cite{ST}.)

\begin{theorem}\label{main of layer}
The following classes of $\Pi$-modules are the same.
\begin{itemize}
\item Layers of $\Pi$.
\item Bricks of $\Pi$.
\item Stones of $\Pi$.
\item $\Pi$-modules which are 2-spherical as $\widehat{\Pi}$-modules,
\end{itemize}
Moreover any module $L$ in these classes satisfies $\End_\Pi(L)=k$.
\end{theorem}
It was shown by Bolten \cite{Bol} that layers and bricks are the same. Also a related observation was given by Sekiya and Yamaura \cite{SY2}.

Our third main result concerns the interplay between arrows in the Hasse quiver of $W$, join-irreducible elements of $W$, and join-irreducible congruences on $W$.
We refer to Section 2.1 for details about the following notions.
Given any arrow $x\to y$ in the Hasse quiver of an arbitrary finite lattice $L$, define $\con(x,y)$ to be the smallest congruence on $L$ such that $x\equiv y$.
This is a join-irreducible congruence. 
If $j$ is a join-irreducible element of $L$, we write $j_*$ for the unique element covered by $j$ in $L$.
The congruence $\con(j,j_*)$ is thus join-irreducible, and it turns out that every join-irreducible element is $\con(j,j_*)$ for some $j$.
When $L$ is the weak order on $W$, the map $j\mapsto\con(j,j_*)$ is the bijection from join-irreducible elements of $W$ to join-irreducible congruences from Theorem~\ref{bijections}.
Since each Hasse arrow of $W$ specifies a join-irreducible congruence, and since join-irreducible congruences are in bijection with join-irreducible elements, we obtain a labelling
of the Hasse arrows of $W$ by join-irreducible elements. 
We call this the \newword{join-irreducible labelling} of $W$.
Besides this labelling coming from lattice theory, there is a labelling of the Hasse quiver coming from representation theory, namely the \newword{layer labelling}.
This labels a Hasse arrow $ws_i\to w$ by the layer $I(w)/I(ws_i)$.
The layer labellings for type $A_2$ and $A_3$ are given in Figures \ref{layer labelling A2} and \ref{layer labelling A3}.
\begin{figure}
\[\begin{xy}
(0,0) *+{0}="A",
(18,-9) *+{\left[\begin{smallmatrix}
1
\end{smallmatrix}\right]}="B",
(18,-24) *+{\left[\begin{smallmatrix}
1\\ &2
\end{smallmatrix}\middle|
\begin{smallmatrix}
1
\end{smallmatrix}\right]}="C",
(-18,-9) *+{\left[\begin{smallmatrix}
2
\end{smallmatrix}\right]}="D",
(-18,-24) *+{\left[\begin{smallmatrix}
2
\end{smallmatrix}\middle|
\begin{smallmatrix}
&2\\ 1
\end{smallmatrix}\right]}="E",
(0,-33) *+{\left[\begin{smallmatrix}
1\\ &2
\end{smallmatrix}\middle|
\begin{smallmatrix}
&2\\ 1
\end{smallmatrix}\right]}="F",
\ar@{<-}"A";"B",|{1}
\ar@{<-}"B";"C",|{\begin{smallmatrix}
1\\ &2
\end{smallmatrix}}
\ar@{<-}"C";"F",|{2}
\ar@{<-}"A";"D",|{2}
\ar@{<-}"D";"E",|{\begin{smallmatrix}
&2\\ 1
\end{smallmatrix}}
\ar@{<-}"E";"F",|{1}
\end{xy}\]
\caption{Layer labelling for $A_2$}
\label{layer labelling A2}
\end{figure}
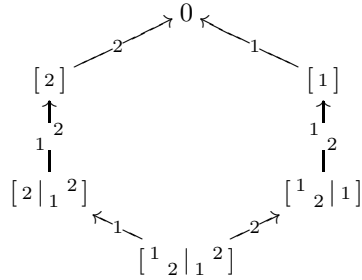
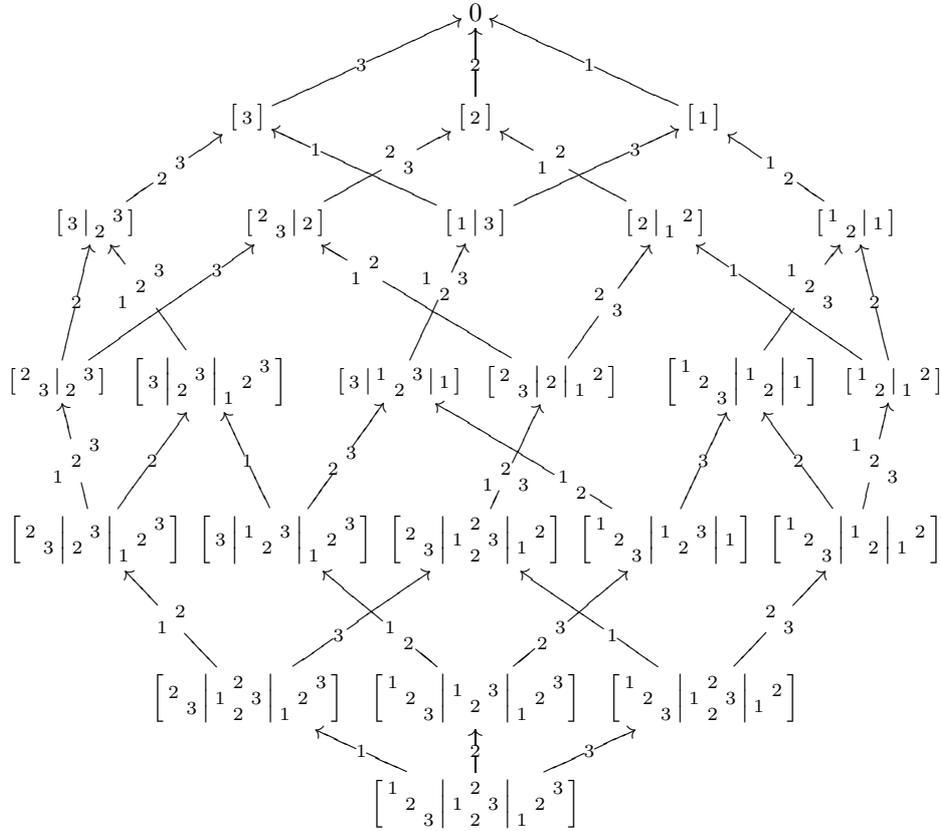
\begin{figure}
\[\begin{xy}
(0,-84) *+{
\left[\begin{smallmatrix}
1\\ &2\\ &&3\end{smallmatrix}\middle|
\begin{smallmatrix}
&2\\ 1&&3\\ &2\end{smallmatrix}\middle|
\begin{smallmatrix}
&&3\\ &2\\ 1\end{smallmatrix}\right]}="1",
(-30,-70) *+{\left[\begin{smallmatrix}
\\ 2\\ &3\end{smallmatrix}\middle|
\begin{smallmatrix}
&2\\ 1&&3\\ &2\end{smallmatrix}\middle|
\begin{smallmatrix}
&&3\\ &2\\ 1\end{smallmatrix}\right]}="2",
(0,-70) *+{\left[\begin{smallmatrix}
1\\ &2\\ &&3\end{smallmatrix}\middle|
\begin{smallmatrix}
1&&3\\ &2\end{smallmatrix}\middle|
\begin{smallmatrix}
&&3\\ &2\\ 1\end{smallmatrix}\right]}="3",
(30,-70) *+{\left[\begin{smallmatrix}
1\\ &2\\ &&3\end{smallmatrix}\middle|
\begin{smallmatrix}
&2\\ 1&&3\\ &2\end{smallmatrix}\middle|
\begin{smallmatrix}
&2\\ 1\end{smallmatrix}\right]}="4",
(-50,-49) *+{\left[\begin{smallmatrix}
2\\ &3\end{smallmatrix}\middle|
\begin{smallmatrix}
&3\\ 2\end{smallmatrix}\middle|
\begin{smallmatrix}
&&3\\ &2\\ 1\end{smallmatrix}\right]}="5",
(-25,-49) *+{\left[\begin{smallmatrix}
3\end{smallmatrix}\middle|
\begin{smallmatrix}
1&&3\\ &2\end{smallmatrix}\middle|
\begin{smallmatrix}
&&3\\ &2\\ 1\end{smallmatrix}\right]}="6",
(0,-49) *+{\left[\begin{smallmatrix}
\\ 2\\ &3\end{smallmatrix}\middle|
\begin{smallmatrix}
&2\\ 1&&3\\ &2\end{smallmatrix}\middle|
\begin{smallmatrix}
&2\\ 1\end{smallmatrix}\right]}="7",
(25,-49) *+{\left[\begin{smallmatrix}
1\\ &2\\ &&3\end{smallmatrix}\middle|
\begin{smallmatrix}
1&&3\\ &2\end{smallmatrix}\middle|
\begin{smallmatrix}
1\end{smallmatrix}\right]}="8",
(50,-49) *+{\left[\begin{smallmatrix}
1\\ &2\\ &&3\end{smallmatrix}\middle|
\begin{smallmatrix}
1\\ &2\end{smallmatrix}\middle|
\begin{smallmatrix}
&2\\ 1\end{smallmatrix}\right]}="9",
(-55,-28) *+{\left[\begin{smallmatrix}
2\\ &3\end{smallmatrix}\middle|
\begin{smallmatrix}
&3\\ 2\end{smallmatrix}\right]}="10",
(-35,-28) *+{\left[\begin{smallmatrix}
3\end{smallmatrix}\middle|
\begin{smallmatrix}
&3\\ 2\end{smallmatrix}\middle|
\begin{smallmatrix}
&&3\\ &2\\ 1\end{smallmatrix}\right]}="11",
(-10,-28) *+{\left[\begin{smallmatrix}
3\end{smallmatrix}\middle|
\begin{smallmatrix}
1&&3\\ &2\end{smallmatrix}\middle|
\begin{smallmatrix}
1\end{smallmatrix}\right]}="12",
(10,-28) *+{\left[\begin{smallmatrix}
\\ 2\\ &3\end{smallmatrix}\middle|
\begin{smallmatrix}
2\end{smallmatrix}\middle|
\begin{smallmatrix}
&2\\ 1\end{smallmatrix}\right]}="13",
(35,-28) *+{\left[\begin{smallmatrix}
1\\ &2\\ &&3\end{smallmatrix}\middle|
\begin{smallmatrix}
1\\ &2\end{smallmatrix}\middle|
\begin{smallmatrix}
1\end{smallmatrix}\right]}="14",
(55,-28) *+{\left[\begin{smallmatrix}
1\\ &2\end{smallmatrix}\middle|
\begin{smallmatrix}
&2\\ 1\end{smallmatrix}\right]}="15",
(-50,-7) *+{\left[\begin{smallmatrix}
3\end{smallmatrix}\middle|
\begin{smallmatrix}
&3\\ 2\end{smallmatrix}\right]}="16",
(-25,-7) *+{\left[\begin{smallmatrix}
2\\ &3\end{smallmatrix}\middle|
\begin{smallmatrix}
2\end{smallmatrix}\right]}="17",
(-0,-7) *+{\left[\begin{smallmatrix}
1\end{smallmatrix}\middle|
\begin{smallmatrix}
3\end{smallmatrix}\right]}="18",
(25,-7) *+{\left[\begin{smallmatrix}
2\end{smallmatrix}\middle|
\begin{smallmatrix}
&2\\ 1\end{smallmatrix}\right]}="19",
(50,-7) *+{\left[\begin{smallmatrix}
1\\ &2\end{smallmatrix}\middle|
\begin{smallmatrix}
1\end{smallmatrix}\right]}="20",
(-30,7) *+{\left[\begin{smallmatrix}
3\end{smallmatrix}\right]}="21",
(0,7) *+{\left[\begin{smallmatrix}
2\end{smallmatrix}\right]}="22",
(30,7) *+{\left[\begin{smallmatrix}
1\end{smallmatrix}\right]}="23",
(0,21) *+{0}="24",
\ar"1";"2",|{1}
\ar"1";"3",|{2}
\ar"1";"4",|{3}
\ar"2";"5",|{\begin{smallmatrix}
&2\\ 1\end{smallmatrix}}
\ar"2";"7",|(.4){3}
\ar"3";"6",|(.4){\begin{smallmatrix}
1\\ &2\end{smallmatrix}}
\ar"3";"8",|(.4){\begin{smallmatrix}
&3\\ 2\end{smallmatrix}}
\ar"4";"7",|(.4){1}
\ar"4";"9",|{\begin{smallmatrix}
2\\ &3\end{smallmatrix}}
\ar"5";"10",|{\begin{smallmatrix}
&&3\\ &2\\ 1\end{smallmatrix}}
\ar"5";"11",|{2}
\ar"6";"11",|{1}
\ar"6";"12",|{\begin{smallmatrix}
&3\\ 2\end{smallmatrix}}
\ar"7";"13",|(.4){\begin{smallmatrix}
&2\\ 1&&3\end{smallmatrix}}
\ar"8";"12",|(.35){\begin{smallmatrix}
1\\ &2\end{smallmatrix}}
\ar"8";"14",|{3}
\ar"9";"14",|{2}
\ar"9";"15",|{\begin{smallmatrix}
1\\ &2\\ &&3\end{smallmatrix}}
\ar"10";"16",|{2}
\ar"10";"17",|(.7){3}
\ar"11";"16",|(.6){\begin{smallmatrix}
&&3\\ &2\\ 1\end{smallmatrix}}
\ar"12";"18",|(.6){\begin{smallmatrix}
1&&3\\ &2\end{smallmatrix}}
\ar"13";"17",|(.7){\begin{smallmatrix}
&2\\ 1\end{smallmatrix}}
\ar"13";"19",|{\begin{smallmatrix}
2\\ &3\end{smallmatrix}}
\ar"14";"20",|(.6){\begin{smallmatrix}
1\\ &2\\ &&3\end{smallmatrix}}
\ar"15";"19",|(.7){1}
\ar"15";"20",|{2}
\ar"16";"21",|{\begin{smallmatrix}
&3\\ 2\end{smallmatrix}}
\ar"17";"22",|(.6){\begin{smallmatrix}
2\\ &3\end{smallmatrix}}
\ar"18";"21",|(.7){1}
\ar"18";"23",|(.7){3}
\ar"19";"22",|(.6){\begin{smallmatrix}
&2\\ 1\end{smallmatrix}}
\ar"20";"23",|{\begin{smallmatrix}
1\\ &2\end{smallmatrix}}
\ar"21";"24",|{3}
\ar"22";"24",|{2}
\ar"23";"24",|{1}
\end{xy}\]
\caption{Layer labelling for $A_3$}
\label{layer labelling A3}
\end{figure}

\begin{theorem}\label{labellings}
The map $j\mapsto I(j_*)/I(j)$ takes the join-irreducible labelling of $W$ to the layer labelling on $W$.
That is, given a Hasse arrow $ws_i\to w$ labelled by the join-irreducible element $j$, which covers the element $j_*$, the layer label on $ws_i\to w$ is $I(j_*)/I(j)$. 
\end{theorem}

This is also proved as part of Theorem~\ref{big diagram}, which gives a commutative diagram shown in Figure~\ref{lat alg diagram fig} between important objects. We include this diagram here, although some elements of it have not yet been explained, as a road map to the major results of the paper. The maps are bijections or surjections as marked with tildes ``$\sim$'' or double-headed arrows.
\begin{figure}
\[\scalebox{1.25}{\begin{xy}
(35,55)*+{\jirr(W)}="J",
(37,52)="J+",
(40,52)="J++",
(-35,55)*+{\mirr(W)}="M",
(-37,52)="M-",
(-40,52)="M--",
(0,55)*+{\ConJI(W)}="C",
(0,15)*+{\Hasse_1(W)}="H",
(0,-33)*+{\Layers\Pi}="L",
(0,-40)*+{\brick\Pi}="S",
(3,-30)="L+",
(-3,-30)="L-",
(-35,-40)*+{{\IndtRig\Pi}}="It",
(35,-40)*+{\IndtmRig\Pi}="ltm",
(-37,-75)*+{\jirr(\tors\Pi)}="jtP",
(-40,-72)="jtP-",
(35,-75)*+{\jirr(\torf\Pi)}="Q",
(0,-75)*+{\lwide\Pi}="W",
(-37,-37)="It-",
(-40,-37)="It--",
(-40,-43)="It--low",
(40,-37)="Itm++",
(40,-43)="Itm++low",
(40,-72)="Q+",
(40,-72)="Q++",
\ar@{->>}|(.65){(x\to y)\mapsto\con(x,y)}"H";"C",
\ar@{->}_{j\mapsto\con(j,j_*)\,\,\,}^{\sim}"J";"C",
\ar@{->}^{\,\,\,m\mapsto\con(m^*,m)}_{\sim}"M";"C",
\ar@{->>}|(.4){\parbox{38pt}{\scriptsize\hspace*{-15pt}\mbox{meet-irreducible}\\\mbox{labelling}}}"H";"M",
\ar@{->>}|(.4){\parbox{38pt}{\scriptsize\hspace*{0pt}\mbox{join-irreducible}\\\hspace*{15pt}\mbox{labelling}}}"H";"J",
\ar@{->>}|(.3){\parbox{38pt}{\scriptsize\hspace*{-8pt}\mbox{layer labelling}}}"H";"L",
\ar@{->}^{\sim}_{\parbox{38pt}{\scriptsize\centering $X\mapsto$\\ $X/\rad X_{\End_{\Pi}(X)}$}}"It";"S",
\ar@{->}_{\sim}^{\parbox{38pt}{\scriptsize\centering $X\mapsto$\\ $\soc X_{\End_{\Pi}(X)}$}}"ltm";"S",
\ar@{->}_(.88){\sim}^(.88){\parbox{100 pt}{\scriptsize$\hspace*{5pt}m\mapsto M(m)\\\hspace*{5pt}:=I(m)e_i$\\\hspace*{-0pt}for $m^*=ms_i$}}"M--";"It--".
\ar@{->}^(.88){\sim}_(.88){\parbox{50pt}{\scriptsize$j\mapsto J(j)\\\hspace*{-5pt}:=(\Pi/I(j))e_i$\\\hspace*{0pt}for $j_*=js_i$}}"J++";"Itm++".
\ar@/_27pt/@{->}|(.75){m\mapsto I(m)/I(m^*)}^{\sim}"M-";"L-",
\ar@/^27pt/@{->}|(.75){j\mapsto I(j_*)/I(j)}^{\sim}"J+";"L+",
\ar@{->}^{M\mapsto\Fac M}_{\sim}"It--low";"jtP-",
\ar@{->}_{M\mapsto\Sub M}^{\sim}"Itm++low";"Q+",
\ar@{->}^{L\mapsto\Filt L}_{\sim}"S";"W",
\ar@{->}^{\WW\mapsto\Tors(\WW)}_{\sim}"W";"jtP",
\ar@{->}_{\WW\mapsto\Torf(\WW)}^{\sim}"W";"Q",
\ar@{=}"L";"S",
\end{xy}}\]
\caption{The correspondences established in Theorem~\ref{big diagram}}
\label{lat alg diagram fig}
\end{figure}
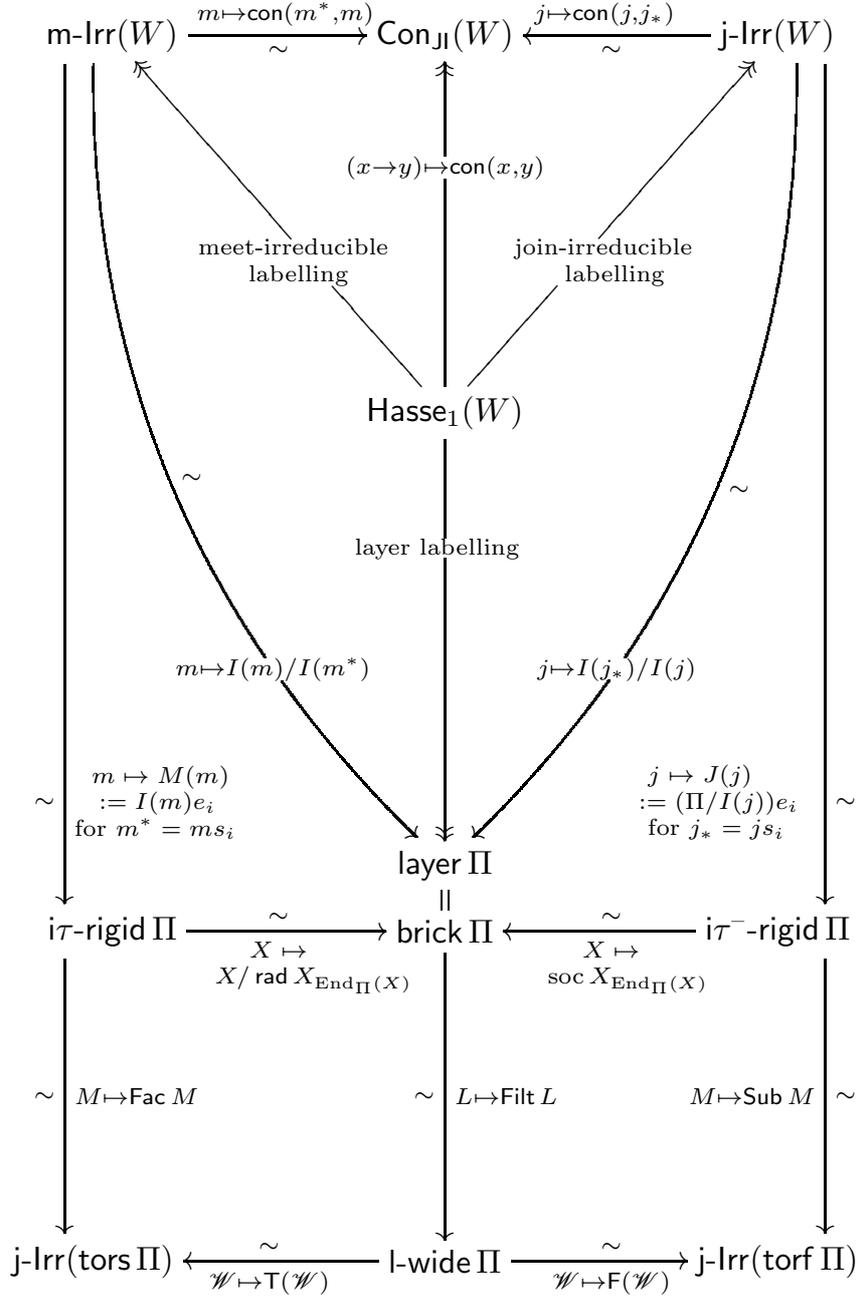

Given two Hasse arrows $x\to y$ and $x'\to y'$ of a lattice $L$, we say that $x\to y$ \newword{forces} $x'\to y'$ if $\con(x,y)\ge\con(x',y')$ in $\Con(L)$.
In other words, $x\to y$ forces $x'\to y'$ if every congruence setting $x\equiv y$ also sets $x'\equiv y'$.
In the weak order on $W$, the forcing order on Hasse arrows restricts to a partial order on Hasse arrows of the form $j\to j_*$ such that $j$ is join-irreducible.
We think of this as a partial order on join-irreducible elements and call it the \newword{forcing order} on join-irreducible elements of $W$.

We say that a pair of layer modules $X,Y$ form a \newword{doubleton} if $\Ext_\Pi^1(Y,X)$ 
and $\Ext_\Pi^1(X,Y)$ are one-dimensional, and the corresponding extensions
are again layer modules.  
We define the \newword{doubleton extension order} on layer modules to be the transitive closure of the relation with $A>B$ if there exists a doubleton $A,C$ such that $B$ is the extension of $A$ by $C$ or of $C$ by $A$.
 Our fourth main result is the following.

\begin{theorem} \label{isom} 
The map $j\mapsto I(j_*)/I(j)$ is an isomorphism from the forcing order on join-irreducible elements of $W$ to the doubleton extension order on layer modules of $\Pi$.
\end{theorem}

In addition to these general results, we show that in type $A_n$, the doubleton extension order coincides with the reverse of the subfactor order (see Theorem \ref{typeA}).
  In Section 6, we give an explicit combinatorial description of the indecomposable $\tau$-rigid
  modules in types $A_n$ and $D_n$. For these cases, elements of the Weyl group $W$ are given by
  (signed) permutations, and join-irreducible elements in $W$ have simple characterizations.
  We prove in Theorems \ref{define X(w) for A}, \ref{define X(w) for D2} and \ref{define X(w) for D1}
 that the indecomposable $\tau$-rigid $\Pi$-modules
  have direct combinatorial descriptions in terms of the corresponding (signed) permutations.
  These concrete examples complement the general theory worked out in the previous sections.

\section{Preliminaries}\label{prelim}
In this section, we review the necessary background on lattices, the weak order, and finite-dimensional algebras.

\subsection{Lattice-theoretic preliminaries}\label{l-prelim} 
Proofs and additional details for the material reviewed here can be found in \cite[Sections~9-5 and~9-6]{regions}.  

For any poset $P$, we say that $x$ \newword{covers} $y$, and we write
$x\gtrdot y$, if $x>y$, and there is no $z\in P$ such that $x>z>y$.  
We represent $P$ by its Hasse quiver $\Hasse(P)=(P,\Hasse_1(P))$, whose vertex set is $P$, and whose arrow set $\Hasse_1(P)$ consists of all arrows $v\rightarrow w$ where $v$ covers $w$.

Given a subset $S$ of $P$, if there is a unique smallest element which is greater than or equal to all elements in $S$,
then this least upper bound is called the \newword{join} of $S$ and denoted $\Join S$.
Similarly, if there is a unique largest element in $P$ that is less than or equal to all elements in $S$, then this element is
called the \newword{meet} of $S$ and denoted $\Meet S$.
A \newword{lattice} $L$ is a poset in which every pair $a,b$ of elements in $L$ has both a meet $a\meet b$ and a join
$a\join b$, and a \newword{complete lattice} $L$ is a poset in which every subset $S$ of $L$ has both a meet and a join.
(Every finite subset of a lattice $L$ has both a meet and a join, but an infinite lattice fails to be complete if it has some infinite subset without a meet or without a join.)

We restrict our attention to finite lattices in this paper.
Some of the assertions made here for finite lattices hold for infinite lattice as well, but some do not.  

An element $j$ of a finite lattice $L$ is called \newword{join-irreducible}, whenever $j=a\join b$ for some $a,b\in L$, either $a=j$ or $b=j$ or both, and $j$ is not the minimum element of $L$.
Equivalently, $j$ is join-irreducible if and only if it covers exactly one element of $L$.
We write $j_*$ for the unique element covered by a join-irreducible element $j$.
Dually, a \newword{meet-irreducible} element of $L$ is an element $m$ that is covered by a unique element $m^*$.
The set of join-irreducible (respectively, meet-irreducible) elements of $L$ is denoted $\jirr L$ (respectively, $\mirr L$).

A \newword{(lattice) congruence} on a lattice $L$ is an equivalence relation $\Theta$ having the property that the $\Theta$-class of $a\join b$ depends only on the $\Theta$-class of $a$ and the $\Theta$-class of $b$, and having the same property for meets.
Given a congruence $\Theta$ on $L$, the set $L/\Theta$ of $\Theta$-classes has a well-defined meet and join operation, making $L/\Theta$ a lattice called the \newword{quotient} of $L$ modulo $\Theta$.

The set of all equivalence relations on a given set $L$ forms a lattice, where the meet of two relations is given
by the intersection of relations and the join of two relations is given by the transitive closure of union of relations.
When $L$ is a lattice, the set $\Con(L)$ consisting of congruences of $L$ is a sublattice of the lattice of equivalence relations.
Furthermore, $\Con(L)$ is a distributive lattice.
We denote by $\ConJI(L)$ the set of all join-irreducible congruences.
As mentioned in the introduction, we have a surjective map $\Hasse_1(L)\to\ConJI(L)$
sending an arrow $x\to y$ to $\con(x,y)$.
Here $\con(x,y)$ is the meet, in $\Con(L)$, of all congruences with $x\equiv y$.
A congruence $\Theta$ on a finite lattice $L$ is determined completely by the set of cover relations $x\gtrdot y$ in $L$ such that $x\equiv y$ modulo $\Theta$.
It is also determined uniquely by the set of join-irreducible elements $j$ in $L$ such that $j\equiv j_*$ modulo $\Theta$, and thus we have an injective map
$\Con(L)\to 2^{\jirr(L)}$.

The map from cover relations $x\gtrdot y$ to join-irreducible congruences is typically not one-to-one.
The restriction of the map to cover relations of the form $j\gtrdot j_*$ is also surjective onto join-irreducible congruences, but may still fail to be one-to-one.
A lattice is called \newword{congruence uniform} if the map $j\mapsto\con(j,j_*)$ is injective (and thus a bijection) from join-irreducible elements to join-irreducible congruences and the map $m\to\con(m^*,m)$ is injective (and thus a bijection) from meet-irreducible elements to join-irreducible congruences.
A finite congruence uniform lattice is always \newword{semidistributive}.
This means that if $x\join y=x\join z$ then $x\join(y\meet z)=x\join y$ and if $x\meet y=x\meet z$ then $x\meet(y\join z)=x\meet y$.

Since $\Con(L)$ is a finite distributive lattice, the Fundamental Theorem of Finite Distributive Lattices says that its elements are naturally identified with order ideals in the subposet $\ConJI(L)$ of $\Con(L)$.
When $L$ is congruence uniform, the subposet $\ConJI(L)$ induces a partial order on the join-irreducible elements of $L$, which we call the \newword{forcing order}.
A congruence $\Theta\in\Con(L)$ corresponds to the order ideal in $\ConJI(L)$ consisting of those join-irreducible congruences below $\Theta$ in $\Con(L)$ (i.e.\ finer than $\Theta$ as equivalence relations).
These are the join-irreducible congruences $\con(j,j_*)$ such that $j\equiv j_*$ modulo $\Theta$.
The forcing order on join-irreducible elements sets $j\le j'$ if and only if $j\equiv j_*$ modulo $\con(j',j'_*)$.

As mentioned above, each cover relation $x\gtrdot y$ in a finite lattice defines a join-irreducible congruence of $L$.
In a finite congruence uniform lattice $L$, each join-irreducible congruence is $\con(j,j_*)$ for a unique join-irreducible element $j$ of $L$.
The map $\Hasse_1(L)\to\jirr(L)$ sending the arrow $x\to y$ to the unique $j$ with $\con(j,j_*)=\con(x,y)$ is called the \newword{join-irreducible labelling} of $L$.
Each join-irreducible congruence is also $\con(m^*,m)$ for a unique meet-irreducible element $m$, and the map $\Hasse_1(L)\to\mirr(L)$ sending $x\to y$ to the unique $m$ with $\con(m^*,m)=\con(x,y)$ is called the \newword{meet-irreducible labelling} of $L$.
These labellings are described explicitly as follows.

The following proposition is \cite[Proposition~9-5.20]{regions}. 
Since that proposition's proof is left to an exercise, we give a proof here.

\begin{proposition}\label{expl label}
Let $L$ be a finite congruence uniform lattice and let $x\to y$ be an arrow in $\Hasse(L)$.
\begin{enumerate}[\rm(a)]
\item The join-irreducible label on $x\to y$ is $j=\Meet\set{z\in L:z\le x,\,z\not\le y}$.
Furthermore, $j\le x$ but $j\not\le y$.
\item The meet-irreducible label on $x\to y$ is $m=\Join\set{z\in L:z\ge y,\,z\not\ge x}$.
Furthermore, $m\ge y$ but $m\not\ge x$.
\end{enumerate}
In particular, if $j$ is a join-irreducible element and $m$ is a meet-irreducible element with $\con(j,j_*)=\con(m^*,m)$, then $j=\Meet\set{z\in L:z\le m^*,\,z\not\le m}$ and $m=\Join\set{z\in L:z\ge j_*,\,z\not\ge j}$.
\end{proposition}
\begin{proof}
The last statements are special cases of assertions (a) and (b).
Assertions (a) and (b) are dual to each other, so by symmetry it is enough to prove (a).
To do so, it is enough to show that $j$ is join-irreducible and that $\con(j,j_*)=\con(x,y)$.

Recall that a congruence uniform finite lattice is also semidistributive.
Every element $z$ of $\set{z\in L:z\le x,\,z\not\le y}$ has $z\join y=x$, so applying semidistributivity several times, we see that $j\join y=x$, so in particular $j\not\le y$.
It is immediate that $j\le x$.
If $j$ covers elements $a$ and $b$, then $a\le y$ and $b\le y$.
But if $a\neq b$, then $j$ is a minimal upper bound for $a$ and $b$, so it must equal $a\join b$.
Since $y$ is another upper bound for $a$ and $b$, we reach the contradiction $j\le y$.
We conclude that $j$ covers at most one element.
If $j$ covers no element, then $j$ is the minimal element of $L$, contradicting again the fact that $j\not\le y$.
We see that $j$ is join-irreducible.

We saw that $j\join y=x$ and we also verify easily that $j\meet y=j_*$.
If $\Theta$ is a congruence with $j\equiv j_*$, then $j\join y\equiv j_*\join y$, or in other words $x\equiv y$.
Conversely, if $\Theta$ has $x\equiv y$, then $j\meet x\equiv j\meet y$, or in other words $j\equiv j_*$.
We see that $\Theta$ has $j\equiv j_*$ if and only if $x\equiv y$, so that $\con(j,j_*)=\con(x,y)$.
\end{proof}

We summarize some of what we know about join-irreducible elements, meet-irreducible elements, and congruences in a finite congruence uniform lattice in Figure~\ref{cong unif diagram fig}. 
If $L$ is a finite congruence uniform lattice with Hasse quiver $\Hasse(L)$, join-irreducible elements $\jirr(L)$, meet-irreducible elements $\mirr(L)$, and join-irreducible congruences $\ConJI(L)$, then the diagram in Figure~\ref{cong unif diagram fig} commutes and the maps are bijections or surjections as marked with tildes ``$\sim$'' or double-headed arrows.
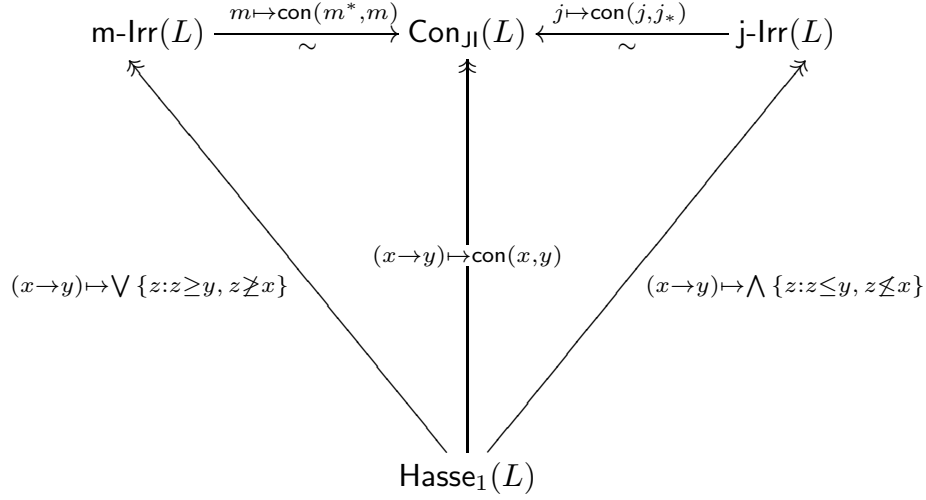
\begin{figure}
\[\scalebox{1.2}{\begin{xy}
(35,0)*+{\jirr(L)}="J",
(37,-3)="J+",
(-35,0)*+{\mirr(L)}="M",
(-37,-3)="M+",
(0,0)*+{\ConJI(L)}="C",
(0,-49)*+{\Hasse_1(L)}="H",
\ar@{->>}|{(x\to y)\mapsto\con(x,y)}"H";"C",
\ar@{->}_{j\mapsto\con(j,j_*)\,\,\,}^{\sim}"J";"C",
\ar@{->}^{\,\,\,m\mapsto\con(m^*,m)}_{\sim}"M";"C",
\ar@{->>}^{(x\to y)\mapsto\Join\set{z:z\ge y,\,z\not\ge x}}"H";"M+",
\ar@{->>}_{(x\to y)\mapsto\Meet\set{z:z\le x,\,z\not\le y}}"H";"J+",
\end{xy}}\]
\caption{Join-irreducible elements, meet-irreducible elements, and congruences in a finite congruence uniform lattice}
\label{cong unif diagram fig}
\end{figure}

Let $x$ be an element of a finite lattice $L$.
The expression $x=\Join S$ is the \newword{canonical join representation} of $x$ if no proper subset of $S$ joins to $x$ and if every join-representation $x=\Join T$ has the property that for all $s\in S$, there exists $t\in T$ with $s\le t$.
An element $x$ may fail to have a canonical join representation, but the canonical join representation of $x$ is unique if it exists.
The \newword{canonical meet representation} is defined dually.
The semidistributive property of a finite lattice $L$, described above, is equivalent to the property that every element of $L$ has a canonical join representation and a canonical meet representation.
A finite congruence uniform lattice is in particular semidistributive, and canonical join and meet representations can be described in terms of the join-irreducible labelling and meet-irreducible labelling as follows.

The following proposition is \cite[Proposition~9-5.30]{regions}. 
Since that proposition's proof is also left to an exercise, we give a proof here. 

\begin{prop}\label{cong unif canon}
If $L$ is a finite congruence uniform lattice, then the canonical join representation of $x\in L$ is $x=\Join J$, where $J$ is the set of join-irreducible labels on arrows starting at $x$ in $\Hasse(L)$.
The canonical meet representation is $x=\Meet M$, where $M$ is the set of meet-irreducible labels on arrows ending at $x$ in $\Hasse(L)$.
\end{prop}

\begin{proof}
We prove the statement for canonical join representations, using Proposition~\ref{expl label} throughout.
The other statement is dual.

First, we check that $x=\Join J$.
On the one hand, each $j\in J$ is below $x$ by Proposition~\ref{expl label}, so $\Join J\le x$.
If $\Join J<x$, then there exists $y$ with $x\gtrdot y\ge\Join J$.
But the label of $x\to y$ is in $J$ and is not below $y$, which is
a contradiction.

Next, we show that no proper subset of $J$ joins to $x$.  
Let $x\to y$ be an arrow with join-irreducible label $j\in J$. For any other arrow $x\to y'$ with join-irreducible label $j'\in J$, we have $y'\not\le y$ and $y'\le x$, from which it follows that $j\le y'$ by Proposition~\ref{expl label}. By symmetry, $j'\le y$ holds, and hence $\Join(J\setminus\{j\})\le y<x$ as desired.

Finally, we show that, if $x=\Join T$ for some $T\subseteq L$, then for all $j\in J$, there exists $t\in T$ with $j\le t$.
Let $y$ have $x\to y$ and $j=\Meet\set{z\in L:z\le x,\,z\not\le y}$.
Every element of $T$ is $\le x$, and if every element of $T$ is $\le y$, then $\Join T\le y$, contradicting the assumption that $x=\Join T$.
Thus there exists some element $t$ of $T$ with $t\le x$ and $t\not\le y$, and this element is above $j$ by definition.
\end{proof}

A \emph{polygon} in a finite lattice $L$ is an interval $[x,y]$ such that $\set{z\in L:x<z<y}$ consists of two disjoint non-empty chains.
(Thus the Hasse quiver of $[x,y]$ is a cycle with one source and one sink.)
The lattice $L$ is \newword{polygonal} if the following two conditions hold:
First, if distinct elements $y_1$ and $y_2$ both cover an element $x$, then $[x,y_1\join y_2]$ is a polygon; 
and second, if an element $y$ covers distinct elements $x_1$ and $x_2$, then $[x_1\meet x_2,y]$ is a polygon.

If $L$ is a polygonal lattice, then we define a quiver $\FPoly(L)$ whose set of vertices is $\Hasse_1(L)$, with arrows defined in every polygon $P$ of $L$ as follows.
The two arrows into the bottom element of $P$ are called \newword{bottom arrows}, while the two arrows from the top element of $P$ are called \newword{top arrows}, and all other arrows of $P$ are called \newword{side arrows}.
Every bottom arrow of $P$ has an arrow (in $\FPoly(L)$) to the opposite top arrow in $P$ (i.e.\ the top arrow in the opposite chain) and has an arrow to every side arrow in $P$.
Every top arrow of $P$ has an arrow to the opposite bottom arrow and every side arrow in~$P$.
For example, a square and hexagon in $L$ would contribute to $\FPoly(L)$ as indicated below.
\[\raisebox{-8pt}{$\xymatrix@R.8em@C1em{&\bullet\ar@{->}[dl]_a\ar@{->}[dr]^c\\ \bullet\ar@{->}[dr]_b&&\bullet\ar@{->}[dl]^d\\  &\bullet}$}\ \ \ \ \ 
\raisebox{-26pt}{$\begin{array}{c}
a\leftrightarrow d\\
b\leftrightarrow c\\
\end{array}$}\qquad\qquad
\xymatrix@R.8em@C1em{&\bullet\ar@{->}[dl]_e\ar@{->}[dr]^h\\ \bullet\ar@{->}[d]_f&&\bullet\ar@{->}[d]^i\\ \bullet\ar@{->}[dr]_g&&\bullet\ar@{->}[dl]^j\\ &\bullet}\qquad
\raisebox{-26pt}{
$\begin{array}{ccc}
e\leftrightarrow j&e\to f&e\to i\\
g\leftrightarrow h&g\to f&g\to i\\
&h\to f&h\to i\\
&j\to f&j\to i\\
\end{array}$}
\]
The following is \cite[Theorem~9-6.5]{regions}.

\begin{theorem}\label{force in poly}
If $L$ is a finite polygonal lattice, and $x\to y$ and $x'\to y'$ are arrows in $\Hasse(L)$, then $\con(x',y')\le\con(x,y)$ if and only if there is a directed path from $x\to y$ to $x'\to y'$ in $\FPoly(L)$.
\end{theorem}

When $L$ is also congruence uniform, the quiver $\FPoly(L)$ is closely related to the forcing order on join-irreducible elements of $L$, as described in the following corollary, which is immediate by combining Theorem~\ref{force in poly} with facts about finite congruence uniform lattices already given in this section.
We say that a quiver $Q$ is \newword{strongly connected} if given any two vertices $x$ and $y$ of the quiver, there exists a path from $x$ to $y$ and a path from $y$ to $x$.
A \newword{strongly connected component} of $Q$ is a set of vertices of $Q$ that is maximal with respect to inducing a strongly connected subquiver.

\begin{cor}\label{force in poly to force}
Suppose $L$ is a finite, polygonal, congruence uniform lattice.
Then each strongly connected component of $\FPoly(L)$ contains exactly one arrow in $\Hasse(L)$ of the form $j\to j_*$ where $j$ is join-irreducible.
This bijection between join-irreducible elements of $L$ and strongly connected components of $\FPoly(L)$ is an isomorphism from the forcing order on join-irreducible elements of $L$ to the partial order induced by $\FPoly(L)$ on its strongly connected components.
\end{cor}

Informally, the corollary says that all forcing in a finite polygonal, congruence uniform lattice comes from forcing in polygons.
Indeed, the notation $\FPoly(L)$ suggests the phrase ``forcing in polygons.''

We define another quiver $\SFPoly(L)$, again with the set of vertices given by $\Hasse_1(L)$, but with strictly fewer arrows than $\FPoly(L)$.
The notation $\SFPoly(L)$ suggests the phrase ``strong forcing in polygons.''
In $\SFPoly(L)$, every bottom arrow of a polygon $P$ has an arrow only to the opposite top arrow in $P$ and every top arrow of $P$ has an arrow only to the opposite bottom arrow in~$P$.
Thus from the square and hexagon as labelled above, $\SFPoly(L)$ gets arrows as indicated below.
\[
\begin{array}{ccccc}
a\leftrightarrow d&b\leftrightarrow c&
\qquad&
e\leftrightarrow j&g\leftrightarrow h
\end{array}
\]
Strong forcing in polygons controls whether two Hasse arrows in a finite polygonal, congruence uniform lattice determine the same congruence, as described in the following corollary.

\begin{cor}\label{strong force in poly}
Suppose $L$ is a finite polygonal, congruence uniform lattice and let $x\to y$ and $x'\to y'$ be arrows of $\Hasse(L)$.
Then $\con(x,y)=\con(x',y')$ if and only if there is a directed path (or equivalently a path) in $\SFPoly(L)$ from $x\to y$ to $x'\to y'$.
\end{cor}

\begin{proof}
The ``if'' direction is immediate by Theorem~\ref{force in poly}, the fact that every arrow is $\SFPoly(L)$ is also an arrow in $\FPoly(L)$, and the fact that arrows in $\SFPoly(L)$ come in opposite pairs.
We prove the converse by proving that for every Hasse arrow $x\to y$, there is a directed path in $\SFPoly(L)$ from $x\to y$ to $j\to j_*$, where $j$ is the unique join-irreducible element with $\con(j,j_*)=\con(x,y)$.
Proposition~\ref{expl label} says in particular that $x\ge j$, so we can argue by induction on the length of a longest maximal chain from $j$ to $x$.
Choose $z$ with $x\gtrdot z\ge j$ and set $y'=y\meet z$.
Since $L$ is polygonal, $[y',x]$ is a polygon $P$, and $x\to y$ is a top arrow in $P$.
Choosing $x'$ so that $x'\to y'$ is the bottom arrow opposite $x\to y$ (i.e.\ in the other chain of $P$), we have $\con(x'\to y')=\con(x\to y)$ by Theorem~\ref{force in poly}.
Thus $\con(j,j_*)=\con(x',y')$, and thus $x'\ge j$.
By induction, there is a directed path in $\SFPoly(L)$ from $x'\to y'$ to $j\to j_*$, and using an arrow in $P$, we obtain a directed path in $\SFPoly(L)$ from $x\to y$ to $j\to j_*$.
\end{proof}

\subsection{Weak order preliminaries}\label{w-o-prelim} 
Fix a simply-laced Dynkin type (i.e., one of $A_n$ for $n\geq 1$, $D_n$ for $n\geq 4$, $E_6$, $E_7$, or $E_8$).
Let $W$ be the finite Weyl group of that type.  
For background on Weyl groups, see \cite{BB}.  

Let $S=\{s_1,\dots,s_n\}$ be the set of simple reflections of $W$.  
By definition, any element $w$ of $W$ can be written as a product of the simple reflections.  
Such an expression for $w$ of minimal length is called \newword{reduced}.  
The length of a reduced expression is the \newword{length} of $w$, denoted $\ell(w)$.   

The \newword{(right) weak order} on $W$ is the partial order with $u \geq v$ if $\ell(u)=\ell(v)+\ell(v^{-1}u)$. 
We write $\Hasse(W)$ for the Hasse quiver of the weak order on $W$.
The arrows of $\Hasse(W)$ are all arrows $ws_i \rightarrow w$ such that $w\in W$, $s_i\in S$, and $\ell(ws_i)>\ell(w)$, or equivalently $\ell(ws_i)=\ell(w)+1$.
For $w\in W$, we denote by $w^+$ (respectively, $w^-$) the set of 
arrows in $\Hasse(W)$ starting (respectively, ending) at $w$.

The weak order on $W$ is a finite lattice.
In particular, it has a maximal element, denoted $w_0$ and often called the \newword{longest element} of $W$.
For our purposes, the most important properties of the weak order are the following.

\begin{theorem}\label{W cong unif}
$W$ is congruence uniform \cite{bounded} and polygonal \cite[Theorem~10-3.7]{regions2}.  
 \end{theorem}

\subsection{Algebraic preliminaries}
Fix a base field $k$, and let $A$ be a finite-dimensional $k$-algebra.
We write $\mod A$ for the finite-dimensional left $A$-modules.  
We denote by $\tau$ and $\tau^-$ the Auslander-Reiten translations of $A$.
They give mutually inverse bijections between isomorphism classes of indecomposable non-projective $A$-modules and those of indecomposable non-injective $A$-modules.

We say that a full subcategory $\TT$ of $\mod A$ is a \newword{torsion class}
if it is closed under factors, isomorphisms and extensions. 
\newword{Torsion-free classes} are defined dually.  We write 
$\tors A$ for the torsion classes of $A$, and $\torf A$ for its 
torsion-free classes.  We view $\tors A$ and $\torf A$ as posets under the inclusion order.  
Then we have an anti-isomorphism given by $\TT\mapsto\TT^\perp=\{X\in\mod A\mid\Hom_A(\TT,X)=0\}$
\[\tors A\to\torf A,\]
whose inverse is given by $\FF\mapsto{}^\perp\FF=\{X\in\mod A\mid\Hom_A(X,\FF)=0\}$.

We recall that a torsion class $\TT$ of $\mod A$ is \newword{functorially finite}
if there exists $M\in\mod A$ such that $\TT=\Fac M$, where $\Fac M$ is the full
subcategory of $\mod A$ consisting of factor modules of finite direct sums of
copies of $M$ \cite{AS}.
We denote by $\ftors A$ (respectively, $\ftorf A$) the set of all functorially finite
torsion (respectively, torsionfree) classes in $\mod A$.
We view $\ftors A$ also as a poset under inclusion. 
The above anti-isomorphism restricts to an anti-isomorphism $\ftors A\to\ftorf A$ \cite{S}.

There is a bijection between $\ftors A$ and a certain class of $A$-modules.
Recall that a module $M\in\mod A$ is called \newword{$\tau$-rigid} if $\Hom_A(M,\tau M)=0$, and \newword{$\tau^-$-rigid} if $\Hom_A(\tau^-M,M)=0$.
A module $M\in\mod A$ is called \newword{$\tau$-tilting} if it is $\tau$-rigid and $|M|=|A|$
holds, where $|M|$ is the number of non-isomorphic indecomposable direct summands of $M$.
A module  $M\in\mod A$ is called \newword{support $\tau$-tilting} if there exists an idempotent 
$e$ of $A$ such that $M$ is a $\tau$-tilting $(A/\langle e\rangle)$-module.
We denote by $\sttilt A$ the set of isomorphism classes of basic support $\tau$-tilting $A$-modules,
and by $\IndtRig A$ (respectively, $\IndtmRig A$) the set of isomorphism classes of indecomposable
$\tau$-rigid (respectively, $\tau^-$-rigid) $A$-modules.
  (See \cite{AIR} for more background on these notions.)
By \cite[2.7]{AIR}, we have a surjection $\{\mbox{$\tau$-rigid $A$-modules}\}\to\ftors A$
given by $M\mapsto\Fac M$, which induces a bijection
\[\sttilt A\xrightarrow{\sim}\ftors A.\]

Recall that $A$ is \newword{$\tau$-tilting finite} if $\sttilt A$ is a finite set, or equivalently, $\IndtRig A$ is a finite set.  
It is shown in \cite{DIJ} and \cite[1.2]{IRTT} that the following conditions are equivalent.
\begin{itemize}
\item $A$ is $\tau$-tilting finite
\item $\ftors A$ is a finite set.
\item $\ftors A$ (respectively, $\ftorf A$) forms a complete lattice.
\item $\ftors A=\tors A$.
\end{itemize}
Via the bijection between $\ftors A$ and $\sttilt A$, we obtain a partial
order on $\sttilt A$. We refer to the poset on $\sttilt A$ as \newword{generation order}. 
The arrows of the Hasse diagram of this poset are 
\newword{mutations} (see \cite[Theorem 0.6]{AIR}).  

Using $\tau$-tilting theory, we have the following description of join-irreducible elements in $\tors A$.

\begin{theorem}\label{join irreducible in sttilt}
Let $A$ be a finite-dimensional $k$-algebra which is $\tau$-tilting finite.
Then we have a bijection given by  $M\mapsto\Fac M$
\[\IndtRig A\to\jirr(\tors A).\]
The inverse map is given by $\TT\mapsto N$, where $M\in\sttilt A$ satisfies $\Fac M=\TT$ and $N$ is a unique indecomposable direct summand of $M$ satisfying $\Fac N=\Fac M$.
\end{theorem}

\begin{proof}
Since $A$ is $\tau$-tilting finite, we have $\sttilt A\simeq\tors A$.
For $\TT\in\tors A$, we take $M\in\sttilt A$ such that $\TT=\Fac M$, and an idempotent $e\in A$ such that $M$ is a $\tau$-tilting $(A/(e))$-module. 
Take a basic module $P\in\mod A$ such that $\add P=\add Ae$. Then we can write $M=M_1\oplus\cdots\oplus M_m$ and $P=P_{m+1}\oplus\cdots\oplus P_n$, where $n=|A|$ and each $M_k$ and $P_k$ is indecomposable.
It is shown in \cite{AIR} that adjacent vertices to $\TT$ in the Hasse quiver of $\tors A$ are given by $\TT_k:=\Fac\mu_k(M)$ for $1\le k\le n$,
where $\mu_k(M)$ is the mutation of $M$ at $M_k$ for $1\le k\le m$, and at $P_k$ for $m<k\le n$. In particular, either $\TT_k\subset\TT$ or $\TT_k\supset\TT$ holds.

We claim that $\TT_k\subset\TT$ holds if and only if $1\le k\le m$ and $M_k\notin\Fac(M/M_k)$. If $m<k\le n$, then $M\in\add\mu_k(M)$, and hence $\TT_k\supset\TT$.
Assume $1\le k\le m$. Then $\mu_k(M)=(M/M_k)\oplus M_k^*$ holds for some $M_k^*$ which is either indecomposable or zero.
In this case, precisely one of $M_k\in\Fac(M/M_k)$ and $M_k^*\in\Fac(M/M_k)$ holds, and the former condition holds if and only if $\TT_k\supset\TT$. Thus the claim follows.

Therefore $\TT$ is join-irreducible if and only if there exists a unique $1\le k\le m$ satisfying $M_k\notin\Fac(M/M_k)$. This holds if and only if $\TT=\Fac M_k$ for some $k$. In fact, the `if' part is clear, and the `only if' part follows from the general fact that all minimal direct summands $N$ of $M$ satisfying $\Fac N=\Fac M$ are isomorphic.
Since $M_k$ is an indecomposable $\tau$-rigid $A$-module, we have the assertion.
\end{proof}

For a complete lattice $L$, we denote by $\mathsf{j\text{-}Irr^c} L$ the set of \newword{completely join-irreducible elements}, that is, elements $a\in L$ such that $a=\Join S$ for a subset $S$ of $L$ implies $a\in S$.
If we drop the $\tau$-tilting finiteness assumption on $A$, then we still have a bijection $\IndtRig A\to\ftors A\cap\mathsf{j\text{-}Irr^c}(\tors A)$ given by $M\mapsto\Fac M$. The proof is the same, we only need to use \cite[Theorem 3.1]{DIJ}.
Note that completely join-irreducible elements in $\tors A$ are not necessarily functorially finite.
For example, consider $\mod kQ$ for a Kronecker quiver $Q$. Then all preinjective modules together with one tube form such a torsion class.

\medskip
We denote by $\brick A$ the set of isomorphism classes of bricks of $A$.
A full subcategory $\WW$ of $\mod A$ is called \newword{wide} if it is closed under kernels, cokernels and extensions.
In this case, $\WW$ forms an abelian category and the inclusion functor $\WW\to\mod A$ is exact.
A wide subcategory $\WW$ is called \newword{local} if it contains a unique simple object up to isomorphism. 
We denote by $\lwide A$ the set of local wide subcategories of $\mod A$.
We have the following easy observation.

\begin{proposition}\label{stone and wide}  
Let $A$ be a finite-dimensional $k$-algebra. Then we have a bijection
\[\brick A\to\lwide A\ \mbox{ given by }\ S\mapsto\Filt S,\]
where $\Filt S$ consists of $A$-modules $X$ which have a filtration $X=X_0\supset X_1\supset\cdots\supset X_{\ell-1}\supset X_\ell=0$ with $\ell\ge0$ such that $X_i/X_{i+1}\simeq S$ for any $0\le i<\ell$.
\end{proposition}

\begin{proof}
Let $S$ be a brick of $A$. Then $\Filt S$ is a wide subcategory by \cite[1.2]{Ri}. Clearly $\Filt S$ has a unique simple object $S$. Conversely, let $\WW$ be a local wide subcategory of $\mod A$ with a simple object $S$. For any endomorphism $f:S\to S$ of $S$, its image belongs to $\WW$. Therefore it is either $0$ or $S$, and $f$ is an isomorphism in the latter case. Therefore $S$ is a brick and we have $\WW=\Filt S$.
\end{proof}

\subsection{Preliminaries on preprojective algebras}
Let $\Pi=\Pi(W)$ be a preprojective algebra of the same Dynkin type as $W$.  
To construct $\Pi$, take the Dynkin diagram and replace each edge by a pair of opposite arrows $a, a^*$ to obtain the quiver $\overline Q$.  
Then $\Pi(W)$ is the path algebra of $\overline Q$ modulo the ideal generated by $\sum_a (aa^*-a^*a)$.
It is
a finite-dimensional self-injective algebra.  
We write $S_i$ for the simple
module corresponding to vertex $i$.  

We let $e_i$ be the idempotent corresponding to the vertex $i$.  Let
$I_i$ be the two-sided ideal $\Pi(1-e_i)\Pi$.  It is maximal as a left
ideal and as a right ideal.  For each $w \in W$, we take a reduced
word $w=s_{i_1}\dots s_{i_k}$, and we define   
\[I(w)=I_{i_1}\dots I_{i_k},\ \Tors(w):=\Fac I(w)\ \mbox{ and }\ \Torf(w):=\Sub(\Pi/I(w)).\]
Here $\Sub X$ refers to the subcategory consisting of subobjects of direct sums of copies of $X$.

The following result due to Mizuno is fundamental.

\begin{theorem}\label{Mizuno's theorem}\cite[2.14, 2.21]{Mi}
\begin{enumerate}[\rm(a)]
\item $I(w)$ does not depend on the choice of the reduced word for $w$.
\item We have bijections given by $w\mapsto I(w)\mapsto\Tors(w)$:
\begin{equation}\label{Mizuno}
W\xrightarrow{\sim}\sttilt\Pi\xrightarrow{\sim}\tors\Pi.
\end{equation}
\item The bijection above from $W$ to $\tors\Pi$ is an anti-isomorphism from weak order on $W$ to inclusion order on torsion classes.   
\end{enumerate}
\end{theorem}

We remark that Mizuno uses right modules, while we use left modules, and
therefore his results need to be suitably translated to account for this
difference.  (In particular, this has the effect that he uses left weak order
while we use right weak order.)  We also remark here that \cite{AIR} uses right
modules, but writes $I_w$ for the ideal which we would refer to as 
$I(w^{-1})$.

Thus $I(e)=\Pi$ gives the maximum torsion class $\Tors(e)=\mod\Pi$ and $I(w_0)=0$ gives the minimum
torsion class $\Tors(w_0)=\{0\}$, where $e\in W$ is the identity and $w_0\in W$ is the longest element.

The following is an easy consequence of \eqref{Mizuno}.

\begin{lemma}\label{additivity}
Let $w,v\in W$. Then $\ell(wv)=\ell(w)+\ell(v)$ holds if and only if
$I(wv)=I(w)I(v)$ holds.
If this holds, then we have $I(wv)=I(w)I(v)=I(w)\otimes_\Pi I(v)$.
\end{lemma}

We prepare the following results on indecomposable $\tau$-rigid $\Pi$-modules.

\begin{corollary}\label{indec tau rigid}
\begin{enumerate}[\rm(a)]\item 
We have a surjection given by $(w,i)\mapsto I(w)e_i$
\[W\times Q_0\to(\IndtRig\Pi)\sqcup\{0\}.\]
\item We have a bijection
\[\mirr W\xrightarrow{\sim}\IndtRig\Pi\]
given by $w\mapsto I(w)e_k$, where $k$ is the unique vertex satisfying $\ell(ws_k)=\ell(w)+1$.
\end{enumerate}
\end{corollary}

\begin{proof}
(a) Since $I(w)e_i$ is a direct summand of a support $\tau$-tilting module $I(w)$, it is $\tau$-rigid. It is either indecomposable or zero since it is a submodule of $\Pi e_i$ and hence its socle is either simple or zero.
The map is surjective since any $\tau$-rigid module is a direct summand of a $\tau$-tilting module.

(b) By Theorem \ref{Mizuno's theorem}, there is a bijection $\mirr W\simeq\jirr(\tors\Pi)$ given by $w\mapsto\Fac I(w)$.
Applying Theorem \ref{join irreducible in sttilt} to $\Pi$, we have a bijection $\jirr(\tors\Pi)\to\IndtRig\Pi$ sending $\Fac I(w)$ to $I(w)e_k$, where $k$ is the unique vertex satisfying $\Fac(I(w)e_k)=\Fac I(w)$.
This equality is equivalent to $\Fac I(ws_k)\subset\Fac I(w)$, and also to $\ell(ws_k)=\ell(w)+1$. Thus the assertion follows.
\end{proof}

We show a few examples.
Following the usual convention, we describe $\Pi$-modules $X$ in terms of their composition factors. Thus the numbers $i$ give a basis of the $k$-vector space $e_iX$, and the action of arrows is shown by the relative position of these numbers. For example,
$\begin{smallmatrix}1&&\\&2&\\&&3\end{smallmatrix}$ refers to the indecomposable $\Pi$-module with 
composition factors $S_1,S_2,S_3$ from top to bottom.  

\begin{example}  \label{A2 ex}
The left picture in Figure~\ref{A2 fig} shows the weak order on permutations in $S_3$ (the Weyl group of type $A_2$).
\begin{figure}
\[\begin{xy}
(0,0) *+{321}="A",
(18,-9) *+{{\color{red}312}}="B",
(18,-18) *+{{\color{red}132}}="C",
(-18,-9) *+{{\color{red}231}}="D",
(-18,-18) *+{{\color{red}213}}="E",
(0,-27) *+{123}="F",
\ar"A";"B",
\ar"B";"C",
\ar"C";"F",
\ar"A";"D",
\ar"D";"E",
\ar"E";"F"
\end{xy}\qquad
\begin{xy}
(0,0) *+{0}="A",
(18,-9) *+{\color{red}\left[\begin{smallmatrix}
1
\end{smallmatrix}\right]}="B",
(18,-18) *+{\color{red}\left[\begin{smallmatrix}
1\\ &2
\end{smallmatrix}\middle|
\begin{smallmatrix}
1
\end{smallmatrix}\right]}="C",
(-18,-9) *+{\color{red}\left[\begin{smallmatrix}
2
\end{smallmatrix}\right]}="D",
(-18,-18) *+{\color{red}\left[\begin{smallmatrix}
2
\end{smallmatrix}\middle|
\begin{smallmatrix}
&2\\ 1
\end{smallmatrix}\right]}="E",
(0,-27) *+{\left[\begin{smallmatrix}
1\\ &2
\end{smallmatrix}\middle|
\begin{smallmatrix}
&2\\ 1
\end{smallmatrix}\right]}="F",
\ar@{<-}"A";"B",
\ar@{<-}"B";"C",
\ar@{<-}"C";"F",
\ar@{<-}"A";"D",
\ar@{<-}"D";"E",
\ar@{<-}"E";"F"
\end{xy}\]
\caption{The weak order and $\sttilt\Pi$ in type $A_2$}
\label{A2 fig}
\end{figure}
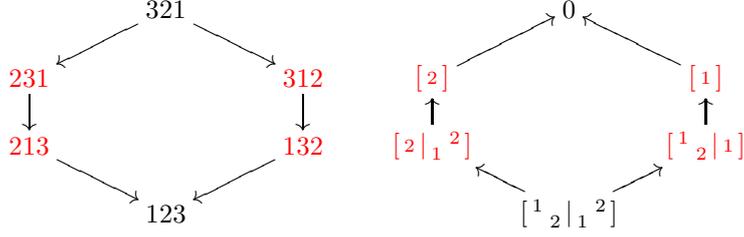
The right picture in Figure~\ref{A2 fig} shows the set $\sttilt\Pi$ for $\Pi$ of type $A_2$, arranged in generation order, where bars mean direct sums.
We show in red the meet-irreducible elements of $W$ and the join-irreducible elements of $\sttilt\Pi$.
The two pictures are arranged so that the map $w\to I(w)$ is accomplished by a translation of the page.
\end{example}

\begin{example}
Figure~\ref{A3 fig} shows the weak order on permutations in $S_4$ (the Weyl group of type $A_3$) and $\sttilt\Pi$ for $\Pi$ of type $A_3$, with the same conventions as described in Example~\ref{A2 ex}. 
\begin{figure}
\[\begin{xy}
(0,0) *+{4321}="1",
(30,-14) *+{\color{red}{4312}}="2",
(0,-14) *+{\color{red}{4231}}="3",
(-30,-14) *+{\color{red}{3421}}="4",
(50,-28) *+{\color{red}{4132}}="5",
(25,-28) *+{\color{red}{4213}}="6",
(0,-28) *+{3412}="7",
(-25,-28) *+{\color{red}{2431}}="8",
(-50,-28) *+{\color{red}{3241}}="9",
(35,-42) *+{\color{red}{1432}}="10",
(55,-42) *+{4123}="11",
(10,-42) *+{2413}="12",
(-10,-42) *+{\color{red}{3142}}="13",
(-55,-42) *+{2341}="14",
(-35,-42) *+{\color{red}{3214}}="15",
(50,-56) *+{1423}="16",
(25,-56) *+{1342}="17",
(-0,-56) *+{\color{red}{2143}}="18",
(-25,-56) *+{3124}="19",
(-50,-56) *+{2314}="20",
(30,-70) *+{1243}="21",
(0,-70) *+{1324}="22",
(-30,-70) *+{2134}="23",
(0,-84) *+{1234}="24",
\ar"1";"2",
\ar"1";"3",
\ar"1";"4",
\ar"2";"5",
\ar"2";"7",
\ar"3";"6",
\ar"3";"8",
\ar"4";"7",
\ar"4";"9",
\ar"5";"10",
\ar"5";"11",
\ar"6";"11",
\ar"6";"12",
\ar"7";"13",
\ar"8";"12",
\ar"8";"14",
\ar"9";"14",
\ar"9";"15",
\ar"10";"16",
\ar"10";"17",
\ar"11";"16",
\ar"12";"18",
\ar"13";"17",
\ar"13";"19",
\ar"14";"20",
\ar"15";"19",
\ar"15";"20",
\ar"16";"21",
\ar"17";"22",
\ar"18";"21",
\ar"18";"23",
\ar"19";"22",
\ar"20";"23",
\ar"21";"24",
\ar"22";"24",
\ar"23";"24",
\end{xy}\]
\vspace{10pt}
\[\begin{xy}
(0,-84) *+{
\left[\begin{smallmatrix}
1\\ &2\\ &&3\end{smallmatrix}\middle|
\begin{smallmatrix}
&2\\ 1&&3\\ &2\end{smallmatrix}\middle|
\begin{smallmatrix}
&&3\\ &2\\ 1\end{smallmatrix}\right]}="1",
(-30,-70) *+{\left[\begin{smallmatrix}
\\ 2\\ &3\end{smallmatrix}\middle|
\begin{smallmatrix}
&2\\ 1&&3\\ &2\end{smallmatrix}\middle|
\begin{smallmatrix}
&&3\\ &2\\ 1\end{smallmatrix}\right]}="2",
(0,-70) *+{\left[\begin{smallmatrix}
1\\ &2\\ &&3\end{smallmatrix}\middle|
\begin{smallmatrix}
1&&3\\ &2\end{smallmatrix}\middle|
\begin{smallmatrix}
&&3\\ &2\\ 1\end{smallmatrix}\right]}="3",
(30,-70) *+{\left[\begin{smallmatrix}
1\\ &2\\ &&3\end{smallmatrix}\middle|
\begin{smallmatrix}
&2\\ 1&&3\\ &2\end{smallmatrix}\middle|
\begin{smallmatrix}
&2\\ 1\end{smallmatrix}\right]}="4",
(-50,-56) *+{\left[\begin{smallmatrix}
2\\ &3\end{smallmatrix}\middle|
\begin{smallmatrix}
&3\\ 2\end{smallmatrix}\middle|
\begin{smallmatrix}
&&3\\ &2\\ 1\end{smallmatrix}\right]}="5",
(-25,-56) *+{\left[\begin{smallmatrix}
3\end{smallmatrix}\middle|
\begin{smallmatrix}
1&&3\\ &2\end{smallmatrix}\middle|
\begin{smallmatrix}
&&3\\ &2\\ 1\end{smallmatrix}\right]}="6",
(0,-56) *+{\color{red}\left[\begin{smallmatrix}
\\ 2\\ &3\end{smallmatrix}\middle|
\begin{smallmatrix}
&2\\ 1&&3\\ &2\end{smallmatrix}\middle|
\begin{smallmatrix}
&2\\ 1\end{smallmatrix}\right]}="7",
(25,-56) *+{\left[\begin{smallmatrix}
1\\ &2\\ &&3\end{smallmatrix}\middle|
\begin{smallmatrix}
1&&3\\ &2\end{smallmatrix}\middle|
\begin{smallmatrix}
1\end{smallmatrix}\right]}="8",
(50,-56) *+{\left[\begin{smallmatrix}
1\\ &2\\ &&3\end{smallmatrix}\middle|
\begin{smallmatrix}
1\\ &2\end{smallmatrix}\middle|
\begin{smallmatrix}
&2\\ 1\end{smallmatrix}\right]}="9",
(-55,-42) *+{\left[\begin{smallmatrix}
2\\ &3\end{smallmatrix}\middle|
\begin{smallmatrix}
&3\\ 2\end{smallmatrix}\right]}="10",
(-35,-42) *+{\color{red}\left[\begin{smallmatrix}
3\end{smallmatrix}\middle|
\begin{smallmatrix}
&3\\ 2\end{smallmatrix}\middle|
\begin{smallmatrix}
&&3\\ &2\\ 1\end{smallmatrix}\right]}="11",
(-10,-42) *+{\color{red}\left[\begin{smallmatrix}
3\end{smallmatrix}\middle|
\begin{smallmatrix}
1&&3\\ &2\end{smallmatrix}\middle|
\begin{smallmatrix}
1\end{smallmatrix}\right]}="12",
(10,-42) *+{\left[\begin{smallmatrix}
\\ 2\\ &3\end{smallmatrix}\middle|
\begin{smallmatrix}
2\end{smallmatrix}\middle|
\begin{smallmatrix}
&2\\ 1\end{smallmatrix}\right]}="13",
(35,-42) *+{\color{red}\left[\begin{smallmatrix}
1\\ &2\\ &&3\end{smallmatrix}\middle|
\begin{smallmatrix}
1\\ &2\end{smallmatrix}\middle|
\begin{smallmatrix}
1\end{smallmatrix}\right]}="14",
(55,-42) *+{\left[\begin{smallmatrix}
1\\ &2\end{smallmatrix}\middle|
\begin{smallmatrix}
&2\\ 1\end{smallmatrix}\right]}="15",
(-50,-28) *+{\color{red}\left[\begin{smallmatrix}
3\end{smallmatrix}\middle|
\begin{smallmatrix}
&3\\ 2\end{smallmatrix}\right]}="16",
(-25,-28) *+{\color{red}\left[\begin{smallmatrix}
2\\ &3\end{smallmatrix}\middle|
\begin{smallmatrix}
2\end{smallmatrix}\right]}="17",
(-0,-28) *+{\left[\begin{smallmatrix}
1\end{smallmatrix}\middle|
\begin{smallmatrix}
3\end{smallmatrix}\right]}="18",
(25,-28) *+{\color{red}\left[\begin{smallmatrix}
2\end{smallmatrix}\middle|
\begin{smallmatrix}
&2\\ 1\end{smallmatrix}\right]}="19",
(50,-28) *+{\color{red}\left[\begin{smallmatrix}
1\\ &2\end{smallmatrix}\middle|
\begin{smallmatrix}
1\end{smallmatrix}\right]}="20",
(-30,-14) *+{\color{red}\left[\begin{smallmatrix}
3\end{smallmatrix}\right]}="21",
(0,-14) *+{\color{red}\left[\begin{smallmatrix}
2\end{smallmatrix}\right]}="22",
(30,-14) *+{\color{red}\left[\begin{smallmatrix}
1\end{smallmatrix}\right]}="23",
(0,0) *+{0}="24",
\ar"1";"2",
\ar"1";"3",
\ar"1";"4",
\ar"2";"5",
\ar"2";"7",
\ar"3";"6",
\ar"3";"8",
\ar"4";"7",
\ar"4";"9",
\ar"5";"10",
\ar"5";"11",
\ar"6";"11",
\ar"6";"12",
\ar"7";"13",
\ar"8";"12",
\ar"8";"14",
\ar"9";"14",
\ar"9";"15",
\ar"10";"16",
\ar"10";"17",
\ar"11";"16",
\ar"12";"18",
\ar"13";"17",
\ar"13";"19",
\ar"14";"20",
\ar"15";"19",
\ar"15";"20",
\ar"16";"21",
\ar"17";"22",
\ar"18";"21",
\ar"18";"23",
\ar"19";"22",
\ar"20";"23",
\ar"21";"24",
\ar"22";"24",
\ar"23";"24",
\end{xy}\]
\caption{The weak order and $\sttilt\Pi$ in type $A_3$}\label{A3 fig}
\end{figure}
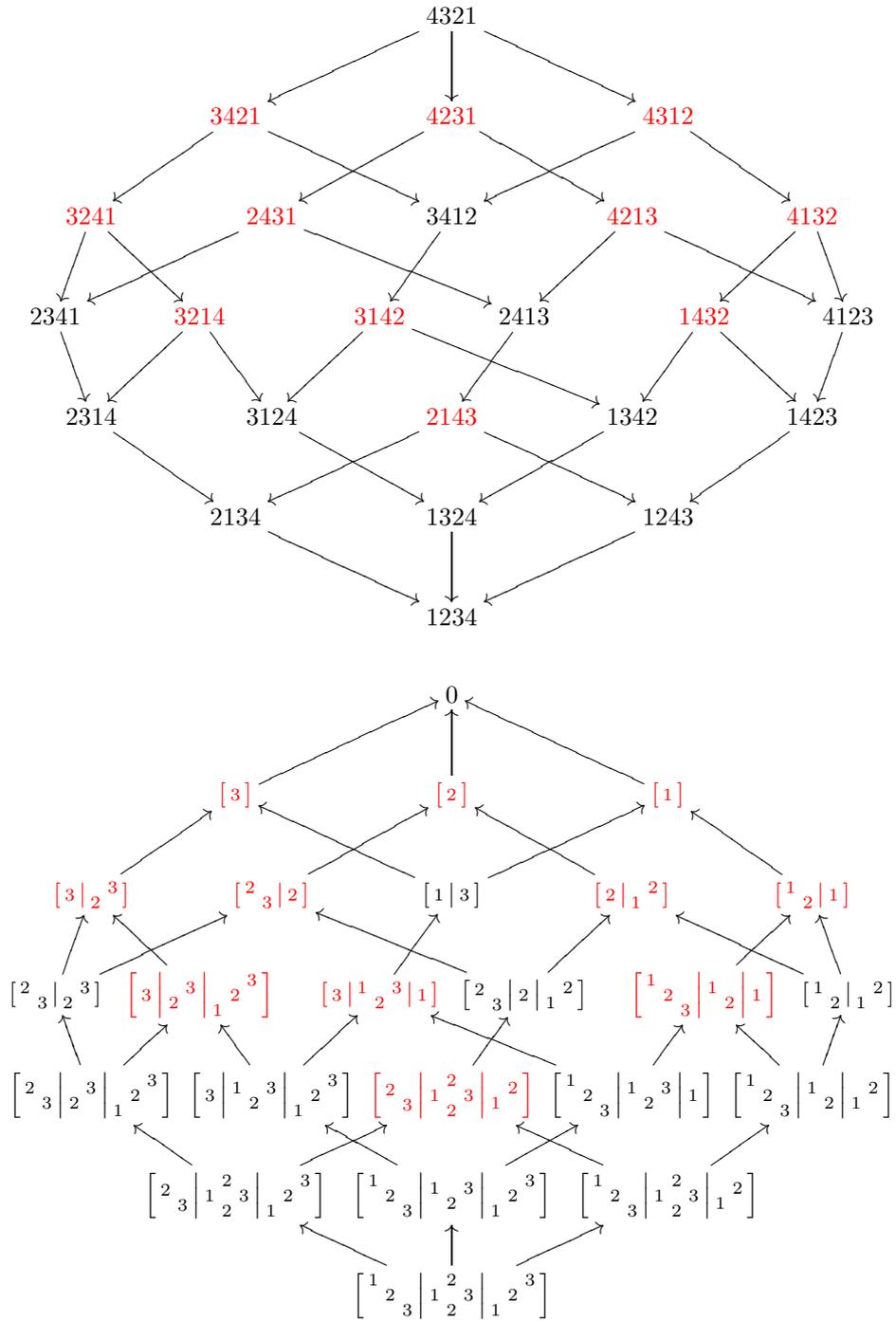
\end{example}

Let $\widehat{\Pi}$ be the preprojective algebra of the extended Dynkin type corresponding to $\Pi$.
We identify $\Pi$-modules with $\widehat{\Pi}$-modules annihilated by
$e_0$, where $e_0$ is an idempotent of $\widehat{\Pi}$ satisfying
$\Pi=\widehat{\Pi}/\langle e_0\rangle$.
Then the category $\mod\Pi$ can be regarded as a full subcategory of $\fd\widehat{\Pi}$ which is closed under extensions. Therefore we have
\begin{equation}\label{hom ext}
\Hom_\Pi(X,Y)=\Hom_{\widehat{\Pi}}(X,Y)\ \mbox{ and }\ \Ext^1_\Pi(X,Y)=\Ext^1_{\widehat{\Pi}}(X,Y)\end{equation}
for any $X,Y\in\mod\Pi$. On the other hand, $\widehat{\Pi}$ is a 2-Calabi--Yau algebra \cite[Section 4.2]{Ke1} and hence there exists a functorial isomorphism
\begin{equation}\label{2CY}
\Hom_{\DDD^{\bo}(\fd\widehat{\Pi})}(X,Y)\simeq D\Hom_{\DDD^{\bo}(\fd\widehat{\Pi})}(Y,X[2])
\end{equation}
for any $X,Y\in\DDD^{\bo}(\fd\widehat{\Pi})$ \cite[4.1]{Ke1}, where $D$ is the $k$-dual.

\section{Homological characterizations of layers}\label{stone}
In this section we prove Theorem~\ref{main of layer}, the homological characterization of layer modules of the preprojective algebra $\Pi$.
The assertion that every layer module $L$ is a stone satisfying $\End_\Pi(L)=k$ was shown in \cite[2.3]{AIRT}.
(Although $k$ was assumed to be algebraically closed in \cite{AIRT}, this assumption was never used.)
It is also easy to show that the stones $L$ of $\Pi$ satisfying $\End_{\Pi}(L)=k$ are exactly the $\Pi$-modules which are $2$-spherical as $\widehat{\Pi}$-modules.
Indeed, for any $\Pi$-module $L$, we have $\Ext^1_{\widehat{\Pi}}(L,L)=\Ext^1_\Pi(L,L)$ and $\Ext^2_{\widehat{\Pi}}(L,L)=D\End_\Pi(L)$ by \eqref{hom ext} and \eqref{2CY}, and the assertion follows.

To complete the proof, we will show that every brick is a layer.
As before, $\DDD^{\bo}(\fd\widehat{\Pi})$ is the bounded derived category of finite-dimensional $\widehat{\Pi}$-modules.
For a vertex $i\in Q_0$, let $\widehat{I}_i:=\widehat{\Pi}(1-e_i)\widehat{\Pi}$ be an ideal of
$\widehat{\Pi}$. Then we have an exact sequence
\begin{equation}\label{I Pi S}
0\to \widehat{I}_i\to\widehat{\Pi}\to S_i\to0
\end{equation}
of $\widehat{\Pi}$-bimodules. Moreover we have an autoequivalence
\[F_i:=\widehat{I}_i\Lotimes_{\widehat{\Pi}}-:\DDD^{\bo}(\fd\widehat{\Pi})\to\DDD^{\bo}(\fd\widehat{\Pi})\]
with quasi-inverse
\[F_i^{-1}:=\RHom_{\widehat{\Pi}}(\widehat{I}_i,-):
\DDD^{\bo}(\fd\widehat{\Pi})\to\DDD^{\bo}(\fd\widehat{\Pi}).\]

The following assertions are easy to check.

\begin{lemma}\label{idempotent quotient} Let $L$ be a finite-dimensional
$\Pi$-module.
\begin{enumerate}[\rm(a)]
\item We have $I_i=\widehat{I}_i\otimes_{\widehat{\Pi}}\Pi
=\Pi\otimes_{\widehat{\Pi}}\widehat{I}_i$.
\item We have  
$I_i\otimes_{\Pi}L=\widehat{I}_i\otimes_{\widehat{\Pi}}L$
and $\Hom_{\Pi}(I_i,L)=\Hom_{\widehat{\Pi}}(\widehat{I}_i,L)$.
\end{enumerate}
\end{lemma}

The following observation plays a key role.

\begin{proposition}\label{new spherical}
Let $L$ be a brick of $\Pi$ and $i$ a vertex of $\overline{Q}$.
\begin{enumerate}[\rm(a)]
\item If $L\not\simeq S_i$, then at least one of $\Hom_\Pi(S_i,L)=0$
or $\Hom_\Pi(L,S_i)=0$ holds.
\item If $\Hom_\Pi(S_i,L)=0$, then $F_i(L)\simeq I_i\otimes_\Pi L$
is a brick of $\Pi$.
\item If $\Hom_\Pi(L,S_i)=0$, then $F_i^{-1}(L)\simeq\Hom_\Pi(I_i,L)$
is a brick of $\Pi$ .
\end{enumerate}
\end{proposition}

\begin{proof}
(a) If both $\Hom_\Pi(S_i,L)$ and $\Hom_\Pi(L,S_i)$ are non-zero,
the composition $L\to S_i\to L$ is a non-zero non-isomorphic
endomorphism of $L$. This is a contradiction since $\End_{\Pi}(L)$
is a division algebra.

(b) We show that $F_i(L)$ is a $\widehat{\Pi}$-module, that is,
$\Tor^{\widehat{\Pi}}_j(\widehat{I}_i,L)=0$ holds for any $j\neq0$.
Since the $\widehat{\Pi}$-module $\widehat{I}_i$ has projective 
dimension at most one \cite[Proposition III.1.4]{BIRS}, we only have to check the case $j=1$.
We have
\[\Tor^{\widehat{\Pi}}_1(\widehat{I}_i,L)\stackrel{\eqref{I Pi S}}{\simeq}
\Tor^{\widehat{\Pi}}_2(S_i,L)\simeq D\Ext^2_{\widehat{\Pi}}(L,S_i)
\stackrel{\eqref{2CY}}{\simeq}\Hom_{\widehat{\Pi}}(S_i,L)=0\]
by our assumption, where the second isomorphism follows from a functorial isomorphism $-\otimes_{\widehat{\Pi}}L\simeq D\Hom_{\widehat{\Pi}}(L,D(-))$.
Thus we have $F_i(L)\simeq\widehat{I}_i\otimes_{\widehat{\Pi}}L$, which is 
isomorphic to $I_i\otimes_\Pi L$ by Lemma \ref{idempotent quotient}.
Since $F_i$ is an autoequivalence of $\DDD^{\bo}(\fd\widehat{\Pi})$,
$F_i(L)$ is a brick of $\Pi$.

(c) Similar to (b).
\end{proof}

We need the following observation.

\begin{lemma}\label{inductive step2}
Let $w\in W$ and let $L$ be a brick of $\Pi$ which belongs to $\Tors(w)$.
For any $i$ satisfying $\Hom_\Pi(L,S_i)\neq0$,
the following assertions hold.
\begin{enumerate}[\rm(a)]
\item $\ell(s_iw)>\ell(w)$ and $I(s_iw)=I_iI(w)=I_i\otimes_\Pi I(w)$.
\item If $L\not\simeq S_i$, then $\Hom_\Pi(S_i,L)=0$ and $F_i(L)$ is a 
brick of $\Pi$ which belongs to $\Tors(s_iw)$.
\end{enumerate}
\end{lemma}

\begin{proof}
(a) Since $L\in\Tors(w)$ and $\Hom_\Pi(L,S_i)\neq0$, we have
$\Hom_\Pi(I(w),S_i)\neq0$. Thus $I_iI(w)\neq I(w)$ holds.
Thus we have the assertions by Lemma \ref{additivity}.

(b) By Proposition \ref{new spherical}, $\Hom_\Pi(S_i,L)=0$ holds, and
$F_i(L)$ is a brick of $\Pi$.
Since we have a surjection $I(w)^{\oplus m}\to L$, we have a surjection
$I(s_iw)^{\oplus m}=I_i\otimes_\Pi I(w)^{\oplus m}\to 
I_i\otimes_\Pi L=F_i(L)$. Thus $F_i(L)$ belongs to $\Tors(s_iw)$.
\end{proof}

For $w\in W$, we have an ideal $\widehat{I}(w)$ of $\widehat{\Pi}$
satisfying $\widehat{I}(w)\supseteq\langle e_0\rangle$ and
$\widehat{I}(w)/\langle e_0\rangle=I(w)$. Let
\[F(w):=\widehat{I}(w)\Lotimes_{\widehat{\Pi}}-:
\DDD^{\bo}(\fd\widehat{\Pi})\to\DDD^{\bo}(\fd\widehat{\Pi}).\]
For any reduced expression $w=s_{i_1}\cdots s_{i_\ell}$, we have
\[
\widehat{I}(w)=\widehat{I}_{i_1}\cdots\widehat{I}_{i_\ell}\ \mbox{ and }\ 
F(w)=F_{i_1}\circ\cdots\circ F_{i_\ell}.
\]

\begin{lemma}\label{repeat inductive step}
Let $L$ be a brick of $\Pi$. Then there exists $v\in W$ such that 
$F(v)(L)$ is a simple $\Pi$-module $S_i$.
\end{lemma}

\begin{proof}
Assume that $L$ is not simple. Taking $i_1$ such that 
$\Hom_\Pi(L,S_{i_1})\neq0$ and applying Lemma 
\ref{inductive step2} to the brick $L\in\mod\Pi=\Tors(e)$,
we have that $F_{i_1}(L)$ is a brick in $\Fac I_{i_1}$.

Assume that $F_{i_1}(L)$ is not simple. Taking $i_2$ such that $\Hom_\Pi(F_{i_1}(L),S_{i_2})\neq0$ and applying Lemma 
\ref{inductive step2} to the brick $F_{i_1}(L)\in\Fac I_{i_1}$,
we have that $F_{i_2}F_{i_1}(L)$ is a brick in $\Fac(I_{i_2}I_{i_1})$
and $\ell(s_{i_2}s_{i_1})=2$ holds.

We repeat this process. Since the lengths of elements in $W$ are 
bounded by $\ell(w_0)$, the process must stop, that is, there exists 
$v\in W$ such that $F(v)(L)$ is a simple $\Pi$-module $S_i$.
\end{proof}

We need the following general observation.

\begin{lemma}\label{dual-layer}
For any $v,w\in W$ satisfying $I(v)\supset I(w)$, we have an isomorphism
$D(I(v)/I(w))\simeq  I(w^{-1}w_0)/I(v^{-1}w_0)$ of $\Pi$-modules.
\end{lemma}

\begin{proof}
Applying $D$ to the exact sequence $0\to I(w)\to I(v)\to I(v)/I(w)
\to 0$ of $\Pi^{\op}$-modules, we have an exact sequence
\begin{equation}\label{1st}
0\to D(I(v)/I(w))\to D(I(v))\to D(I(w))\to0
\end{equation}
of $\Pi$-modules. On the other hand, we have an exact sequence
\begin{equation}\label{2nd}
0\to I(w^{-1}w_0)/I(v^{-1}w_0)\to\Pi/I(v^{-1}w_0)\to
\Pi/I(w^{-1}w_0)\to0
\end{equation}
of $\Pi$-modules. Using an isomorphism $D(I(w))\simeq \Pi/I(w^{-1}w_0)$
of $\Pi$-modules for any $w\in W$ \cite{ORT}
and comparing \eqref{1st} and \eqref{2nd}, we have the assertion.
\end{proof}

Now we are ready to finish the proof of Theorem \ref{main of layer}.

\begin{proof}[Proof of Theorem \ref{main of layer}]
As shown in the first paragraph of this section, we can complete the proof by showing that every brick is a layer.
Let $L$ be a brick of $\Pi$.
By Lemma \ref{repeat inductive step}, there exists $v\in W$ such that 
$F(v)(L)=\widehat{I}(v)\Lotimes_{\widehat{\Pi}}L$ is a simple $\Pi$-module $S_i$.
Thus $L$ equals 
\begin{eqnarray*}
\RHom_{\widehat{\Pi}}(\widehat{I}(v),S_i)&\!\!=\!\!&\Hom_{\widehat{\Pi}}(\widehat{I}(v),S_i)
\stackrel{{\rm Lem. \ref{idempotent quotient}}}{=}\Hom_\Pi(I(v),S_i)=D(DS_i\otimes_\Pi I(v))\\
&\!\!=\!\!&D((\Pi/I_i)\otimes_\Pi I(v))=D(I(v)/I_iI(v))
\stackrel{{\rm Lem. \ref{additivity}}}{=}D(I(v)/I(s_iv)).
\end{eqnarray*}
This is a layer of $\Pi$ by Lemma \ref{dual-layer}.
\end{proof}

\section{Bijections: Theorems \ref{bijections}\ and \ref{labellings}}\label{layer} 
Throughout this section, let $\Pi$ be a preprojective algebra of Dynkin type and $W$ the corresponding Weyl group.
The main result of this section is that Figure \ref{lat alg diagram fig} is commutative.
Commutativity of that diagram includes Theorems~\ref{bijections} and~\ref{labellings}.

For an arrow $a:ws_i\to w$ in $H$, we have a natural inclusion $I(w)\supset I(ws_i)$ of ideals of $\Pi$, and we associate to $a$ the $\Pi$-module
$$L(a):=I(w)/I(ws_i)=(I(w)e_i)/(I(w)I_ie_i),$$
where the right equality follows from $I(w)=I(w)e_i\oplus I(w)(1-e_i)$, $I(ws_i)=I(w)I_ie_i\oplus I(w)I_i(1-e_i)$ and $I_i(1-e_i)=\Pi(1-e_i)$.
We also define maps $L:\mirr W\to\Layers\Pi$ and $M:\mirr W\to\IndtRig\Pi$ as
\[L(m):=L(a)\ \mbox{ and }\ M(m):=I(m)e_i\]
for $m\in\mirr W$ and $a:m^*=ms_i\to m$ in $H$ (see Proposition \ref{indec tau rigid} for $M(m)$).

\begin{theorem}\label{big diagram}  
Let $\Pi$ be a preprojective algebra of Dynkin type and $W$ the corresponding Weyl group. Then we have the commutative diagram shown in Figure~\ref{lat alg diagram fig}. 
The maps are bijections or surjections as marked with tildes ``$\sim$'' or double-headed arrows.
\end{theorem}

We start with the following simple observation, where $\rad M(m)_{\End_{\Pi}(M(m))}$ is the radical of the $\End_{\Pi}(M(m))^{\op}$-module $M(m)$.

\begin{lemma}\label{from M to L}
$L(m)=M(m)/\rad M(m)_{\End_{\Pi}(M(m))}$ holds for any $m\in\mirr W$.
\end{lemma}

\begin{proof}
We have $M(m)=I(m)e_i$ and $L(m)=I(m)e_i/I(m)I_ie_i$.
Thus we only have to show $I(m)I_ie_i=(I(m)e_i)\rad\End_\Pi(I(m)e_i)$.

Since $\Tors(m)=\Fac I(m)$ is join-irreducible, we have $I(m)\in\Fac(I(m)e_i)$. Thus 
\begin{eqnarray*}
(I(m)e_i)\rad\End_\Pi(I(m)e_i)&=&\sum_{f\in \rad\End_\Pi(I(m)e_i)}\Image(f:I(m)e_i\to I(m)e_i)\\
&=&\sum_{g\in\rad_\Pi(I(m),I(m)e_i)}\Image(g:I(m)\to I(m)e_i),
\end{eqnarray*}
where $\rad_{\Pi}$ is the radical of the category $\mod\Pi$ and hence $\rad_\Pi(I(m),I(m)e_i)$ consists of morphisms which are not split epimorphisms.

By \cite[Lemma 2.7]{Mi}, we have a surjection $\Pi\to\End_\Pi(I(m))$ given by $x\mapsto(y\mapsto yx)$. This induces
surjections $\Pi e_i\to\Hom_\Pi(I(m),I(m)e_i)$ and $I_i e_i\to\rad_\Pi(I(m),I(m)e_i)$.
Therefore we have
\[\sum_{g\in\rad_\Pi(I(m),I(m)e_i)}\Image(g:I(m)\to I(m)e_i)=I(m)I_ie_i,\]
which completes the proof.
\end{proof}

In Section~\ref{l-prelim} (particularly Theorem~\ref{force in poly} and Corollary~\ref{force in poly to force}), we described how the forcing order on join-irreducible elements interacts with the polygons (squares and rectangles) in $W$.
One ingredient in the proof of Theorem~\ref{big diagram} is a similar description of how the layer labelling interacts with polygons.

Recall from the introduction that a pair of layer modules $X,Y$ form a \newword{doubleton} if $\Ext_\Pi^1(Y,X)$ and $\Ext_\Pi^1(X,Y)$ are one-dimensional, and the corresponding extensions
are again layer modules.  
Recall also that the \newword{doubleton extension order} on layer modules is the transitive closure of the relation with $A>B$ if there exists a doubleton $A,C$ such that $B$ is the extension of $A$ by $C$ or of $C$ by $A$.  

\begin{proposition}\label{layer poly} 
For a polygon in $W$ (necessarily a square or hexagon), the layer labelling has the following configuration:
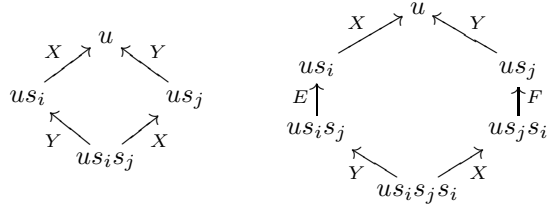
\[
\raisebox{-11pt}{$\xymatrix@R1em@C.5em{&us_is_j\\ us_i\ar@{<-}[ur]^Y&&us_j\ar@{<-}[ul]_X\\ &u\ar@{<-}[ul]^{X}\ar@{<-}[ur]_Y}$}\ \ \ \ \quad
\xymatrix@R1em@C.5em{&us_is_js_i\\ 
us_is_j\ar@{<-}[ur]^Y&&us_js_i\ar@{<-}[ul]_X\\ 
us_i\ar@{<-}[u]^E&&us_j\ar@{<-}[u]_F\\ 
&u\ar@{<-}[ul]^X\ar@{<-}[ur]_Y}\ \ \ 
\]
Moreover, in the hexagon case the layers $X$ and $Y$ form a doubleton and there exist short exact sequences of $\Pi$-modules:
\[0\to X\to E\to Y\to0\ \mbox{ and }\ 0\to Y\to F\to X\to0.\]
Thus $X\ge E$, $X\ge F$, $Y\ge E$, and $Y\ge F$ in the doubleton extension order.
\end{proposition}

\begin{proof}
Of the two diagrams in the statement of the theorem, the hexagon occurs if and only if $i$ and $j$ are neighbouring in~$\overline{Q}$.
We argue the hexagon case.  The square case is similar but simpler.
What we need to show is the following.
\begin{enumerate}[(i)]
\item There are isomorphisms of $\Pi$-modules:
\[I(u)/I(u)I_i\simeq I(u)I_jI_i/I(u)I_iI_jI_i\ \mbox{ and }\ I(u)/I(u)I_j\simeq I(u)I_iI_j/I(u)I_iI_jI_i.\]
\item There are short exact sequences of $\Pi$-modules:
\begin{eqnarray*}
&0\to I(u)/I(u)I_i\to I(u)I_i/I(u)I_iI_j\to I(u)I_iI_j/I(u)I_iI_jI_i\to0,&\\
&0\to I(u)/I(u)I_j\to I(u)I_j/I(u)I_jI_i\to I(u)I_jI_i/I(u)I_iI_jI_i\to0.&
\end{eqnarray*}
\end{enumerate}
We first show (i). The partial order tells us, by Lemma \ref{additivity},
that we have
\[I(u)/I(u)I_i=I(u)\otimes_\Pi(\Pi/I_i)\ \mbox{ and }\ 
I(u)I_jI_i/I(u)I_iI_jI_i=I(u)\otimes_\Pi(I_jI_i/I_iI_jI_i).\]
Since $\Pi/I_i\simeq S_i\simeq I_jI_i/I_iI_jI_i$ holds by an easy calculation
for the preprojective algebra $\Pi/I_iI_jI_i$ of type $A_2$ (see Figure \ref{A2 fig}), we have
\[I(u)/I(u)I_i=I(u)\otimes_\Pi(\Pi/I_i)\simeq I(u)\otimes_\Pi(I_jI_i/I_iI_jI_i)=I(u)I_jI_i/I(u)I_iI_jI_i.\]
The other isomorphism follows by interchanging $i$ and $j$.

Now we show (ii). Again by an easy calculation
for the preprojective algebra $\Pi/I_iI_jI_i$ of type $A_2$ (see Figure \ref{A2 fig}),
we have $\Pi/I_i\simeq S_i$, $I_i/I_iI_j\simeq{S_j\choose S_i}$ and
$I_iI_j/I_iI_jI_i\simeq S_j$. Thus there exists an exact sequence
\begin{eqnarray*}
0\to\Pi/I_i\to I_i/I_iI_j\to I_iI_j/I_iI_jI_i\to0
\end{eqnarray*}
of $\Pi$-modules. Applying $I(u)\otimes_\Pi-$ and using $\Tor^\Pi_1(I(u),S_j)=0$ (which follows from $\ell(us_j)=\ell(u)+1$ and Lemma \ref{additivity}), we obtain the first 
sequence. The second one follows by interchanging $i$ and $j$.

We verify that $X$ and $Y$ form a doubleton in the hexagon case. 
If the labels are simple, then the extension groups are certainly 
one-dimensional, and the extensions are layers by ${S_j\choose S_i}\simeq I_i/I_iI_j$.
Any hexagon is obtained by applying $I(u) \otimes_\Pi -$ to such a hexagon, as above.  
This implies the desired result once we note that, by Lemma \ref{idempotent quotient}, we can instead consider applying $\widehat{I}(u)\Lotimes_{\widehat{\Pi}}-$, which is an auto-equivalence of $\DDD^{\bo}(\fd\widehat{\Pi})$ by \cite{BIRS}.

Now $X\ge E$, $X\ge F$, $Y\ge E$, and $Y\ge F$ because there exist short exact sequences $0\to X\to E\to Y\to0$ and $0\to Y\to F\to X\to0$.
\end{proof}

The ideas in the proof above also lead to the following lemma.
For vertices $i\neq j$ in $\overline{Q}$, let $W_{i,j}:=\langle s_i,s_j\rangle\subset W$ be a parabolic subgroup of $W$.
For $w\in W$, the coset $wW_{i,j}$ is an interval in the weak order on $W$.
We write $H|_{wW_{i,j}}$ for the restriction of $\Hasse(W)$ to $wW_{i,j}$.
Define $w^+$ (respectively, $w^-$) to be the set of arrows in $\Hasse(W)$ starting (respectively, ending) at $w$.
For a set $S$ of $\Pi$-modules, we denote by $\Tors(S)$ (respectively, $\Torf(S)$) the smallest torsion (respectively, torsionfree) class in $\mod\Pi$ containing $S$.
For convenience and brevity, we will omit set braces inside the operator $\Tors(\bullet)$, so that, for example $\Tors(L_0,M_i\mid i\in I)$ would mean $\Tors(\set{L_0}\cup\set{M_i\mid i\in I})$.
Recall that the layer labelling of $\Hasse(W)$ maps each Hasse arrow $x\to y$ to the isomorphism class of the corresponding concrete layer $L(x\to y):=I(y)/I(x)$.

\begin{lemma}\label{inductive step}
If $a:ws_i\to w$ is an arrow in $\Hasse(W)$, and $j\neq i$ is a vertex in $\overline{Q}$, then
\[
\Tors(L(b)\mid b\in w^-\cap H|_{w W_{i,j}})
=\Tors(L(a),L(b)\mid b\in (ws_i)^-\cap H|_{wW_{i,j}}).\]
\end{lemma}

\begin{proof}
We have either the square or the hexagon in Proposition~\ref{layer poly}.
We argue the hexagon case. The square case is similar but simpler.
If $w$ coincides with $u$ in Proposition \ref{layer poly}, the desired equality reduces to an identity $\Tors(X,Y)=\Tors(X,E)$,
which follows from the exact sequence $0\to X\to E\to Y\to0$.
If $w$ coincides with $us_j$ in Proposition \ref{layer poly}, the desired equality reduces to an identity $\Tors(F)=\Tors(F,X)$,
which follows from the exact sequence $0\to Y\to F\to X\to0$.
If $w$ coincides with $us_is_j$ in Proposition \ref{layer poly}, the desired equality reduces to an identity $\Tors(Y)=\Tors(Y)$,
which clearly holds.
\end{proof}

Another ingredient in the proof of Theorem~\ref{big diagram} is a precise connection between torsion classes and
layer modules.

\begin{theorem}\label{canon m j thm}  
For any $w\in W$, we have
\[\Tors(w)=\Tors(L(a)\mid a\in w^-)\ \mbox{ and }\ 
\Torf(w)=\Torf(L(a)\mid a\in w^+).\]
\end{theorem}
\begin{proof}
We only prove the first equality since the second one is proved similarly.
We use decreasing induction on $W$.
The statement is clear for the longest element $w_0$ since both sides are $\{0\}$ in this case.

Let $a:ws_i\to w$ be an arrow in $\Hasse(W)$.
Assume that the assertion holds for $ws_i$, that is,
\begin{equation}\label{induction}
\Tors(ws_i)=\Tors(L(b)\mid b\in(ws_i)^-).
\end{equation}
Using obvious decompositions
\begin{equation}\label{decompose}
w^-=\bigcup_{j\neq i} \left(w^-\cap H|_{wW_{i,j}}\right)
\mbox{ and }\ 
(ws_i)^-=\bigcup_{j\neq i}\left((ws_i)^-\cap H|_{wW_{i,j}}\right),
\end{equation}
we have
\begin{eqnarray*}
\Tors(L(b)\mid b\in w^-)
&\stackrel{\eqref{decompose}}{=}&\Tors(\Tors(L(b)\mid b\in w^-\cap H|_{wW_{i,j}})\mid j\neq i)\\
&\stackrel{{\rm Lem. \ref{inductive step}}}{=}&
\Tors(\Tors(L(a),L(b)\mid b\in (ws_i)^-\cap H|_{wW_{i,j}})\mid j\neq i)\\
&\stackrel{\eqref{decompose}}{=}&
\Tors(L(a),\Tors(L(b)\mid b\in (ws_i)^-))\\
&\stackrel{\eqref{induction}}{=}&\Tors(L(a),\Tors(ws_i))\\
&=&\Tors(w),
\end{eqnarray*}
where the last equality follows from having an exact sequence
$0\to I(ws_i)\to I(w)\to L(a)\to0$ and $I(ws_i)\in\Tors(w)$.
\end{proof}

\begin{prop}\label{concrete}
Two arrows $x\to y$ and $x'\to y'$ in $\Hasse(W)$ have $\con(x,y)=\con(x',y')$ if and only if they have the same layer label.
\end{prop}
\begin{proof}
Proposition~\ref{layer poly} implies in particular that the map $(x\to y)\mapsto I(y)/I(x)$ is constant on components of $\SFPoly(W)$.
Thus by Corollary~\ref{strong force in poly}, if $x\to y$ and $x'\to y'$ have $\con(x,y)=\con(x',y')$, then they have the same layer labelling.

If $x\to y$ and $x'\to y'$ have the same layer labelling, let $m$ be the unique meet-irreducible element with $\con(m^*,m)=\con(x,y)$, and similarly let $m'$ be the meet-irreducible corresponding to $x'\to y'$.
Then $L(m^*\to m)=L((m')^*\to m')$, so by Theorem \ref{canon m j thm}, 
$\Tors(m)=\Tors(L(m^*\to m))=\Tors(L((m')^*\to m'))=\Tors(m')$.
Since $\Tors:W\to\tors\Pi$ is a bijection, we see that $m=m'$, so that $\con(x,y)=\con(x',y')$.
\end{proof}

\begin{proof}[Proof of Theorem \ref{big diagram}]
Commutativity of the two triangles involving $\mirr(W)$, $\jirr(W)$, $\ConJI(W)$, and $\Hasse_1(W)$, and correctness of the markings of the arrows as bijections or surjections was established in Section~\ref{l-prelim} in the more general context of congruence uniform lattices.
(See especially Figure~\ref{cong unif diagram fig}.)

The map from $\Hasse_1(W)$ to $\Layers\Pi$ is surjective by definition.
Proposition~\ref{concrete}, together with the commutativity of the triangles containing $\ConJI(W)$, then implies the commutativity of the two triangles containing the map from $\Hasse_1(W)$ to $\Layers\Pi$ and also implies that the map from $\mirr(W)$ to $\Layers\Pi$ is a bijection.
The triangle containing $\mirr(W)$, $\IndtRig\Pi$ and $\Layers\Pi$ commutes by Lemma \ref{from M to L},
and the left bottom square commutes by Theorem \ref{canon m j thm}. The remaining part of the diagram commutes dually.
The maps $M:\mirr W\to\IndtRig\Pi$, $\Fac:\IndtRig\Pi\to\jirr(\tors\Pi)$ and $\Filt:\brick\Pi\to\lwide\Pi$
are bijections by Corollary \ref{indec tau rigid}, Theorem \ref{join irreducible in sttilt} and Proposition \ref{stone and wide}.
The bijectivity of the other maps follows from commutativity.
\end{proof}

The bijections $\jirr W\simeq\Layers\Pi\simeq\mirr W$ given in Theorem \ref{big diagram},
combined with Proposition \ref{cong unif canon}, imply the following corollary.

\begin{corollary} $\Tors(w)$ is the smallest subcategory of $\Pi$-mod 
which is closed under extensions and quotients and contains the layers
corresponding to the canonical meet representation of $w$.  $\Torf(w)$ is 
the smallest subcategory of $\Pi$-mod which is closed under extensions and
subobjects and contains the layers corresponding to the canonical join
representation of $w$.  
\end{corollary}

\begin{remark}
Let $a:ws_i\to w$ be an arrow in $H$.
One might wonder if $L(a)$ is a unique brick of $\Pi$
which belongs to $\Tors(w)\backslash\Tors(ws_i)$.

This has an easy counterexample: let $\Pi$ be of type $A_2$, and
let $w:=e$ and $s_i:=s_1$. Then both layer modules $S_1$ and $P_1$
belong to $(\mod\Pi)\backslash(\Fac I_1)$.

On the other hand, $L(a)$ is a unique brick in $\Tors(w)\cap\Torf(ws_i)$ \cite{paperB}.
\end{remark}

\section{Doubleton extension order on layer modules}\label{order}
In this section, we prove Theorem~\ref{isom} and a characterization of the doubleton extension order on layer modules.
The last ingredient needed is the following proposition.

\begin{prop}\label{doubleton hex converse}
Suppose $A$, $B$, and $C$ are layers of $\Pi$ such that $A,C$ is a doubleton and $B$ is the extension of $C$ by $A$.  
Then there exists a hexagon in weak order such that the layer ordering of one of its chains is $(A,B,C)$, read either from bottom to top or from top to bottom.
\end{prop}

Before proving Proposition~\ref{doubleton hex converse}, we show how it completes the proof of Theorem~\ref{isom}.
\begin{proof}[Proof of Theorem~\ref{isom}]
Theorem~\ref{big diagram} already states that the map $j\mapsto I(j_*)/I(j)$ is a bijection from the set of join-irreducible elements of $W$ to the set of layer modules of $\Pi$.
Proposition~\ref{layer poly} implies that every arrow in the quiver $\FPoly(W)$ (defined in Section~\ref{l-prelim}) gives rise to an order relation in the doubleton extension order.
Proposition~\ref{doubleton hex converse} shows that each doubleton extension comes from some arrow in $\FPoly(W)$.
Corollary~\ref{force in poly to force} (which applies in light of Theorem~\ref{W cong unif}) thus implies that $j\mapsto I(j_*)/I(j)$ is an isomorphism from the forcing order on join-irreducible elements of $W$ to the doubleton extension order on layers.
\end{proof}

We now prepare to prove Proposition~\ref{doubleton hex converse}.

\begin{lemma} \label{mystery lemma}
If $X,Y$ is a doubleton, then  $\Hom_\Pi(X,Y)=\Hom_\Pi(Y,X)=0$.
\end{lemma}

\begin{proof}
Let $\langle-,-\rangle$ be the Euler form on $K_0(\fd\widehat{\Pi})$.
Since the extensions $E$ of $X,Y$ are layers by assumption, we have 
$\langle[X]+[Y],[X]+[Y]\rangle=\langle[E],[E]\rangle=2$. Thus
\[2\langle[X],[Y]\rangle=\langle[X]+[Y],[X]+[Y]\rangle-\langle[X],[X]\rangle-\langle[Y],[Y]\rangle=-2.\]
Since $\Ext_{\widehat{\Pi}}^1(X,Y)=\Ext_\Pi^1(X,Y)$ is one-dimensional by assumption, it follows that
$\Hom_\Pi(X,Y)=\Hom_{\widehat{\Pi}}(X,Y)=0$ and $\Hom_\Pi(Y,X)=D\Ext^2_{\widehat{\Pi}}(X,Y)=0$.
\end{proof}

The following lemma is an analogue for doubletons of Lemma \ref{inductive step2}.  

\begin{lemma} \label{double-two-six} If $X,Y$ is a doubleton contained in $\Tors(w)$,
and $\Hom_\Pi(X,S_i)\ne 0$ and $X\not\simeq S_i$, 
then $\ell(s_iw)=\ell(w)+1$ and 
$F_i(X), F_i(Y)$ form a doubleton contained in $\Tors(s_iw)$.  
\end{lemma}

\begin{proof} Since $\Hom_\Pi(X,S_i)\ne 0$, we have $\ell(s_iw)=\ell(w)+1$ by
Lemma \ref{inductive step2} and $\Hom_\Pi(S_i,X)= 0$ by Proposition \ref{new spherical}(a).
Since $\Hom_\Pi(X,Y)=0$ by Lemma \ref{mystery lemma}, then $\Hom_\Pi(S_i,Y)= 0$ also.  
Thus, by Proposition \ref{new spherical}(b), $F_i(X)$ and $F_i(Y)$ are layer 
modules.  Since 
$F_i$ is an auto-equivalence of $\DDD^{\bo}(\fd\widehat{\Pi})$,
it follows that $F_i(X)$, $F_i(Y)$ still form a doubleton.   

$F_i(X)$ belongs to $\Tors(s_iw)$ by Lemma \ref{inductive step2}(b).  If $\Hom_\Pi(Y,S_i)\ne 0$, then
Lemma \ref{inductive step2} is also directly 
applicable to $Y$, and tells us that $F_i(Y)$ also belongs to
$\Tors(s_iw)$.  In fact, the proof of Lemma \ref{inductive step2}(b) is actually applicable
to $Y$ even if $\Hom_\Pi(Y,S_i)=0$ --- all that is really needed is that 
$\Hom_\Pi(S_i,Y)=0$, and we have this.  Thus, $F_i(Y)$ also belongs to 
$\Tors(s_iw)$.  
\end{proof}

The following lemma is an analogue for doubletons of Lemma \ref{repeat inductive step}.

\begin{lemma} \label{double-two-seven} If $X,Y$ is a doubleton, then there exists $v\in W$ such that
$F(v)(X)$ and $F(v)(Y)$ are simple modules. \end{lemma}

\begin{proof}
Suppose that at least one of $X$ and $Y$ is not simple.  Without loss
of generality, suppose that $X$ is not simple.  Choose $S_i$
so that $\Hom_\Pi(X,S_i)\ne 0$.  Applying Lemma \ref{double-two-six}, 
we see that $F_i(X)$ and $F_i(Y)$ are a doubleton in $\Fac I_{s_i}$.  Assume that one of 
$F_i(X)$ and $F_i(Y)$ is not simple.  Repeat the previous procedure.  
As in the proof of Lemma \ref{repeat inductive step}, the procedure must terminate, at which 
point we have obtained a doubleton of simple modules.  
\end{proof}

\begin{proof}[Proof of Proposition \ref{doubleton hex converse}]   
Suppose that we have a doubleton $A,C$, with $B$ the extension
of $C$ by $A$.  We will establish 
that there exists a hexagon in weak order such that one
side of it is labelled $(A,B,C)$.  By Lemma \ref{double-two-seven}, there
exists $v\in W$ such that $F(v)(A)$, $F(v)(C)$ form a doubleton of 
simple modules, say $S_i, S_j$.  As in the proof of Theorem \ref{main of layer},
we conclude that $A,C$ are isomorphic to $D(I(v)/I(s_iv))$ and 
$D(I(v)/I(s_jv))$.  By Lemma \ref{dual-layer}, these are the labels of the 
Hasse arrows $v^{-1}w_0\to v^{-1}s_iw_0$ and $v^{-1}w_0\to v^{-1}s_jw_0$ which form the two top arrows of the desired hexagon.  Since the extension groups between $A$ and
$C$ are one-dimensional, the side of the hexagon whose arrows are labelled $(A,?,C)$, has $B$ as the label of its middle side.  
\end{proof}

It is immediate from the definition that if $A\ge B$ in the doubleton extension order, then $A$ is a subfactor of $B$. 
The converse does not hold in general.
In Section \ref{explicit}, we define conventions regarding modules over the preprojective algebra of type $D_n$.
In the notation of that section, in type $D_4$,
\[{\tiny\begin{array}{|cc|}\hline -2&{\begin{smallmatrix}-1\\ 1\end{smallmatrix}}\\ \hline -3& \\ \hline\end{array}}
\,\,\not\ge\,\,
{\tiny\begin{array}{|ccc|}\hline &2&3\\ \hline -2&{\begin{smallmatrix}-1\\ 1\end{smallmatrix}}&\\ \hline\end{array}}\,\,,\]
even though the first of these modules is a subfactor of the second.

We now show that the converse does hold in type $A_n$.  
We denote by $\SS$ the set of \newword{non-revisiting walks} on the double quiver $\overline{Q}$. By definition, these are walks in
$Q$ which follow a sequence of arrows either with or against the direction
of the arrow and which do not visit any vertex more than once.  We 
identify a walk and its reverse walk.

Let $I_{\cyc}$ denote the ideal of $\Pi$ generated by all 2-cycles and let $\overline{\Pi}:=\Pi/I_{\cyc}$.
To any $p\in \SS$, we can associate an indecomposable $\overline{\Pi}$-module $X_p$ called a \newword{string module}, and these exhaust the indecomposable $\overline{\Pi}$-modules, see \cite{WW}.  

\begin{theorem}\label{typeA} 
Suppose $\Pi$ is the preprojective algebra of type $A_n$.  
\begin{enumerate}[\rm(a)] 
\item The layer modules for $\Pi$ are exactly the indecomposable $\overline\Pi$-modules (which are exactly the string modules).
\item The doubleton extension order on layer modules is the opposite of subfactor order.
\end{enumerate}
\end{theorem} 

\begin{proof} (a) For each $i\in Q_0$, let $x_i\in\Pi$ be the 2-cycle at the vertex $i$. It follows from relations of $\Pi$ that $x:=\sum_{i\in Q_0}x_i$ is a central element, and $I_{\rm cyc}$ is generated by $x$.
If $X$ is a $\Pi$-module, then multiplication by $x$ is a 
surjective morphism onto $I_{\rm cyc}X \subset X$.  Therefore, if $X$ is a brick, $I_{\rm cyc}X$ must be
zero.  Conversely, it is easy to see that any string module for 
$\overline \Pi$ is a brick.  

(b) As mentioned above, we need only show that if $A$ is a subfactor of $B$, then $A\ge B$.
It suffices to consider the case when $A$ is a submodule or a quotient module of $B$ such that $\dim_kA=\dim_kB-1$.
In this case, it is clear from the shape of string modules that there is a doubleton $A,S$ for some simple $\Pi$-module $S$ such that $B$ is one of its two extensions. Thus $A\ge B$.
\end{proof}

\begin{example}\label{non-revisiting for A3}
Let $n=3$. Then the Hasse quiver of $\SS$ with respect to the opposite of
subfactor order is the following:  
\[\xymatrix{
1\ar[dr]\ar[d]&&2\ar[dr]\ar[dl]\ar[dll]\ar[drr]&&3\ar[d]\ar[dl]\\
{\begin{smallmatrix}1&\\ &2\end{smallmatrix}}\ar[dr]\ar[d]&
{\begin{smallmatrix}&2\\ 1&\end{smallmatrix}}\ar[drr]\ar[drrr]&&
{\begin{smallmatrix}2&\\ &3\end{smallmatrix}}\ar[d]\ar[dlll]&
{\begin{smallmatrix}&3\\ 2&\end{smallmatrix}}\ar[d]\ar[dlll]\\
{\begin{smallmatrix}1&&\\ &2&\\ &&3\end{smallmatrix}}&
{\begin{smallmatrix}1&&3\\ &2&\end{smallmatrix}}&&
{\begin{smallmatrix}&2&\\ 1&&3\end{smallmatrix}}&
{\begin{smallmatrix}&&3\\ &2&\\ 1&&\end{smallmatrix}}
}\]
In light of Theorems~\ref{isom} and~\ref{typeA}, this quiver is the Hasse quiver of the forcing order on join-irreducible elements in type $A_3$.
Compare \cite[Figure~4]{congruence}.
\end{example}

\section{Combinatorial description of indecomposable $\tau$-rigid modules}\label{explicit}

Let $\Pi$ be a preprojective algebra of Dynkin type and $W$ the corresponding Weyl group.
A combinatorial description of join-irreducible elements in $W$ is well-known for type $A$ and $D$.
We refer to Sections 1.5, 2.1, 2.4, and 2.6 in \cite{BB} for type A, and Section~8.2 in \cite{BB} for type D.
(The description of join-irreducible elements is not given explicitly in \cite{BB}, but it is easily worked out from the combinatorial models developed there.)
On the other hand, recall from Theorem \ref{big diagram} that we have a bijection
\[J:\jirr W\to\IndtmRig\Pi\]
given by $J(w):=(\Pi/I(w))e_i$ for a unique arrow $w\to ws_i$ in the Hasse quiver of $W$ starting at $w$.

The main result of this section is to give a combinatorial description of the $\Pi$-module $J(w)$ for each join-irreducible element $w\in W$ in type $A$ or $D$.
See Theorem \ref{define X(w) for A} for type $A$ and Theorems~\ref{define X(w) for D2} and~\ref{define X(w) for D1} for type $D$.
It will be interesting to compare our results with Bongartz's description of bricks for type $A$ and $D$ \cite{Bon}.

\subsection{Type $A$}
Let $\Pi$ be a preprojective algebra of type $A_n$. It is given by a quiver
\[\xymatrix@R=.3em@C=3em{1\ar@<2pt>[r]^{x_1}&2\ar@<2pt>[r]^{x_2}\ar@<2pt>[l]^{y_2}&3\ar@<2pt>[r]^{x_3}\ar@<2pt>[l]^{y_3}&\ar@{.}[r]\ar@<2pt>[l]^{y_4}&\ar@<2pt>[r]^-{x_{n-2}}&n-1\ar@<2pt>[r]^-{x_{n-1}}\ar@<2pt>[l]^-{y_{n-1}}&n\ar@<2pt>[l]^-{y_n}}\]
with relations $x_1y_2=0$, $x_iy_{i+1}=y_{i}x_{i-1}$ for $2\le i\le n-1$ and $y_nx_{n-1}=0$.
We denote by $S_\ell$ the simple $\Pi$-module corresponding to the vertex $\ell$, and by $P_\ell$ the projective cover of $S_\ell$.

Let $W=\mathfrak{S}_{n+1}$ be the Weyl group of $\Pi$.  We use the convention that the product $ww'$ of elements $w,w'\in W$ is given by $(ww')(i)=w(w'(i))$ for $i\in\{1,\ldots,n+1\}$. The elements of $W$ are the permutations
\[w=[i_1,\ldots,i_{n+1}].\]
This is join-irreducible if and only if there exists $\ell\in\{1,\ldots,n\}$ such that
\[i_1<\cdots<i_\ell>i_{\ell+1}<\cdots<i_{n+1}.  \]
In this case, we say that $w$ is of \emph{type $\ell$}. There exists a unique arrow $w\to ws_\ell$ starting at $w$ in the Hasse quiver of $W$.
The number of join-irreducible elements of type $\ell$ is given by
\[{n+1\choose\ell}-1,\]
and therefore $\#\jirr W=2^{n+1}-n-2$.

As a quiver representation, $P_\ell$ is given by 
\[{\tiny\xymatrix@R1em@C2.5em{
\ell\ar[r]\ar[d]&\ell{-}1\ar[r]\ar[d]&\ell{-}2\ar[r]\ar[d]&\cdots\ar[r]&3\ar[r]\ar[d]&2\ar[r]\ar[d]&1\ar[d]\\
\ell{+}1\ar[r]\ar[d]&\ell\ar[r]\ar[d]&\ell{-}1\ar[r]\ar[d]&\cdots\ar[r]&4\ar[r]\ar[d]&3\ar[r]\ar[d]&2\ar[d]\\
\ell{+}2\ar[r]\ar[d]&\ell{+}1\ar[r]\ar[d]&\ell\ar[r]\ar[d]&\cdots\ar[r]&5\ar[r]\ar[d]&4\ar[r]\ar[d]&3\ar[d]\\
\raisebox{-2pt}[7pt][4pt]{\vdots}\ar[d]&\raisebox{-2pt}[7pt][4pt]{\vdots}\ar[d]&\raisebox{-2pt}[7pt][4pt]{\vdots}\ar[d]&\cdots&\raisebox{-2pt}[7pt][4pt]{\vdots}\ar[d]&\raisebox{-2pt}[7pt][4pt]{\vdots}\ar[d]&\raisebox{-2pt}[7pt][4pt]{\vdots}\ar[d]\\
n{-}1\ar[r]\ar[d]&n{-}2\ar[r]\ar[d]&n{-}3\ar[r]\ar[d]&\cdots\ar[r]&n{-}\ell{+}2\ar[r]\ar[d]&n{-}\ell{+}1\ar[r]\ar[d]&n{-}\ell\ar[d]\\
n\ar[r]&n{-}1\ar[r]&n{-}2\ar[r]&\cdots\ar[r]&n{-}\ell{+}3\ar[r]&n{-}\ell{+}2\ar[r]&n{-}\ell{+}1,
}}\]
where each number $i$ shows a $k$-vector space $k$ lying on the vertex $i$, and each arrow is the identity map of $k$.
Submodules (respectively, factor modules) of $P_\ell$ correspond bijectively to subquivers that are closed under successors (respectively, predecessors).

We represent $P_\ell$ in abbreviated form as an array of numbers in rows as follows:
\[P_\ell={\tiny\begin{array}{|ccccccc|}
\hline
\ell&\ell{-}1&\ell{-}2&\cdots&3&2&1\\ \hline
\ell{+}1&\ell&\ell{-}1&\cdots&4&3&2\\ \hline
\ell{+}2&\ell{+}1&\ell&\cdots&5&4&3\\ \hline
\vdots&\vdots&\vdots&\ddots&\vdots&\vdots&\vdots\\ \hline
n{-}1&n{-}2&n{-}3&\cdots&n{-}\ell{+}2&n{-}\ell{+}1&n{-}\ell\\ \hline
n&n{-}1&n{-}2&\cdots&n{-}\ell{+}3&n{-}\ell{+}2&n{-}\ell{+}1\\ \hline
\end{array}}\]
Submodules and factor modules are similarly represented by sub-arrays of $P_\ell$, as, for example, in the following theorem.

\begin{theorem}\label{define X(w) for A}
Let  $w=[i_1,\ldots,i_{n+1}]\in W$ be a join-irreducible element of type $\ell$.
Then $J(w)$ is a factor module of $P_\ell$ which has the form
\[J(w)={\tiny\begin{array}{|cccccccccc|}
\hline
\ell&\ell{-}1&\ell{-}2&\cdots&\cdots&\cdots&\cdots&i_{\ell{+}1}&&\\ \hline
\ell{+}1&\ell&\ell{-}1&\cdots&\cdots&\cdots&i_{\ell{+}2}&&&\\ \hline
\ell{+}2&\ell{+}1&\ell&\cdots&\cdots&i_{\ell{+}3}&&&&\\ \hline
\vdots&\vdots&\vdots&\vdots&\vdots&&&&&\\ \hline
n{-}1&\cdots&\cdots&i_{n}&&&&&&\\ \hline
n&\cdots&i_{n{+}1}&&&&&&&\\ \hline
\end{array}.}\]
In particular, any factor module of $P_\ell$ is indecomposable $\tau^-$-rigid.
\end{theorem}

Note that $i_m\le m$ holds for any $m\in\{\ell+1,\ldots,n+1\}$, and if $i_{m}=m$ holds, then the row starting at $m-1$ is empty.
For example, for $n=5$ and $w=[351246]$, then $\ell=2$ and
$J(w)={\tiny\begin{array}{|cc|}\hline 2&1\\ \hline 3&2\\ \hline 4&\\ \hline &\\ \hline\end{array}}=
{\tiny\begin{xy}
0;<3pt,0pt>:<0pt,3pt>::
(15,-5)*+{4}="44",
(15,0)*+{3}="45",
(15,5)*+{2}="46",
(20,0)*+{2}="55",
(20,5)*+{1}="56",
\ar@<0pt>"45";"44",
\ar@<0pt>"46";"45",
\ar@<0pt>"56";"55",

\ar@<0pt>"45";"55",
\ar@<0pt>"46";"56",
\end{xy}}$.

To prove this, we need the following easy observation.

\begin{lemma}\label{expression for A}
Let  $w=[i_1,\ldots,i_{n+1}]\in W$ be a join-irreducible element of type $\ell$.
Then we have a reduced expression $w=x_{n+1}x_{n}\cdots x_{\ell+2}x_{\ell+1}$,
where $x_m=s_{i_{m}}s_{i_{m}+1}\cdots s_{m-2}s_{m-1}$ for $m\in\{\ell+1,\ldots,n+1\}$.
\end{lemma}

\begin{proof}
If $\ell=n$, then $w=[1,2,\cdots,i-1,i+1,\ldots,n+1,i]$ holds for some $i\in\{1,\ldots,n\}$. Thus $w=x_{n+1}$ holds clearly.
In the rest, assume $\ell\neq n$, and let $v:=x_{n+1}^{-1}w$. 
Since $x_{n+1}^{-1}$ is a cyclic permutation $(n+1,n,\ldots,i_{n+1}+1,i_{n+1})$, it preserves the total order on $\{i_1,\ldots,i_n\}$, and therefore $v$ is a join-irreducible element of type $\ell$ and clearly satisfies $v(n+1)=n+1$.
Since $x_{n+1}^{-1}$ fixes any element in $\{i_{\ell+1},\ldots,i_n\}$, we have $v(m)=i_m$ for any $m\in\{\ell+1,\ldots,n\}$ and $v(n+1)=n+1$.
Inductively on $n$, we have $v=x_{n}\cdots x_{\ell+1}$, and therefore $w=x_{n+1}x_{n}\cdots x_{\ell+1}$.
\end{proof}

Now we are ready to prove Theorem \ref{define X(w) for A}.

\begin{proof}[Proof of Theorem \ref{define X(w) for A}]
Clearly we have
\[I(x_{\ell+1})e_\ell=I_{i_{\ell+1}}\cdots I_{\ell}e_\ell={\tiny\begin{array}{|ccccccc|}
\hline
&&&&i_{\ell{+}1}{+}1&\cdots&1\\ \hline
\ell{+}1&\ell&\cdots&\cdots&\cdots&\cdots&2\\ \hline
\ell{+}2&\ell{+}1&\cdots&\cdots&\cdots&\cdots&3\\ \hline
\vdots&\vdots&\vdots&\vdots&\vdots&\vdots&\vdots\\ \hline
n&n{-}1&\cdots&\cdots&\cdots&\cdots&n{-}\ell{+}1\\ \hline
\end{array}}\]
and
\[I(x_{\ell+2}x_{\ell+1})e_\ell={\tiny\begin{array}{|ccccccc|}
\hline
&&&&i_{\ell{+}1}{+}1&\cdots&1\\ \hline
&&&i_{\ell{+}2}{+}1&\cdots&\cdots&2\\ \hline
\ell{+}2&\ell{+}1&\cdots&\cdots&\cdots&\cdots&3\\ \hline
\vdots&\vdots&\vdots&\vdots&\vdots&\vdots&\vdots\\ \hline
n&n{-}1&\cdots&\cdots&\cdots&\cdots&n{-}\ell{+}1\\ \hline
\end{array}.}\]
Repeating similar calculations, we have
\[I(w)e_\ell={\tiny\begin{array}{|cccccccc|}
\hline
&&&&&i_{\ell{+}1}{+}1&\cdots&1\\ \hline
&&&&i_{\ell{+}2}{+}1&\cdots&\cdots&2\\ \hline
&&&i_{\ell{+}3}{+}1&\cdots&\cdots&\cdots&3\\ \hline
&&\vdots&\vdots&\vdots&\vdots&\vdots&\vdots\\ \hline
&i_{n{+}1}{+}1&\cdots&\cdots&\cdots&\cdots&\cdots&n{-}\ell{+}1\\ \hline
\end{array}.}\]
Therefore $J(w)=(\Pi/I(w))e_{\ell}$ has the desired form.
\end{proof}

\begin{example}
Let $w=[n-\ell+2,n-\ell+3,\ldots,n+1,1,2,\ldots,n-\ell+1]$ be a join-irreducible element of type $\ell$.
Then we have a reduced expression
\[w=x_{n+1}x_{n}\cdots x_{\ell+2}x_{\ell+1}\ \mbox{ for }\ x_m=s_{m-\ell}\cdots s_{m-2}s_{m-1},\]
and the corresponding indecomposable $\tau^-$-rigid $\Pi$-module is $J(w)=P_\ell$.
\end{example}

\begin{example}
Consider type $A_3$. Below we show the expression given in Lemma \ref{expression for A} by using dots: $w=x_{n+1}\cdot x_{n}\cdot\cdots\cdot x_{\ell+2}\cdot x_{\ell+1}$.

We have the following 3 join-irreducible elements of type $1$.

\noindent
$J(2134)=J(s_1)={\tiny\begin{array}{|c|}\hline 1\\ \hline\end{array}}$,\ \ \ 
$J(3124)=J(s_2\cdot s_1)={\tiny\begin{array}{|c|}\hline 1\\ \hline 2\\ \hline\end{array}}$,\ \ \ 
$J(4123)=J(s_3\cdot s_2\cdot s_1)={\tiny\begin{array}{|c|}\hline 1\\ \hline 2\\ \hline 3\\ \hline\end{array}}$.

We have the following 5 join-irreducible elements of type $2$.

\noindent
$J(1324)=J(s_2)={\tiny\begin{array}{|c|}\hline 2\\ \hline\end{array}}$,\ \ \ 
$J(2314)=J(s_1s_2)={\tiny\begin{array}{|cc|}\hline 2&1\\ \hline\end{array}}$,\ \ \ 
$J(1423)=J(s_3\cdot s_2)={\tiny\begin{array}{|c|}\hline 2\\ \hline 3\\ \hline\end{array}}$,

\noindent
$J(2413)=J(s_3\cdot s_1s_2)={\tiny\begin{array}{|cc|}\hline 2&1\\ \hline 3&\\ \hline\end{array}}$,\ \ \ 
$J(3412)=J(s_2s_3\cdot s_1s_2)={\tiny\begin{array}{|cc|}\hline 2&1\\ \hline 3&2\\ \hline\end{array}}$.

We have the following 3 join-irreducible elements of type $3$.

\noindent
$J(1243)=J(s_3)={\tiny\begin{array}{|c|}\hline 3\\ \hline\end{array}}$,\ \ \ 
$J(1342)=J(s_2s_3)={\tiny\begin{array}{|cc|}\hline 3&2\\ \hline\end{array}}$,

\noindent
$J(2341)=J(s_1s_2s_3)={\tiny\begin{array}{|ccc|}\hline 3&2&1\\ \hline\end{array}}$.
\end{example}

\subsection{Type $D$}

Let $\Pi$ be a preprojective algebra of type $D_n$.
Then $\Pi$ is given by a quiver\\[-15pt]
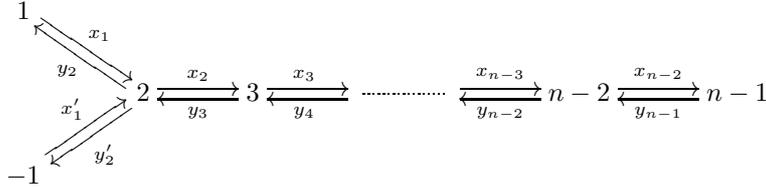
\[\xymatrix@R=0.6em@C=3em{1\ar@<2pt>[rdd]^-{x_1}\\\\
&2\ar@<2pt>[r]^{x_2}\ar@<2pt>[ldd]^-{y'_2}\ar@<2pt>[luu]^-{y_2}&3\ar@<2pt>[r]^{x_3}\ar@<2pt>[l]^{y_3}&\ar@{..}[r]\ar@<2pt>[l]^{y_4}&\ar@<2pt>[r]^-{x_{n-3}}&n-2\ar@<2pt>[r]^-{x_{n-2}}\ar@<2pt>[l]^-{y_{n-2}}&n-1\ar@<2pt>[l]^-{y_{n-1}}\\\\
-1\ar@<2pt>[ruu]^-{x'_1}}\]
with relations $x_1y_2=0$, $x'_1y'_2=0$, $x_2y_3=y_2x_1+y'_2x'_1$, $x_iy_{i+1}=y_ix_{i-1}$ for $3\le i\le n-2$ and $y_{n-1}x_{n-2}=0$.
Let $S_\ell$ be the simple $\Pi$-module corresponding to the vertex $\ell$, and let $P_\ell$ be the projective cover of $S_\ell$.

The Weyl group $W$ of $\Pi$ is the group of automorphisms $w$ of the set $\{\pm1,\ldots,\pm n\}$ satisfying $w(-\ell)=-w(\ell)$ for any $\ell$ and $\#\{\ell\in\{1,\ldots,n\}\mid w(\ell)<0\}$ is even.
Setting $i_\ell=w(\ell)$ for $\ell\in\{1,\ldots,n\}$, we write $w\in W$ as a sequence
\[w=[i_1,\ldots,i_n]\in\{\pm1,\ldots,\pm n\}^n\]
satisfying the following two conditions.
\begin{itemize}
\item $|i_1|,\ldots,|i_n|$ is a permutation of $1,\ldots,n$.
\item The number of negative integers is even.
\end{itemize}
We often denote $-i$ by $\underline{i}$. The simple reflections in $W$ are given by
\[s_{-1}=[\underline{2},\underline{1},3,4,\ldots,n]\ \mbox{ and }\ s_\ell=[1,\ldots,\ell-1,\ell+1,\ell,\ell+2,\ldots,n]\]
with $\ell\in\{1,\ldots,n-1\}$. The length of $w=[i_1,\ldots,i_n]$ is given by
\[\ell(w)=\#\{1\le \ell<m\le n\mid i_\ell>i_{m}\}+\#\{1\le \ell<m\le n\mid -i_\ell>i_{m}\}.\]
For $\ell\in\{1,\ldots,n-1\}$, $\ell(ws_\ell)-\ell(w)$ is $1$ if $i_\ell<i_{\ell+1}$, and $-1$ otherwise.
Moreover $\ell(ws_{-1})-\ell(w)$ is $1$ if $-i_1<i_2$, and $-1$ otherwise.
Therefore an element $w=[i_1,\ldots,i_n]\in W$ is join-irreducible if and only if one of the following conditions is satisfied. \begin{itemize}
\item $i_1<\cdots<i_n$ and $-i_1>i_2$.
\item There exists $\ell\in\{1,\ldots,n-1\}$ such that $i_1<\cdots<i_\ell>i_{\ell+1}<\cdots<i_n$ and $-i_1<i_2$.
\end{itemize}
We say that $w$ is of \emph{type $-1$} (respectively, \emph{type $\ell$}) if the first (respectively, second) condition is satisfied. 
Then there exists a unique arrow $w\to ws_\ell$ starting at $w$ in the Hasse quiver of $W$.
The number of join-irreducible elements of type $\ell\neq\pm1$ (respectively, $1$, $-1$) is
\[2^{n-\ell}{n\choose\ell}-1\ \ \ (\mbox{respectively, $2^{n-1}-1$, $2^{n-1}-1$}),\]
and therefore $\#\jirr W=3^n-n2^{n-1}-n-1$.

As a quiver representation, $P_\ell$ for $\ell=\pm1$ is given by
\[{\tiny\xymatrix@R1em@C2.5em{
\pm1\ar[d]&&&&&\\
2\ar[r]\ar[d]&\mp1\ar[d]&&&&\\
3\ar[r]\ar[d]&2\ar[r]\ar[d]&\pm1\ar[d]&&&\\
\raisebox{-2pt}[8pt][3pt]{\vdots}\ar[d]&\raisebox{-2pt}[8pt][3pt]{\vdots}\ar[d]&\raisebox{-2pt}[8pt][3pt]{\vdots}\ar[d]&\raisebox{-2pt}{$\ddots$}&&\\
n{-}3\ar[r]\ar[d]&n{-}4\ar[r]\ar[d]&n{-}5\ar[r]\ar[d]&\cdots\ar[r]&\pm({-}1)^{n}\ar[d]&\\
n{-}2\ar[r]\ar[d]&n{-}3\ar[r]\ar[d]&n{-}4\ar[r]\ar[d]&\cdots\ar[r]&2\ar[r]\ar[d]&\mp({-}1)^{n}\ar[d]&\\
n{-}1\ar[r]&n{-}2\ar[r]&n{-}3\ar[r]&\cdots\ar[r]&3\ar[r]&2\ar[r]&\pm({-}1)^n
}}
\]
where each number $i$ shows a $k$-vector space $k$ lying on the vertex $i$, and each arrow is the identity map of $k$.
Again, submodules (respectively, factor modules) of $P_\ell$ correspond bijectively to subquivers that are closed under successors (respectively, predecessors).

We represent $P_\ell$ in abbreviated form as an array of numbers in rows as follows:
\[P_\ell={\tiny\begin{array}{|ccccccc|}
\hline
\pm1&&&&&&\\ \hline
2&\mp1&&&&&\\ \hline
3&2&\pm1&&&&\\ \hline
\vdots&\vdots&\vdots&\ddots&&&\\ \hline
n{-}3&n{-}4&n{-}5&\cdots&\pm(-1)^{n}&&\\ \hline
n{-}2&n{-}3&n{-}4&\cdots&2&\mp(-1)^{n}&\\ \hline
n{-}1&n{-}2&n{-}3&\cdots&3&2&\pm(-1)^n\\ \hline
\end{array}.}\]
Submodules and factor modules of $P_\ell$ are again represented by subarrays.

\begin{theorem}\label{define X(w) for D2}
Let $w=[i_1,\ldots,i_n]\in W$ be a join-irreducible element of type $\ell=\pm1$.
Then $J(w)$ is a factor module of $P_\ell$ which has the form
\[J(w)={\tiny\begin{array}{|cccccccc|}
\hline
i'_2&&&&&&&\\ \hline
2&i'_3&&&&&&\\ \hline
3&\cdots&i'_4&&&&&\\ \hline
\vdots&\vdots&\vdots&\ddots&&&&\\ \hline
n{-}2&n{-}3&\cdots&\cdots&i'_{n{-}1}&&&\\ \hline
n{-}1&n{-}2&\cdots&i'_n&&&&\\ \hline
\end{array}.}\]
where $i'_m:=\max\{i_m,(-1)^m\ell\}$ for $m\in\{2,\ldots,n\}$.
\end{theorem}

Note that $i'_m\le m$ holds for any $m\in\{2,\ldots,n\}$, and if $i'_{m}=m$ holds, then the row starting at $m-1$ is empty.
For example, for $n=6$ and $w=[5\underline{3}\underline{1}246]$, then $\ell=1$ and
$J(w)={\tiny\begin{array}{|ccccc|}\hline 1&&&&\\ \hline 2&-1&&&\\ \hline 3&2&&&\\ \hline 4&&&&\\ \hline &&&&\\ \hline\end{array}}={\tiny\begin{xy}
0;<4pt,0pt>:<0pt,3pt>::
(15,-7.5)*+{4}="44",
(15,-2.5)*+{3}="45",
(15,2.5)*+{2}="46",
(15,7.5)*+{1}="47",
(20,-2.5)*+{2}="55",
(20,2.5)*+{-1}="56",
\ar@<0pt>"45";"44",
\ar@<0pt>"46";"45",
\ar@<0pt>"47";"46",
\ar@<0pt>"56";"55",

\ar@<0pt>"45";"55",
\ar@<0pt>"46";"56",
\end{xy}}$.

To prove this, we need the following observation.

\begin{lemma}\label{expression for D2}
Let  $w=[i_1,\ldots,i_{n}]\in W$ be a join-irreducible element of type $\ell=\pm1$.
Then we have a reduced expression $w=x_{n}x_{n-1}\cdots x_3x_2$, where
\[x_m=\left\{\begin{array}{ll}
s_{i_{m}}s_{i_{m}+1}\cdots s_{m-2}s_{m-1}&\mbox{if $i_{m}\ge2$}\\
s_{(-1)^m\ell}s_{2}\cdots s_{m-2}s_{m-1}&\mbox{if $i_{m}<2$}
\end{array}\right.\ \mbox{ for }\ m\in\{2,\ldots,n\}.\]
\end{lemma}

\begin{proof}
The case $n=2$ can be checked directly. In the rest, we assume $n\ge3$.
Assume that the assertion holds for $n-1$.
Again,  let $v:=x_{n}^{-1}w$.

Assume $i_n>0$. 
Then $v$ is obtained from $w$ by replacing $i_n$ with $n$ and then, for each $i_j$ with $i_n<|i_j|$ replacing $i_j$ by an entry with the same sign but absolute value $|i_j|-1$.
In particular, $v$ is a join-irreducible element of type $\ell$ and satisfies $v(n)=n$.
Fix $m\in\{2,\ldots,n-1\}$. If $i_m>0$, then $v(m)=i_m$ holds, and if $i_m<0$, then $v(m)<0$ holds.
By our assumption on induction, we have $v=x_{n-1}\cdots x_2$, and therefore $w=x_nx_{n-1}\cdots x_2$.

Assume $i_n<0$. In this case, we have $w=[(-1)^{n-1}n,\underline{n-1},\underline{n-2},\ldots,\underline{1}]$ and $\ell=(-1)^{n-1}$. Thus $x_{n}^{-1}=[\underline{n},\underline{1},2,\ldots,n-1]$ and $v=[(-1)^{n-1}(n-1),\underline{n-2},\ldots,\underline{2},1,n]$ hold.
Therefore $v$ is a join-irreducible element of type $\ell$ and satisfies $v(n)=n$.
Since $v(m)<2$ holds for any $m\in\{2,\ldots,n-1\}$, we have $v=x_{n-1}\cdots x_2$ by our assumption on induction, and therefore $w=x_nx_{n-1}\cdots x_2$.
\end{proof}

Now we are ready to prove Theorem \ref{define X(w) for D2}.

\begin{proof}[Proof of Theorem \ref{define X(w) for D2}]
Using Lemma \ref{expression for D2}, one can calculate $I(w)$ as in the proof of Theorem \ref{define X(w) for A}, and we obtain the desired assertion.
\end{proof}

\begin{example}
Let $w=[n,\underline{n-1},\underline{n-2},\ldots,\underline{2},(-1)^n]$ be a join-irreducible element of type $1$.
Then we have a reduced expression
\[w=x_nx_{n-1}\cdots x_{3}x_{2}\ \mbox{ for }\ x_m=s_{(-1)^m}s_2s_3\cdots s_{m-1},\]
and the corresponding indecomposable $\tau^-$-rigid $\Pi$-module is $J(w)=P_1$.

Similarly, let $w=[\underline{n},\underline{n-1},\underline{n-2},\ldots,\underline{2},(-1)^{n+1}]$ be a join-irreducible element of type $-1$. Then we have a reduced expression
\[w=x_nx_{n-1}\cdots x_{3}x_{2}\ \mbox{ for }\ x_m=s_{(-1)^{m+1}}s_2s_3\cdots s_{m-1},\]
and the corresponding indecomposable $\tau^-$-rigid $\Pi$-module is $J(w)=P_{-1}$.
\end{example}

\begin{example}
Consider type $D_4$. Below we show the expression given in Lemma \ref{expression for D2} by using dots: $w=x_{n}\cdot x_{n-1}\cdot\cdots\cdot x_{3}\cdot x_{2}$.

We have the following 7 join-irreducible elements of type $1$.

\noindent
$J(2134)=J(s_1)={\tiny\begin{array}{|c|}\hline 1\\ \hline\end{array}}$,\ \ \ 
$J(3124)=J(s_2\cdot s_1)={\tiny\begin{array}{|c|}\hline 1\\ \hline 2\\ \hline\end{array}}$,

\noindent
$J(3\underline{2}\underline{1}4)=J(s_{-1}s_2\cdot s_1)={\tiny\begin{array}{|cc|}\hline 1&\\ \hline 2&-1\\ \hline\end{array}}$,\ \ \ $J(4123)=J(s_3\cdot s_2\cdot s_1)={\tiny\begin{array}{|c|}\hline 1\\ \hline 2\\ \hline 3\\ \hline\end{array}}$,

\noindent
$J(4\underline{2}\underline{1}3)=J(s_3\cdot s_{-1}s_2\cdot s_1)={\tiny\begin{array}{|cc|}\hline 
1&\\ \hline 2&-1\\ \hline 3&\\ \hline\end{array}}$,\ \ \ 
$J(4\underline{3}\underline{1}2)=J(s_2s_3\cdot s_{-1}s_2\cdot s_1)={\tiny\begin{array}{|cc|}\hline 1&\\ \hline 2&-1\\ \hline 3&2\\ \hline\end{array}}$,

\noindent
$J(4\underline{3}\underline{2}1)=J(s_1s_2s_3\cdot s_{-1}s_2\cdot s_1)={\tiny\begin{array}{|ccc|}\hline 1&&\\ \hline 2&-1&\\ \hline 3&2&1\\ \hline\end{array}}$.

We have the following 7 join-irreducible elements of type $-1$.

\noindent
$J(\underline{2}\underline{1}34)=J(s_{-1})={\tiny\begin{array}{|c|}\hline -1\\ \hline\end{array}}$,\ \ \ 
$J(\underline{3}\underline{1}24)=J(s_2\cdot s_{-1})={\tiny\begin{array}{|c|}\hline -1\\ \hline 2\\ \hline\end{array}}$,

\noindent
$J(\underline{3}\underline{2}14)=J(s_1s_2\cdot s_{-1})={\tiny\begin{array}{|cc|}\hline -1&\\ \hline2&1\\ \hline\end{array}}$,\ \ \ $J(\underline{4}\underline{1}23)=J(s_3\cdot s_2\cdot s_{-1})={\tiny\begin{array}{|c|}\hline -1\\ \hline 2\\ \hline 3\\ \hline\end{array}}$,

\noindent
$J(\underline{4}\underline{2}13)=J(s_3\cdot s_1s_2\cdot s_{-1})={\tiny\begin{array}{|cc|}\hline -1&\\ \hline 2&1\\ \hline 3&\\ \hline\end{array}}$,\ \ \ 
$J(\underline{4}\underline{3}12)=J(s_2s_3\cdot s_1s_2\cdot s_{-1})={\tiny\begin{array}{|cc|}\hline -1&\\ \hline 2&1\\ \hline 3&2\\ \hline\end{array}}$,

\noindent
$J(\underline{4}\underline{3}\underline{2}\underline{1})=J(s_{-1}s_2s_3\cdot s_1s_2\cdot s_{-1})={\tiny\begin{array}{|ccc|}\hline -1&&\\ \hline 2&1&\\ \hline 3&2&-1\\ \hline\end{array}}$.
\end{example}

In the rest of this section, let $w$ be a join-irreducible element $i_1<\cdots<i_\ell>i_{\ell+1}<\cdots<i_n$ of type $\ell\neq\pm1$.
Note that all integers in $i_2,\ldots,i_\ell$ must be positive, and therefore $w$ can be recovered from the latter part $i_{\ell+1},\ldots,i_n$.

We need the following preparation on the structure of $P_\ell$.

\begin{lemma}
Let $\alpha$ and $\beta$ be scalars satisfying $\alpha+\beta=1$.
As a quiver representation, $P_\ell$ with $\ell\neq\pm1$ is given by Figure \ref{Pell}
where each number $i$ shows a $k$-vector space $k$ lying on the vertex $i$ if $i\ge-1$ and $-i$ if $i\le-2$.
Each unlabelled arrow is the identity map of $k$, and
each arrow labelled by a scalar $\gamma$ is a linear map multiplying by $\gamma$.
\end{lemma}

\begin{proof}
Since all relations of $\Pi$ are satisfied, this gives a $\Pi$-module $X$.
It is easy to check that $X$ is generated by $\ell$ in the upper left corner.
Thus we have a surjective morphism $\pi:P_\ell\to X$ of $\Pi$-modules.
On the other hand, it is well-known that the Loewy length of any indecomposable projective $\Pi$-module is equal to $h-1$,
where $h$ is the Coxeter number. For type $D_n$, we have $h-1=2n-3$.
Since the length of the path from $\ell$ in the left corner to $-\ell$ in the lower right corner is $2n-4$,
the above $\pi$ must be an isomorphism.
\end{proof}

\begin{example}
Let $n=6$. Then $P_2$ is given by the following quiver representations,
where the left one is the case $(\alpha,\beta)=(0,1)$, and the right one is the case $(\alpha,\beta)=(1,0)$.
\begin{equation*}
{\tiny\begin{xy}
0;<6pt,0pt>:<0pt,5pt>::
(15,10)*+{5}="43",
(15,15)*+{4}="44",
(15,20)*+{3}="45",
(15,25)*+{2}="46",
(20,10)*+{4}="53",
(20,15)*+{3}="54",
(20,20)*+{2}="55",
(20,26)*+{1}="56",
(20,24)*+{{-}1}="56-",
(25,10)*+{3}="63",
(25,15)*+{2}="64",
(25,21)*+{{-}1}="65",
(25,19)*+{1}="65-",
(25,25)*+{{-}2}="66",
(30,10)*+{2}="73",
(30,16)*+{1}="74",
(30,14)*+{{-}1}="74-",
(30,20)*+{{-}2}="75",
(30,25)*+{{-}3}="76",
(35,11)*+{{-}1}="83-",
(35,9)*+{1}="83",
(35,15)*+{{-}2}="84",
(35,20)*+{{-}3}="85",
(35,25)*+{{-}4}="86",
(40,10)*+{{-}2}="93",
(40,15)*+{{-}3}="94",
(40,20)*+{{-}4}="95",
(40,25)*+{{-}5}="96",
\ar@<0pt>"44";"43",
\ar@<0pt>"45";"44",
\ar@<0pt>"46";"45",
\ar@<0pt>"54";"53",
\ar@<0pt>"55";"54",
\ar@<0pt>"56-";"55",
\ar@<0pt>"64";"63",
\ar@<0pt>"65-";"64",
\ar@<0pt>"66";"65",
\ar@<0pt>"74-";"73",
\ar@<0pt>"75";"74",
\ar@<0pt>"76";"75",
\ar@<0pt>"84";"83-",
\ar@<0pt>"85";"84",
\ar@<0pt>"86";"85",
\ar@<0pt>"94";"93",
\ar@<0pt>"95";"94",
\ar@<0pt>"96";"95",

\ar@<0pt>"43";"53",
\ar@<0pt>"44";"54",
\ar@<0pt>"45";"55",
\ar@<0pt>"46";"56",
\ar@<0pt>"46";"56-",
\ar@<0pt>"53";"63",
\ar@<0pt>"54";"64",
\ar@<0pt>"55";"65",
\ar@<0pt>"55";"65-",
\ar@<0pt>"56";"66",
\ar@<0pt>_{-1}"56-";"66",
\ar@<0pt>"63";"73",
\ar@<0pt>"64";"74",
\ar@<0pt>"64";"74-",
\ar@<0pt>"65";"75",
\ar@<0pt>_{-1}"65-";"75",
\ar@<0pt>"66";"76",
\ar@<0pt>"73";"83",
\ar@<0pt>"73";"83-",
\ar@<0pt>"74";"84",
\ar@<0pt>_{-1}"74-";"84",
\ar@<0pt>"75";"85",
\ar@<0pt>"76";"86",
\ar@<0pt>_{-1}"83";"93",
\ar@<0pt>"83-";"93",
\ar@<0pt>"84";"94",
\ar@<0pt>"85";"95",
\ar@<0pt>"86";"96",
\end{xy}
\ \ \ \ \ \begin{xy}
0;<6pt,0pt>:<0pt,5pt>::
(15,10)*+{5}="43",
(15,15)*+{4}="44",
(15,20)*+{3}="45",
(15,25)*+{2}="46",
(20,10)*+{4}="53",
(20,15)*+{3}="54",
(20,20)*+{2}="55",
(20,26)*+{{-}1}="56",
(20,24)*+{1}="56-",
(25,10)*+{3}="63",
(25,15)*+{2}="64",
(25,21)*+{1}="65",
(25,19)*+{{-}1}="65-",
(25,25)*+{{-}2}="66",
(30,10)*+{2}="73",
(30,16)*+{{-}1}="74",
(30,14)*+{1}="74-",
(30,20)*+{{-}2}="75",
(30,25)*+{{-}3}="76",
(35,11)*+{1}="83-",
(35,9)*+{{-}1}="83",
(35,15)*+{{-}2}="84",
(35,20)*+{{-}3}="85",
(35,25)*+{{-}4}="86",
(40,10)*+{{-}2}="93",
(40,15)*+{{-}3}="94",
(40,20)*+{{-}4}="95",
(40,25)*+{{-}5}="96",
\ar@<0pt>"44";"43",
\ar@<0pt>"45";"44",
\ar@<0pt>"46";"45",
\ar@<0pt>"54";"53",
\ar@<0pt>"55";"54",
\ar@<0pt>"56-";"55",
\ar@<0pt>"64";"63",
\ar@<0pt>"65-";"64",
\ar@<0pt>^{-1}"66";"65",
\ar@<0pt>"74-";"73",
\ar@<0pt>^{-1}"75";"74",
\ar@<0pt>"76";"75",
\ar@<0pt>^{-1}"84";"83-",
\ar@<0pt>"85";"84",
\ar@<0pt>"86";"85",
\ar@<0pt>"94";"93",
\ar@<0pt>"95";"94",
\ar@<0pt>"96";"95",

\ar@<0pt>"43";"53",
\ar@<0pt>"44";"54",
\ar@<0pt>"45";"55",
\ar@<0pt>"46";"56",
\ar@<0pt>"46";"56-",
\ar@<0pt>"53";"63",
\ar@<0pt>"54";"64",
\ar@<0pt>"55";"65",
\ar@<0pt>"55";"65-",
\ar@<0pt>^{-1}"56";"66",
\ar@<0pt>"56-";"66",
\ar@<0pt>"63";"73",
\ar@<0pt>"64";"74",
\ar@<0pt>"64";"74-",
\ar@<0pt>^{-1}"65";"75",
\ar@<0pt>"65-";"75",
\ar@<0pt>"66";"76",
\ar@<0pt>"73";"83",
\ar@<0pt>"73";"83-",
\ar@<0pt>^{-1}"74";"84",
\ar@<0pt>"74-";"84",
\ar@<0pt>"75";"85",
\ar@<0pt>"76";"86",
\ar@<0pt>"83";"93",
\ar@<0pt>^{-1}"83-";"93",
\ar@<0pt>"84";"94",
\ar@<0pt>"85";"95",
\ar@<0pt>"86";"96",
\end{xy}}\end{equation*}
\end{example}

We write the quiver $P_\ell$ of Figure~\ref{Pell} in abbreviated form as the following array of numbers.
\begin{equation}\label{Pell array}{\tiny\begin{array}{|ccccccccccccc|}
\hline
\ell&\ell{-}1&\cdots&2&{\begin{smallmatrix}1\\ {-}1\end{smallmatrix}}&{-}2&\cdots&1{-}\ell&{-}\ell&{-}\ell{-}1&\cdots&2{-}n&1{-}n\\ \hline
\ell{+}1&\ell&\cdots&3&2&{\begin{smallmatrix}1\\ {-}1\end{smallmatrix}}&\cdots&2{-}\ell&1{-}\ell&{-}\ell&\cdots&3{-}n&2{-}n\\ \hline
\vdots&\vdots&\ddots&\ddots&\ddots&\ddots&\ddots&\ddots&\ddots&\ddots&
\ddots&\vdots&\vdots\\ \hline
n{-}2&n{-}3&\cdots&\ell&\ell{-}1&\ell{-}2&\cdots&{\begin{smallmatrix}1\\ {-}1\end{smallmatrix}}&{-}2&{-}3&\cdots&{-}\ell&{-}\ell{-}1\\ \hline
n{-}1&n{-}2&\cdots&\ell{+}1&\ell&\ell{-}1&\cdots&2&{\begin{smallmatrix}1\\ {-}1\end{smallmatrix}}&{-}2&\cdots&1{-}\ell&{-}\ell\\ \hline
\end{array}.}
\end{equation}

The description of factor $P_{\pm1}$ was no more complicated than the analogous description in type $A$.
However, for $\ell>1$, the description of factor modules of $P_\ell$ is much more complicated in type $D$.
For example, consider the direct sum $k^2$ of two copies of $k$ corresponding to $-2$ in the first row and $2$ in the second row.
Then subspaces of $k^2$ generate distinct submodules of $P_\ell$.
Fortunately, to describe $J(w)$ for $w$ join-irreducible, we only need the following special class of factor modules.

\begin{definition}\label{successor-closed}
Let $S$ be a subarray (that is, an arbitrary subset of the entries) of the array \eqref{Pell array}.
We say that $S$ is \newword{predecessor-closed} if it is closed under predecessors in the quiver Figure~\ref{Pell}.

Now we fix scalars $\alpha$ and $\beta$ satisfying $\alpha+\beta=1$.
We say that $S$ is
\newword{$(\alpha,\beta)$-predecessor-closed} if it is closed under predecessors in the subquiver of Figure~\ref{Pell} obtained by removing all arrows indexed by the scalar $0$.
An $(\alpha,\beta)$-predecessor-closed subarray $S$ gives a factor module of $P_\ell$ in a natural way.
Clearly, if $(\alpha,\beta)\neq(1,0),(0,1)$, then $S$ is $(\alpha,\beta)$-predecessor-closed if and only if it is predecessor-closed.
\end{definition}

For $m\in\{2,\ldots,n-1\}$, let $C(m,j)$ be the following subset of numbers in the row of \eqref{Pell array} starting at $m$.
\[{\footnotesize C(m,j):=\left\{\begin{array}{ll}
\emptyset&j>m,\\
\begin{array}{|cccccc|}
\hline m&m{-}1&\cdots&j{+}2&j{+}1&j\\ \hline
\end{array}&m\ge j\ge 1,\\
\begin{array}{|cccccc|}
\hline m&m{-}1&\cdots&3&2&{-}1\\ \hline
\end{array}&j={-}1,\\
\begin{array}{|cccccc|}
\hline 
m&m{-}1&\cdots&3&2&{\begin{smallmatrix}1\\ {-}1\end{smallmatrix}}\\ \hline
\end{array}&j={-}2,\\
\begin{array}{|ccccccccccc|}
\hline 
m&m{-}1&\cdots&3&2&{\begin{smallmatrix}1\\ {-}1\end{smallmatrix}}&{-}2&{-}3&\cdots&j{+}2&j{+}1\\ \hline
\end{array}&j\le {-}3.
\end{array}\right.}\]
We simply write $C(m,j)$ as $\begin{array}{|ccccc|}
\hline m&m{-}1&m{-}2&\cdots&j'\\ \hline
\end{array}$, where $j':=j$ if $j\ge-1$ and $j':=j+1$ if $j\le -2$.

Let $S(w)$ be the subarray of the array \eqref{Pell array} given by
\begin{equation} \label{S def}
S(w)={\tiny\begin{array}{|c|}\hline C(\ell,i_{\ell{+}1})\\ \hline C(\ell{+}1,i_{\ell{+}2})\\ \hline C(\ell{+}2,i_{\ell{+}3})\\ \hline\vdots\\ \hline C(n{-}2,i_{n{-}1})\\ \hline C(n{-}1,i_n)\\ \hline\end{array}}
={\tiny\begin{array}{|ccccccccc|}
\hline
\ell&\ell{-}1&\ell{-}2&\cdots&\cdots&\cdots&\cdots&i'_{\ell{+}1}&\\ \hline
\ell{+}1&\ell&\ell{-}1&\cdots&\cdots&\cdots&i'_{\ell{+}2}&&\\ \hline
\ell{+}2&\ell{+}1&\ell&\cdots&\cdots&i'_{\ell{+}3}&&&\\ \hline
\vdots&\vdots&\vdots&\vdots&\vdots&&&&\\ \hline
n{-}2&\cdots&\cdots&i'_{n{-}1}&&&&&\\ \hline
n{-}1&\cdots&i'_{n}&&&&&&\\ \hline
\end{array}.}\end{equation}
Note that $i'_m\le m$ holds for any $m\in\{\ell+1,\ldots,n\}$, and if $i'_{m}=m$ holds, then the row starting at $m-1$ is empty.

\begin{theorem}\label{define X(w) for D1}
Let $w=[i_1,\ldots,i_n]\in W$ be a join-irreducible element of type $\ell\neq\pm1$, and $S(w)$ subarray of the array \eqref{Pell array} described in \eqref{S def}.
\begin{enumerate}[\rm(a)]
\item If $\{1,2\}\not\subset\{|i_{\ell+1}|,\ldots,|i_n|\}$, then $S(w)$ is predecessor-closed.
\item If $\{1,2\}\subset\{|i_{\ell+1}|,\ldots,|i_n|\}$, then $S(w)$ is $(1,0)$ or $(0,1)$-predecessor-closed.
\end{enumerate}
In either case, we have a factor module of $P_\ell$ corresponding to $S(w)$.
\begin{enumerate}[\rm(c)]
\item $J(w)$ is the factor module of $P_\ell$ corresponding to $S(w)$.
\end{enumerate}
\end{theorem}

For example, let $n=6$ and $w=[\underline{3}4\underline{5}126]$ and $w'=[34\underline{5}\underline{2}16]$. Then
\[S(w)={\tiny\begin{array}{|cccccc|}\hline 2&{\begin{smallmatrix}1\\ -1\end{smallmatrix}}&-2&-3&-4&\\ \hline 3&2&1&&&\\ \hline 4&3&2&&&\\ \hline &&&&&\\ \hline\end{array}}\ \mbox{ and }\ 
S(w')={\tiny\begin{array}{|cccccc|}\hline 2&{\begin{smallmatrix}1\\ -1\end{smallmatrix}}&-2&-3&-4&\\ \hline 3&2&{\begin{smallmatrix}1\\ -1\end{smallmatrix}}&&&\\ \hline 4&3&2&1&&\\ \hline &&&&&\\ \hline\end{array}}\]
Thus $S(w)$ is $(0,1)$-predecessor-closed and $S(w')$ is $(1,0)$-predecessor-closed, and
\[J(w)={\tiny\begin{xy}
0;<5pt,0pt>:<0pt,4pt>::
(15,-5)*+{4}="44",
(15,0)*+{3}="45",
(15,5)*+{2}="46",
(20,-5)*+{3}="54",
(20,0)*+{2}="55",
(20,6)*+{1}="56",
(20,4)*+{{-}1}="56-",
(25,-5)*+{2}="64",
(25,-1)*+{1}="65-",
(25,5)*+{{-}2}="66",
(30,5)*+{{-}3}="76",
(35,5)*+{{-}4}="86",
\ar@<0pt>"45";"44",
\ar@<0pt>"46";"45",
\ar@<0pt>"55";"54",
\ar@<0pt>"56-";"55",
\ar@<0pt>"65-";"64",

\ar@<0pt>"44";"54",
\ar@<0pt>"45";"55",
\ar@<0pt>"46";"56",
\ar@<0pt>"46";"56-",
\ar@<0pt>"54";"64",
\ar@<0pt>"55";"65-",
\ar@<0pt>"56";"66",
\ar@<0pt>_{-1}"56-";"66",
\ar@<0pt>"66";"76",
\ar@<0pt>"76";"86",
\end{xy}}\ \mbox{ and }\ 
J(w')={\tiny\begin{xy}
0;<5pt,0pt>:<0pt,4pt>::
(15,-5)*+{4}="44",
(15,0)*+{3}="45",
(15,5)*+{2}="46",
(20,-5)*+{3}="54",
(20,0)*+{2}="55",
(20,6)*+{{-}1}="56",
(20,4)*+{1}="56-",
(25,-5)*+{2}="64",
(25,1)*+{1}="65",
(25,-1)*+{{-}1}="65-",
(25,5)*+{{-}2}="66",
(30,5)*+{{-}3}="76",
(30,-6)*+{1}="74",
(35,5)*+{{-}4}="86",
\ar@<0pt>"45";"44",
\ar@<0pt>"46";"45",
\ar@<0pt>"55";"54",
\ar@<0pt>"56-";"55",
\ar@<0pt>"65-";"64",
\ar@<0pt>^{-1}"66";"65",

\ar@<0pt>"44";"54",
\ar@<0pt>"45";"55",
\ar@<0pt>"46";"56",
\ar@<0pt>"46";"56-",
\ar@<0pt>"54";"64",
\ar@<0pt>"55";"65-",
\ar@<0pt>"55";"65",
\ar@<0pt>^{-1}"56";"66",
\ar@<0pt>"56-";"66",
\ar@<0pt>"64";"74",
\ar@<0pt>"66";"76",
\ar@<0pt>"76";"86",
\end{xy}}.\]

To prove Theorem \ref{define X(w) for D1}, we need some preparation.
Let $w=[i_1,\ldots,i_n]\in W$ be a join-irreducible element of type $\ell$. 
For each $m\in\{\ell+1,\ldots,n\}$, let
\[j_m:=i_m+\#\{m'\mid m\le m'\le n,\ |i_{m'}|\le|i_m|\},\]
Clearly $i_m>0$ implies $j_m=i_m+1>1$, and $i_m<0$ implies $j_m\le0$. 
In the latter case, we have $j_m=-\#\{m'\mid1\le m'\le\ell,\ |i_{m'}|\le |i_m|\}$.
Therefore $j_{\ell+1}\le j_{\ell+2}\le\cdots\le j_n$ holds.

\begin{lemma}\label{reduced expression for type D}
Let $w=[i_1,\ldots,i_n]\in W$ be a join-irreducible element of type $\ell\neq\pm1$.
Then we have a reduced expression $w=x_nx_{n-1}\cdots x_{\ell+2}x_{\ell+1}$, where $x_m$ for $m\in\{\ell+1,\ldots,n\}$ is given by
\begin{equation}\label{xm}
x_m=\left\{\begin{array}{ll}
s_{i_m}s_{i_m+1}\cdots s_{m-1}&\mbox{if $j_m>0$,}\\
s_{\epsilon_m}s_{2}s_{3}\cdots s_{m-1}\ \mbox{ for }\ \epsilon_m:=(-1)^{\#\{ m'\mid m\le m'\le n,\ i_{m'}<0\}}&\mbox{if $j_m=0$,}\\
s_{-1}s_{1}s_{2}s_{3}\cdots s_{m-1}&\mbox{if $j_m=-1$,}\\
s_{-j_m}s_{-j_m-1}\cdots s_3s_2s_{-1}s_1s_2s_3\cdots s_{m-1}&\mbox{if $j_m<-1$.}
\end{array}\right.\end{equation}
\end{lemma}

\begin{proof}
Since $i_1<i_2$ and $-i_1<i_2$, we have $0<i_2<\cdots<i_\ell$.
Let $h_1<\cdots<h_{\ell+1}$ be the reordering of
\begin{itemize}
\item $i_1,i_2,\ldots,i_\ell,i_{\ell+1}$ if $j_{\ell+1}>0$ or ($j_{\ell+1}=0$ and $i_1>0$),
\item $-i_1,i_2,\ldots,i_\ell,-i_{\ell+1}$ if $j_{\ell+1}<0$ or ($j_{\ell+1}=0$ and $i_1<0$),
\end{itemize}
and let $v:=[h_1,\ldots,h_{\ell+1},i_{\ell+2},\ldots,i_n]\in W$. It is easy to check that $v$ is either an identity or join-irreducible of type $\ell+1$. We will show
\begin{equation}\label{from w to v}
wx_{\ell+1}^{-1}=v\ \mbox { and }\ \ell(w)-\ell(x_{\ell+1})=\ell(v).
\end{equation}
Then the assertion follows inductively.

Let $t$ be the unique integer satisfying $h_t=|i_{\ell+1}|$.

Assume $j_{\ell+1}>0$ and so $i_{\ell+1}>0$. Since all positive integers smaller than $i_{\ell+1}$ appear in $|i_1|,\ldots,|i_\ell|$ we have $t=i_{\ell+1}$. Thus \eqref{from w to v} follows from $x_{\ell+1}=s_{t}s_{t+1}\cdots s_{\ell}$.

Assume $j_{\ell+1}=0$ and so $i_{\ell+1}<0$. Since all positive integers smaller than $|i_{\ell+1}|$ appear in $|i_{\ell+2}|,\ldots,|i_n|$, we have $t=1$. If $i_1>0$, then $\epsilon_{\ell+1}=1$ holds, and \eqref{from w to v} follows from $x_{\ell+1}=s_{1}s_{2}s_{3}\cdots s_{\ell}$.
If $i_1<0$, then $\epsilon_{\ell+1}=-1$ holds, and \eqref{from w to v} follows from $x_{\ell+1}=s_{-1}s_{2}s_{3}\cdots s_{\ell}$.

Assume $j_{\ell+1}<0$ and so $i_{\ell+1}<0$. Then $t=1-j_{\ell+1}$ holds, and \eqref{from w to v} follows easily from
$x_{\ell+1}=s_{t-1}s_{t-2}\cdots s_3s_2s_{-1}s_1s_2s_3\cdots s_{\ell}$.
\end{proof}

Now we are ready to prove Theorem \ref{define X(w) for D1}.

\begin{proof}[Proof of Theorem \ref{define X(w) for D1}]
(a) and (b) are easily checked.

(c) Choose scalars $\alpha$ and $\beta$ such that $S(w)$ is $(\alpha,\beta)$-predecessor-closed.
Using the quiver representation of $P_\ell$ given in Figure~\ref{Pell} and the reduced expression of $w$ given in Lemma \ref{reduced expression for type D}, we calculate $I(w)e_\ell$.
Dividing into three cases, we show that the first row of $J(w)$ coincides with that of $S(w)$.

Assume $j_{\ell+1}>0$. By \eqref{xm}, the first row of $I(x_{\ell+1})e_\ell$ is given by
\[\left\{\begin{array}{rl}
\begin{array}{|ccccccccc|}
\hline i_{\ell+1}{-}1&i_{\ell+1}{-}2&\cdots&2&{\begin{smallmatrix}1\\ {-}1\end{smallmatrix}}&{-}2&\cdots&2{-}n&1{-}n\\ \hline
\end{array}&\mbox{if $i_{\ell+1}\ge2$,}\\
\begin{array}{|cccccc|}
\hline {-}1&{-}2&{-}3&\cdots&2{-}n&1{-}n\\ \hline
\end{array}&\mbox{if $i_{\ell+1}=1$.}
\end{array}\right.\]
Since $0<i_{\ell+1}<i_{\ell+2}<\cdots<i_n$ holds, both $s_{i_{\ell+1}-1}$ (when $i_{\ell+1}\ge2$) and $s_{-1}$ (when $i_{\ell+1}=1$) do not appear in $x_nx_{n-1}\cdots x_{\ell+2}$ by \eqref{xm}.
Thus the first row of $I(w)e_\ell$ coincides with that of $I(x_{\ell+1})e_\ell$,
and the first row of $J(w)$ coincides with that of $S(w)$.

Assume $j_{\ell+1}<0$. By \eqref{xm}, the first row of $I(x_{\ell+1})e_\ell$ is given by
\[\begin{array}{|cccccc|}
\hline 
j_{\ell+1}{-}1&j_{\ell+1}{-}2&j_{\ell+1}{-}3&\cdots&2{-}n&1{-}n\\ \hline
\end{array}.\]
Let $m_0:=\max\{m\mid \ell+1\le m\le n,\ |i_m|\le|i_{\ell+1}|\}$, which equals $j_{\ell+1}-i_{\ell+1}+\ell$.
We show that, for any $\ell+1\le m\le m_0$, the first row of $I(x_mx_{m-1}\cdots x_{\ell+1})e_\ell$ is
\[\begin{array}{|ccccc|}
\hline 
j_{\ell+1}{-}m{+}\ell&j_{\ell+1}{-}m{+}\ell{-}1&\cdots&2{-}n&1{-}n\\ \hline
\end{array}.\]
The case $m=\ell+1$ was shown above. Assume that this is the case for $m-1(<m_0)$.
It suffices to show that $s_{-j_{\ell+1}+m-1-\ell}$ appears in $x_{m}$, and that $s_{-j_{\ell+1}+m-\ell}$ does not appear in the left side of $s_{-j_{\ell+1}+m-1-\ell}$ in $x_{m}$. Since $-j_{\ell+1}\le \ell$ holds, we have
\[2\le -j_{\ell+1}+m-1-\ell\le m-1.\]
If $j_{m}\ge0$, then the assertion follows from \eqref{xm} and
\[i_{m}\le i_{m_0}-m_0+m\le -i_{\ell+1}-m_0+m-1=-j_{\ell+1}+m-1-\ell.\]
If $j_{m}<0$, then the assertion follows from \eqref{xm} and $-j_{m}\le -j_{\ell+1}<-j_{\ell+1}+m-1-\ell$.
Inductively, we have shown that the first row of $I(x_{m_0}x_{m_0-1}\cdots x_{\ell+1})e_\ell$ is
\begin{equation}\label{desired 1st}
\begin{array}{|ccccc|}
\hline 
i_{\ell+1}&i_{\ell+1}{-}1&\cdots&2{-}n&1{-}n\\ \hline
\end{array}\end{equation}
since $j_{\ell+1}-m_0+\ell=i_{\ell+1}$.
Since $0<-i_{\ell+1}<i_{m_0+1}<\cdots<i_n$ holds, $s_{-i_{\ell+1}}$ does not appear in $x_nx_{n-1}\cdots x_{m_0+1}$.
Thus the first row of $I(w)e_\ell$ coincides with \eqref{desired 1st}, and the first row of $J(w)$ coincides with that of $S(w)$. 

Assume $j_{\ell+1}=0$. By \eqref{xm}, the first row of $I(x_{\ell+1})e_\ell$ is given by
\[\begin{array}{|cccccc|}
\hline {-}\epsilon_{\ell+1}&{-}2&{-}3&\cdots&2{-}n&1{-}n\\ \hline
\end{array}.\]
By a similar argument as in the case $j_{\ell+1}<0$, one can check that the first row of $J(w)$ coincides with that of $S(w)$.

We have completed to show that the first row of $J(w)$ coincides with that of $S(w)$.
All rows of $I(x_{\ell+1})e_\ell$ except the first one coincide with those of $P_\ell$ since $s_{\ell+1}$ does not appear in $x_{\ell+1}$.
Repeating the same calculation, all rows of $J(w)$ coincide with that of $S(w)$.
\end{proof}

\begin{example}
Let $w=[(-1)^{n-\ell},2,\ldots,\ell,\underline{n},\underline{n-1},\ldots,\underline{\ell+1}]$ be a join-irreducible element of type $\ell$.
Then we have a reduced expression
\[w=x_nx_{n-1}\cdots x_{\ell+2}x_{\ell+1}\ \mbox{ for }\ x_m=s_\ell\cdots s_3s_2s_{-1}s_1s_2s_3\cdots s_{m-1},\]
and the corresponding indecomposable $\tau^-$-rigid $\Pi$-module is $J(w)=P_\ell$.
\end{example}

\begin{example}
Consider type $D_4$. Below we show the expression given in Lemma \ref{reduced expression for type D} by using dots: $w=x_{n}\cdot x_{n-1}\cdot\cdots\cdot x_{\ell+2}\cdot x_{\ell+1}$.

We have the following 7 join-irreducible elements of type $3$.

\noindent
$J(1243)=J(s_3)={\tiny\begin{array}{|c|}\hline 3\\ \hline\end{array}}$,\ \ \ 
$J(1342)=J(s_2s_3)={\tiny\begin{array}{|cc|}\hline 3&2\\ \hline\end{array}}$,

\noindent
$J(2341)=J(s_1s_2s_3)={\tiny\begin{array}{|ccc|}\hline 3&2&1\\ \hline\end{array}}$,\ \ \ 
$J(\underline{2}34\underline{1})=J(s_{-1}s_2s_3)={\tiny\begin{array}{|ccc|}\hline 3&2&-1\\ \hline\end{array}}$,

\noindent
$J(\underline{1}34\underline{2})=J(s_{-1}s_1s_2s_3)={\tiny\begin{array}{|ccc|}\hline 3&2&{\begin{smallmatrix}1\\ -1\end{smallmatrix}}\\ \hline\end{array}}$,\ \ \ 
$J(\underline{1}24\underline{3})=J(s_2s_{-1}s_1s_2s_3)={\tiny\begin{array}{|cccc|}\hline 3&2& {\begin{smallmatrix}1\\ -1\end{smallmatrix}}&-2\\ \hline\end{array}}$,

\noindent
$J(\underline{1}23\underline{4})=J(s_3s_2s_{-1}s_1s_2s_3)={\tiny\begin{array}{|ccccc|}\hline 3&2& {\begin{smallmatrix}1\\ -1\end{smallmatrix}}&-2&-3\\ \hline\end{array}}$.

We have the following 23 join-irreducible elements of type $2$.
They are predecessor-closed unless otherwise specified.

\noindent
$J(1324)=J(s_2)={\tiny\begin{array}{|c|}\hline 2\\ \hline\end{array}}$,\ \ \ 
$J(\underline{2}3\underline{1}4)=J(s_{-1}s_2)={\tiny\begin{array}{|cc|}\hline 2&-1\\ \hline\end{array}}$,

\noindent
$J(2314)=J(s_1s_2)={\tiny\begin{array}{|cc|}\hline 2&1\\ \hline\end{array}}$,\ \ \ 
$J(1423)=J(s_3\cdot s_2)={\tiny\begin{array}{|c|}\hline 2\\ \hline 3\\ \hline\end{array}}$,

\noindent
$J(2413)=J(s_3\cdot s_1s_2)={\tiny\begin{array}{|cc|}\hline 2&1\\ \hline 3&\\ \hline\end{array}}$,\ \ \ 
$J(\underline{2}4\underline{1}3)=J(s_3\cdot s_{-1}s_2)={\tiny\begin{array}{|cc|}\hline 2&-1\\ \hline 3&\\ \hline\end{array}}$,

\noindent
$J(\underline{1}3\underline{2}4)=J(s_{-1}s_1s_2)={\tiny\begin{array}{|cc|}\hline 2&{\begin{smallmatrix}1\\ -1\end{smallmatrix}}\\ \hline\end{array}}$,\ \ \ 
$J(\underline{1}4\underline{2}3)=J(s_3\cdot s_{-1}s_1s_2)={\tiny\begin{array}{|cc|}\hline 2&{\begin{smallmatrix}1\\ -1\end{smallmatrix}}\\ \hline 3&\\ \hline\end{array}}$,

\noindent
$J(\underline{1}2\underline{3}4)=J(s_2s_{-1}s_1s_2)={\tiny\begin{array}{|ccc|}\hline 2& {\begin{smallmatrix}1\\ -1\end{smallmatrix}}&-2\\ \hline\end{array}}$,

\noindent
$J(\underline{3}4\underline{1}2)=J(s_2s_3\cdot s_{-1}s_2)={\tiny\begin{array}{|cc|}\hline 2&-1\\ \hline 3&2\\ \hline\end{array}}$ (this is $(0,1)$-predecessor closed),

\noindent
$J(3412)=J(s_2s_3\cdot s_1s_2)={\tiny\begin{array}{|cc|}\hline 2&1\\ \hline 3&2\\ \hline\end{array}}$ (this is $(1,0)$-predecessor closed),

\noindent
$J(34\underline{2}\underline{1})=J(s_{-1}s_2s_3\cdot s_1s_2)={\tiny\begin{array}{|ccc|}\hline 2&{\begin{smallmatrix}1\\ -1\end{smallmatrix}}&\\ \hline 3&2&-1\\ \hline\end{array}}$ (this is $(1,0)$-predecessor closed),

\noindent
$J(\underline{3}4\underline{2}1)=J(s_1s_2s_3\cdot s_{-1}s_2)={\tiny\begin{array}{|ccc|}\hline 2&{\begin{smallmatrix}1\\ -1\end{smallmatrix}}&\\ \hline 3&2&1\\ \hline\end{array}}$  (this is $(0,1)$-predecessor closed),

\noindent
$J(\underline{1}2\underline{4}3)=J(s_3\cdot s_2s_{-1}s_1s_2)={\tiny\begin{array}{|cccc|}\hline 2&{\begin{smallmatrix}1\\ -1\end{smallmatrix}}&-2&-3\\ \hline 3&&&\\ \hline\end{array}}$,

\noindent
$J(\underline{1}4\underline{3}2)=J(s_2s_3\cdot s_{-1}s_1s_2)={\tiny\begin{array}{|ccc|}\hline 2&{\begin{smallmatrix}1\\ -1\end{smallmatrix}}&-2\\ \hline 3&2&\\ \hline\end{array}}$,

\noindent
$J(24\underline{3}\underline{1})=J(s_{-1}s_2s_3\cdot s_{-1}s_1s_2)={\tiny\begin{array}{|ccc|}\hline 2&{\begin{smallmatrix}1\\ -1\end{smallmatrix}}&-2\\ \hline 3&2&-1\\ \hline\end{array}}$,

\noindent
$J(\underline{1}3\underline{4}2)=J(s_2s_3\cdot s_2s_{-1}s_1s_2)={\tiny\begin{array}{|cccc|}\hline 2&{\begin{smallmatrix}1\\ -1\end{smallmatrix}}&-2&-3\\ \hline 3&2&&\\ \hline\end{array}}$,

\noindent
$J(\underline{2}4\underline{3}1)=J(s_1s_2s_3\cdot s_{-1}s_1s_2)={\tiny\begin{array}{|ccc|}\hline 2&{\begin{smallmatrix}1\\ -1\end{smallmatrix}}&-2\\ \hline 3&2&1\\ \hline\end{array}}$,

\noindent
$J(23\underline{4}\underline{1})=J(s_{-1}s_2s_3\cdot s_2s_{-1}s_1s_2)={\tiny\begin{array}{|cccc|}\hline 2&{\begin{smallmatrix}1\\ -1\end{smallmatrix}}&-2&-3\\ \hline 3&2&-1&\\ \hline\end{array}}$,

\noindent
$J(14\underline{3}\underline{2})=J(s_{-1}s_1s_2s_3\cdot s_{-1}s_1s_2)={\tiny\begin{array}{|ccc|}\hline 2&{\begin{smallmatrix}1\\ -1\end{smallmatrix}}&-2\\ \hline 3&2&{\begin{smallmatrix}1\\ -1\end{smallmatrix}}\\ \hline\end{array}}$,

\noindent
$J(\underline{2}3\underline{4}1)=J(s_1s_2s_3\cdot s_2s_{-1}s_1s_2)={\tiny\begin{array}{|cccc|}\hline 2&{\begin{smallmatrix}1\\ -1\end{smallmatrix}}&-2&-3\\ \hline 3&2&1&\\ \hline\end{array}}$,

\noindent
$J(13\underline{4}\underline{2})=J(s_{-1}s_1s_2s_3\cdot s_2s_{-1}s_1s_2)={\tiny\begin{array}{|cccc|}\hline 2&{\begin{smallmatrix}1\\ -1\end{smallmatrix}}&-2&-3\\ \hline 3&2&{\begin{smallmatrix}1\\ -1\end{smallmatrix}}&\\ \hline\end{array}}$,

\noindent
$J(12\underline{4}\underline{3})=J(s_2s_{-1}s_1s_2s_3\cdot s_2s_{-1}s_1s_2)={\tiny\begin{array}{|cccc|}\hline 2&{\begin{smallmatrix}1\\ -1\end{smallmatrix}}&-2&-3\\ \hline 3&2&{\begin{smallmatrix}1\\ -1\end{smallmatrix}}&-2\\ \hline\end{array}}$.
\end{example}

\begin{landscape}
\begin{figure}
{\tiny\begin{xy}
0;<7.8pt,0pt>:<0pt,11pt>::
(0,0)*+{n{-}1}="11",
(0,5)*+{n{-}2}="12",
(0,10)*+{n{-}3}="12.5",
(0,15)*+{\raisebox{-2pt}[7pt][4pt]{\vdots}}="13",
(0,20)*+{\ell{+}2}="14",
(0,25)*+{\ell{+}1}="15",
(0,30)*+{\ell}="16",
(5,0)*+{n{-}2}="21",
(5,5)*+{n{-}3}="22",
(5,10)*+{n{-}4}="22.5",
(5,15)*+{\raisebox{-2pt}[7pt][4pt]{\vdots}}="23",
(5,20)*+{\ell{+}1}="24",
(5,25)*+{\ell}="25",
(5,30)*+{\ell{-}1}="26",
(10,0)*+{\cdots}="31",
(10,5)*+{\cdots}="32",
(10,10)*+{\cdots}="32.5",
(10,15)*+{\cdots}="33",
(10,20)*+{\cdots}="34",
(10,25)*+{\cdots}="35",
(10,30)*+{\cdots}="36",
(15,0)*+{n{-}\ell{+}1}="41",
(15,5)*+{n{-}\ell}="42",
(15,10)*+{n{-}\ell{-}1}="42.5",
(15,15)*+{\raisebox{-2pt}[7pt][4pt]{\vdots}}="43",
(15,20)*+{4}="44",
(15,25)*+{3}="45",
(15,30)*+{2}="46",
(20,0)*+{n{-}\ell}="51",
(20,5)*+{n{-}\ell{-}1}="52",
(20,10)*+{n{-}\ell{-}2}="52.5",
(20,15)*+{\raisebox{-2pt}[7pt][4pt]{\vdots}}="53",
(20,20)*+{3}="54",
(20,25)*+{2}="55",
(21,31)*+{1}="56",
(19,29)*+{{-}1}="56-",
(25,0)*+{n{-}\ell{-}1}="61",
(25,5)*+{n{-}\ell{-}2}="62",
(25,10)*+{n{-}\ell{-}3}="62.5",
(25,15)*+{\raisebox{-2pt}[7pt][4pt]{\vdots}}="63",
(25,20)*+{2}="64",
(26,26)*+{{-}1}="65",
(24,24)*+{1}="65-",
(25,30)*+{{-}2}="66",
(30,0)*+{n{-}\ell{-}2}="71",
(30,5)*+{n{-}\ell{-}3}="72",
(30,10)*+{n{-}\ell{-}4}="72.5",
(30,15)*+{\raisebox{-2pt}[7pt][4pt]{\vdots}}="73",
(31,21)*+{1}="74",
(29,19)*+{{-}1}="74-",
(30,25)*+{{-}2}="75",
(30,30)*+{{-}3}="76",
(35,0)*+{n{-}\ell{-}3}="81",
(35,5)*+{n{-}\ell{-}4}="82",
(35,10)*+{n{-}\ell{-}5}="82.5",
(36,17)="83",
(35,15)*+{\raisebox{-2pt}[7pt][4pt]{\vdots}}="83middle",
(34,16)="83-",
(35,20)*+{{-}2}="84",
(35,25)*+{{-}3}="85",
(35,30)*+{{-}4}="86",
(40,0)*+{\cdots}="91",
(40,5)*+{\cdots}="92",
(40,10)*+{\cdots}="92.5",
(40,15)*+{\cdots}="93",
(40,20)*+{\cdots}="94",
(40,25)*+{\cdots}="95",
(40,30)*+{\cdots}="96",
(45,0)*+{4}="c1",
(45,5)*+{3}="c2",
(45,10)*+{2}="c2.5",
(46,14)="c3",
(45,15)*+{\raisebox{-2pt}[7pt][4pt]{\vdots}}="c3middle",
(44,13)="c3-",
(45,20)*+{{-}n{+}\ell{+}4}="c4",
(45,25)*+{{-}n{+}\ell{+}3}="c5",
(45,30)*+{{-}n{+}\ell{+}2}="c6",
(50,0)*+{3}="d1",
(50,5)*+{2}="d2",
(51,11)*+{\pm1}="d2.5",
(49,9)*+{\mp1}="d2.5-",
(50,15)*+{\raisebox{-2pt}[7pt][4pt]{\vdots}}="d3",
(50,20)*+{{-}n{+}\ell{+}3}="d4",
(50,25)*+{{-}n{+}\ell{+}2}="d5",
(50,30)*+{{-}n{+}\ell{+}1}="d6",
(55,0)*+{2}="e1",
(56,6)*+{\mp1}="e2",
(54,4)*+{\pm1}="e2-",
(55,10)*+{{-}2}="e2.5",
(55,15)*+{\raisebox{-2pt}[7pt][4pt]{\vdots}}="e3",
(55,20)*+{{-}n{+}\ell{+}2}="e4",
(55,25)*+{{-}n{+}\ell{+}1}="e5",
(55,30)*+{{-}n{+}\ell}="e6",
(61,1)*+{\pm1}="f1",
(59,-1)*+{\mp1}="f1-",
(60,5)*+{{-}2}="f2",
(60,10)*+{{-}3}="f2.5",
(60,15)*+{\raisebox{-2pt}[7pt][4pt]{\vdots}}="f3",
(60,20)*+{{-}n{+}\ell{+}1}="f4",
(60,25)*+{{-}n{+}\ell}="f5",
(60,30)*+{{-}n{+}\ell{-}1}="f6",
(65,0)*+{{-2}}="g1",
(65,5)*+{{-}3}="g2",
(65,10)*+{{-}4}="g2.5",
(65,15)*+{\vdots}="g3",
(65,20)*+{{-}n{+}\ell}="g4",
(65,25)*+{{-}n{+}\ell{-}1}="g5",
(65,30)*+{{-}n{+}\ell{-}2}="g6",
(70,0)*+{\cdots}="h1",
(70,5)*+{\cdots}="h2",
(70,10)*+{\cdots}="h2.5",
(70,15)*+{\cdots}="h3",
(70,20)*+{\cdots}="h4",
(70,25)*+{\cdots}="h5",
(70,30)*+{\cdots}="h6",
(75,0)*+{1{-}\ell}="a1",
(75,5)*+{{-}\ell}="a2",
(75,10)*+{{-}\ell{-}1}="a2.5",
(75,15)*+{\raisebox{-2pt}[7pt][4pt]{\vdots}}="a3",
(75,20)*+{4{-}n}="a4",
(75,25)*+{3{-}n}="a5",
(75,30)*+{2{-}n}="a6",
(80,0)*+{{-}\ell}="b1",
(80,5)*+{{-}\ell{-}1}="b2",
(80,10)*+{{-}\ell{-}2}="b2.5",
(80,15)*+{\raisebox{-2pt}[7pt][4pt]{\vdots}}="b3",
(80,20)*+{3{-}n}="b4",
(80,25)*+{2{-}n}="b5",
(80,30)*+{1{-}n}="b6",
\ar@<0pt>"12";"11",
\ar@<0pt>"12.5";"12",
\ar@<0pt>"13";"12.5",
\ar@<0pt>"14";"13",
\ar@<0pt>"15";"14",
\ar@<0pt>"16";"15",
\ar@<0pt>"22";"21",
\ar@<0pt>"22.5";"22",
\ar@<0pt>"23";"22.5",
\ar@<0pt>"24";"23",
\ar@<0pt>"25";"24",
\ar@<0pt>"26";"25",
\ar@<0pt>"42";"41",
\ar@<0pt>"42.5";"42",
\ar@<0pt>"43";"42.5",
\ar@<0pt>"44";"43",
\ar@<0pt>"45";"44",
\ar@<0pt>"46";"45",
\ar@<0pt>"52";"51",
\ar@<0pt>"52.5";"52",
\ar@<0pt>"53";"52.5",
\ar@<0pt>"54";"53",
\ar@<0pt>"55";"54",
\ar@<0pt>^(0.55)\alpha"56";"55",
\ar@<0pt>_(0.4)\beta"56-";"55",
\ar@<0pt>"62";"61",
\ar@<0pt>"62.5";"62",
\ar@<0pt>"63";"62.5",
\ar@<0pt>"64";"63",
\ar@<0pt>^(0.55)\alpha"65";"64",
\ar@<0pt>_(0.4)\beta"65-";"64",
\ar@<0pt>^\beta"66";"65",
\ar@<0pt>_{-\alpha}"66";"65-",
\ar@<0pt>"72";"71",
\ar@<0pt>"72.5";"72",
\ar@<0pt>"73";"72.5",
\ar@<0pt>^(0.55)\alpha"74";"73",
\ar@<0pt>_(0.4)\beta"74-";"73",
\ar@<0pt>^\beta"75";"74",
\ar@<0pt>_{-\alpha}"75";"74-",
\ar@<0pt>"76";"75",
\ar@<0pt>"82";"81",
\ar@<0pt>"82.5";"82",
\ar@<0pt>"83middle";"82.5",
\ar@<0pt>^(0.6)\beta"84";"83",
\ar@<0pt>_(0.7){-\alpha}"84";"83-",
\ar@<0pt>"85";"84",
\ar@<0pt>"86";"85",
\ar@<0pt>"c2";"c1",
\ar@<0pt>"c2.5";"c2",
\ar@<0pt>^(0.55)\alpha"c3";"c2.5",
\ar@<0pt>_(0.4)\beta"c3-";"c2.5",
\ar@<0pt>"c4";"c3middle",
\ar@<0pt>"c5";"c4",
\ar@<0pt>"c6";"c5",
\ar@<0pt>"d2";"d1",
\ar@<0pt>^(0.55)\alpha"d2.5";"d2",
\ar@<0pt>_(0.4)\beta"d2.5-";"d2",
\ar@<0pt>^\beta"d3";"d2.5",
\ar@<0pt>_{-\alpha}"d3";"d2.5-",
\ar@<0pt>"d4";"d3",
\ar@<0pt>"d5";"d4",
\ar@<0pt>"d6";"d5",
\ar@<0pt>^(0.55)\alpha"e2";"e1",
\ar@<0pt>_(0.4)\beta"e2-";"e1",
\ar@<0pt>^\beta"e2.5";"e2",
\ar@<0pt>_{-\alpha}"e2.5";"e2-",
\ar@<0pt>"e3";"e2.5",
\ar@<0pt>"e4";"e3",
\ar@<0pt>"e5";"e4",
\ar@<0pt>"e6";"e5",
\ar@<0pt>^\beta"f2";"f1",
\ar@<0pt>_{-\alpha}"f2";"f1-",
\ar@<0pt>"f2.5";"f2",
\ar@<0pt>"f3";"f2.5",
\ar@<0pt>"f4";"f3",
\ar@<0pt>"f5";"f4",
\ar@<0pt>"f6";"f5",
\ar@<0pt>"g2";"g1",
\ar@<0pt>"g2.5";"g2",
\ar@<0pt>"g3";"g2.5",
\ar@<0pt>"g4";"g3",
\ar@<0pt>"g5";"g4",
\ar@<0pt>"g6";"g5",
\ar@<0pt>"a2";"a1",
\ar@<0pt>"a2.5";"a2",
\ar@<0pt>"a3";"a2.5",
\ar@<0pt>"a4";"a3",
\ar@<0pt>"a5";"a4",
\ar@<0pt>"a6";"a5",
\ar@<0pt>"b2";"b1",
\ar@<0pt>"b2.5";"b2",
\ar@<0pt>"b3";"b2.5",
\ar@<0pt>"b4";"b3",
\ar@<0pt>"b5";"b4",
\ar@<0pt>"b6";"b5",

\ar@<0pt>"11";"21",
\ar@<0pt>"12";"22",
\ar@<0pt>"12.5";"22.5",
\ar@<0pt>"14";"24",
\ar@<0pt>"15";"25",
\ar@<0pt>"16";"26",
\ar@<0pt>"21";"31",
\ar@<0pt>"22";"32",
\ar@<0pt>"22.5";"32.5",
\ar@<0pt>"24";"34",
\ar@<0pt>"25";"35",
\ar@<0pt>"26";"36",
\ar@<0pt>"31";"41",
\ar@<0pt>"32";"42",
\ar@<0pt>"32.5";"42.5",
\ar@<0pt>"34";"44",
\ar@<0pt>"35";"45",
\ar@<0pt>"36";"46",
\ar@<0pt>"41";"51",
\ar@<0pt>"42";"52",
\ar@<0pt>"42.5";"52.5",
\ar@<0pt>"44";"54",
\ar@<0pt>"45";"55",
\ar@<0pt>"46";"56",
\ar@<0pt>"46";"56-",
\ar@<0pt>"51";"61",
\ar@<0pt>"52";"62",
\ar@<0pt>"52.5";"62.5",
\ar@<0pt>"54";"64",
\ar@<0pt>"55";"65",
\ar@<0pt>"55";"65-",
\ar@<0pt>"56";"66",
\ar@<0pt>_{-1}"56-";"66",
\ar@<0pt>"61";"71",
\ar@<0pt>"62";"72",
\ar@<0pt>"62.5";"72.5",
\ar@<0pt>"64";"74",
\ar@<0pt>"64";"74-",
\ar@<0pt>"65";"75",
\ar@<0pt>_{-1}"65-";"75",
\ar@<0pt>"66";"76",
\ar@<0pt>"71";"81",
\ar@<0pt>"72";"82",
\ar@<0pt>"72.5";"82.5",
\ar@<0pt>"74";"84",
\ar@<0pt>_{-1}"74-";"84",
\ar@<0pt>"75";"85",
\ar@<0pt>"76";"86",
\ar@<0pt>"81";"91",
\ar@<0pt>"82";"92",
\ar@<0pt>"82.5";"92.5",
\ar@<0pt>"84";"94",
\ar@<0pt>"85";"95",
\ar@<0pt>"86";"96",
\ar@<0pt>"91";"c1",
\ar@<0pt>"92";"c2",
\ar@<0pt>"92.5";"c2.5",
\ar@<0pt>"94";"c4",
\ar@<0pt>"95";"c5",
\ar@<0pt>"96";"c6",
\ar@<0pt>"c1";"d1",
\ar@<0pt>"c2";"d2",
\ar@<0pt>"c2.5";"d2.5",
\ar@<0pt>"c2.5";"d2.5-",
\ar@<0pt>"c4";"d4",
\ar@<0pt>"c5";"d5",
\ar@<0pt>"c6";"d6",
\ar@<0pt>"d1";"e1",
\ar@<0pt>"d2";"e2",
\ar@<0pt>"d2";"e2-",
\ar@<0pt>"d2.5";"e2.5",
\ar@<0pt>_{-1}"d2.5-";"e2.5",
\ar@<0pt>"d4";"e4",
\ar@<0pt>"d5";"e5",
\ar@<0pt>"d6";"e6",
\ar@<0pt>"e1";"f1",
\ar@<0pt>"e1";"f1-",
\ar@<0pt>"e2";"f2",
\ar@<0pt>_{-1}"e2-";"f2",
\ar@<0pt>"e2";"f2",
\ar@<0pt>"e2.5";"f2.5",
\ar@<0pt>"e4";"f4",
\ar@<0pt>"e5";"f5",
\ar@<0pt>"e6";"f6",
\ar@<0pt>"f1";"g1",
\ar@<0pt>_{-1}"f1-";"g1",
\ar@<0pt>"f2";"g2",
\ar@<0pt>"f2.5";"g2.5",
\ar@<0pt>"f4";"g4",
\ar@<0pt>"f5";"g5",
\ar@<0pt>"f6";"g6",
\ar@<0pt>"g1";"h1",
\ar@<0pt>"g2";"h2",
\ar@<0pt>"g2.5";"h2.5",
\ar@<0pt>"g4";"h4",
\ar@<0pt>"g5";"h5",
\ar@<0pt>"g6";"h6",
\ar@<0pt>"h1";"a1",
\ar@<0pt>"h2";"a2",
\ar@<0pt>"h2.5";"a2.5",
\ar@<0pt>"h4";"a4",
\ar@<0pt>"h5";"a5",
\ar@<0pt>"h6";"a6",
\ar@<0pt>"a1";"b1",
\ar@<0pt>"a2";"b2",
\ar@<0pt>"a2.5";"b2.5",
\ar@<0pt>"a4";"b4",
\ar@<0pt>"a5";"b5",
\ar@<0pt>"a6";"b6",
\end{xy}}
\caption{The indecomposable projective $\Pi$-module $P_\ell$}\label{Pell}
\end{figure}
\end{landscape}

\end{document}